\documentclass[preprint,dvipsnames]{elsarticle}
\usepackage[utf8]{inputenc}
\usepackage{amsmath, amssymb, amsthm, amsfonts}
\usepackage{bm}
\usepackage{titlesec}
\usepackage{mathtools}
\usepackage{stmaryrd}
\usepackage[dvipsnames]{xcolor}
\usepackage{caption}
\usepackage{tikz, graphicx}
\usepackage{subcaption}
\usepackage{hyperref, cleveref}
\usepackage[titletoc]{appendix}
\usepackage[normalem]{ulem}
\usepackage{geometry}
\usepackage{algorithm}
\usepackage{algpseudocode}
\usepackage{multirow}
\usepackage{MnSymbol}

\makeatletter
\newcommand*\rel@kern[1]{\kern#1\dimexpr\macc@kerna}
\newcommand*\widebar[1]{%
  \begingroup
  \def\mathaccent##1##2{%
    \rel@kern{0.8}%
    \overline{\rel@kern{-0.8}\macc@nucleus\rel@kern{0.2}}%
    \rel@kern{-0.2}%
  }%
  \macc@depth\@ne
  \let\math@bgroup\@empty \let\math@egroup\macc@set@skewchar
  \mathsurround\z@ \frozen@everymath{\mathgroup\macc@group\relax}%
  \macc@set@skewchar\relax
  \let\mathaccentV\macc@nested@a
  \macc@nested@a\relax111{#1}%
  \endgroup
}
\makeatother

\usetikzlibrary{positioning}

\theoremstyle{theorem}
\newtheorem{theorem}{Theorem}

\renewcommand{\hat}[1]{\widehat{#1}} 
 
\newcommand{\pd}[2]{\frac{\partial#1}{\partial#2}}

\newcommand{\LRp}[1]{\left( #1 \right)} 
\newcommand{\LRb}[1]{\left| #1 \right|} 
\newcommand{\LRc}[1]{\left\{ #1 \right\}} 
\newcommand{\jump}[1] {\ensuremath{\left\llbracket#1\right\rrbracket}} 
\newcommand{\avg}[1] {\ensuremath{\LRc{\!\!\LRc{#1}\!\!}}}

\DeclareMathOperator*{\argmin}{arg\,min}
\newcommand{\rmd}{{\rm{d}}}
\newcommand{\rmt}{{\rm{t}}}
\newcommand{\p}{\partial}
\newcommand{\f}{\frac}
\newcommand{\sumid}{\sum_{i=1}^d}
\newcommand{\sumjd}{\sum_{j=1}^d}
\newcommand{\csp}{,\qquad}

\newcommand{\eps}{\epsilon}

\newcommand{\Go}{\Omega}

\newcommand{\Gd}{\Delta}

\newcommand{\R}{\mathbb{R}}
\newcommand{\N}{\mathbb{N}}

\newtheorem{remark}{Remark}

\begin{document}

\title{Entropy stable reduced order modeling of nonlinear conservation laws using discontinuous Galerkin methods} 
\author[1]{Ray Qu}
\author[2]{Akil Narayan}
\author[1]{Jesse Chan}

\affiliation[1]{Department of Computational Applied Mathematics and Operations Research, Rice University, Houston, TX, 77005}
\affiliation[2]{Department of Mathematics, The University of Utah, Salt Lake City, UT 84112}

\begin{abstract}
  Reduced order models (ROMs) are inexpensive surrogate models that reduce costs associated with many-query scenarios. Current methods for constructing entropy stable ROMs for nonlinear conservation laws utilize full order models (FOMs) based on finite volume methods (FVMs). This work describes how to generalize the construction of entropy stable ROMs from FVM FOMs to high order discontinuous Galerkin (DG) FOMs. Significant innovations of our work include the introduction of a new ``test basis'' which significantly improves accuracy for DG FOMs, a dimension-by-dimension hyper-reduction strategy, and a simplification of the boundary hyper-reduction step based on ``Carathéodory pruning''. 
\end{abstract}

\maketitle

\section{Introduction}
Entropy-stable numerical methods have emerged as robust and versatile approaches that not only preserve conservation properties but also guarantee that physically meaningful solutions satisfy an entropy inequality. These methods originated with Tadmor, who introduced a discretely entropy stable finite volume method (FVM) based on entropy-conservative numerical fluxes \cite{Tadmor87}. This method was later extended to higher orders by LeFloch, Mercier, and Rohde using high-order flux functions \cite{LeFloch2002}, and by Fjordholm, Mishra, and Tadmor via ENO reconstructions \cite{Fjordholm12}. More recently, entropy-stable discretizations have been generalized to high-order DG methods by combining entropy-conservative fluxes with summation by parts (SBP) operators \cite{Fisher13, Carpenter14, Gassner16, Chen17, Chan18}. 

\par However, while entropy stable numerical methods significantly improve robustness, their high computational cost still makes them infeasible for many-query settings in e.g., uncertainty quantification. Reduced order modeling aims to construct an efficient surrogate model to reduce the cost of many-query settings. However, constructing a reduced order model (ROM) which retains nonlinear entropy stability is non-trivial. Barone et al.\ combined Galerkin projection and proper orthogonal decomposition (POD) to construct energy stable ROMs for linearized compressible flows \cite{Barone09}. The same authors also discussed stability and convergence with respect to different boundary treatments \cite{Kalashnikova10}. More recently, Chan constructed entropy stable ROMs for nonlinear conservation laws from a finite volume full order model (FOM) \cite{Chan20ROM}. Ofstie and Nadarajah have also independently developed an entropy stable ROM formulation based on high order DG methods in \cite{ofstie2025}, though they do not treat hyper-reduction.


\par This work is interested in extending the work of \cite{Chan20ROM} to high order discontinuous Galerkin (DG) methods. The construction of ROMs based on DG methods has been an active area of research recently. For example, Du and Yano investigate the benefits of adaptive DG meshes after hyper-reduction, demonstrating substantial reductions in computational costs while preserving the accuracy of the models \cite{Du21}. Yu and Hesthaven adopt the discrete empirical interpolation method (DEIM) to hyper-reduce upwinding dissipation, enhancing both the accuracy and stability of the computations for shock-dominated flows \cite{Yu22}. 

In this work, we propose a way to construct entropy stable ROMs using high order DG methods, generalizing previous entropy stable ROM work using FVMs \cite{Chan20ROM}. The main contribution of this paper is the generalization of the ``test basis'' introduced in \cite{Chan20ROM} to construct entropy stable ROM discretizations. The new test basis introduced in this work significantly improves accuracy for ROMs constructed from DG FOMs. Additionally, we introduce a simpler technique for boundary hyper-reduction step based on ``Carathéodory pruning'' \cite{vandenBos20}. 

The paper proceeds as follows: Section~\ref{sec:nonlin_cons_laws} introduces entropy stability theory for nonlinear conservation laws. Section~\ref{sec:1Ddomains} describes the construction of DG FOMs and ROMs, including a more generalized hyper-reduction procedure for DG ROMs. Section~\ref{sec:HDdomains} extends this construction to higher-dimensional domains, and introduces a boundary hyper-reduction step using Carathéodory pruning. Section~\ref{sec:viscosity} discusses the FOM and ROM discretizations of artificial viscosity, and  Section~\ref{sec:NumExp} presents numerical experiments. Section~\ref{sec:conclusion} concludes the paper and discusses some directions for future work.

\section{Nonlinear conservation laws}\label{sec:nonlin_cons_laws}
\par This work focuses multidimensional systems of nonlinear conservation laws; these are represented by a system of PDEs in $d$ dimensions:
\begin{equation}
    \f{\p \bm{u}}{\p \rmt} + \sum_{i=1}^{d} \f{\p \bm{f}^i(\bm{u})}{\p \bm{x}^i} = \mathbf{0}, 
    \label{eq:NCL}
\end{equation}
where $\bm{u}\in\mathbb{R}^n$ are the conservative variables and $\bm{f}^i$ are nonlinear flux functions. We also assume that this system has an entropy-entropy flux pair, with the primary player being the entropy function $S : \R^n \rightarrow \R$, which is a convex function and hence whose Jacobian is invertible \cite{Tadmor03, Chen17}. This motivates introduction of entropy variables $\bm{v}$ through the bijective entropy function Jacobian,
\begin{align}\label{eq:S-def}
  S(\bm{u}) \;\;&\text{convex}, & \bm{v}(\bm{u}) &\coloneqq \f{\partial S}{\partial \bm{u}} (\bm{u}), & \bm{u}(\bm{v}) &\coloneqq \left(\f{\partial S}{\partial \bm{u}}\right)^{-1} (\bm{v}),
\end{align}
where $\left(\f{\partial S}{\partial \bm{u}}\right)^{-1}$ denotes the functional inverse.
Often $S$ also serves as an energy of the system.

We will also make use of the entropy potential function $\bm{\psi}(\bm{u}) \in \R^d$. With $\bm{f} = [\bm{f}^1,\, \ldots,\, \bm{f}^d] \in \R^{n \times d}$, the entropy and entropy potential functions together prescribe the entropy flux function $\bm{v}^T \bm{f} - \bm{\psi}^T(\bm{u})$, which characterizes the flow of entropy.

\par By left-multiplying the entropy variables to \eqref{eq:NCL} and integrating by parts, one can obtain an entropy conservation law of the system on the continuous level satisfying
\begin{equation}
   \int_\Go \f{\p S(\bm{u})}{\p \rmt} + \sumid\int_{\p\Go}\LRp{\bm{v}^T\bm{f}^i(\bm{u})-\psi^i(\bm{u})}\bm{n}^i = 0,
    \label{eq:NCL_ES}
\end{equation}
where $\psi^i(\bm{u})$ and $\bm{n}^i$ are the entropy potential and outward normal vector for the $i$th coordinate, respectively. The second term above is an entropy flux, measuring boundary effects on the total change in entropy inside the domain $\Omega$. While numerical schemes are often derived from the PDE statement \eqref{eq:NCL}, but ensuring that these schemes are also discretized versions of the entropy conservation condition \eqref{eq:NCL_ES} provides entropy (or energy) stability guarantees.

\section{Construction of entropy conservative DG ROMs on 1D domains}
\label{sec:1Ddomains}

We start by providing details for systems over 1D domains $(d=1)$ with periodic boundary conditions and then extend the construction of DG ROMs to 1D domains with weakly-imposed boundary conditions. In particular, we now write $\bm{f} \equiv \bm{f}^1$ for our single $n$-component flux function, and $\psi \equiv \psi^1$ for our single entropy potential function.
\subsection{Notation and summation by parts operators}
\label{sec:DG_notation}
In DG methods, the domain $\Go$ is partitioned into $K$ non-overlapping elements $\{D^k\}_{k=1}^K$. On each element, we use the ansatz that the solution lies within a space of polynomials of degree $p\in \N_0$; it is computationally convenient to equivalently encode the solution through its value at $N_p$ interpolation nodes on each element.
To enforce the PDE, we use collocated interpolation and quadrature nodes for the DG FOM \cite{Gassner16, Chen17, Gassner13, Chen20} with the Gauss-Lobatto quadrature rule, which induces a summation by parts (SBP) structure in the resulting matrix operators \cite{Chan19Colloc}. The SBP property is the discrete analogue of an integration by parts property, and is a useful framework for ensuring energy stability. 

Let $\bm{u}_{i,k} \in \R^n$ represent the conservative variable values at the $i$-th node in the $k$-th cell, and define $\bm{u}_k=[\bm{u}^T_{1,k},...,\bm{u}^T_{N_p,k}]^T$ as a size-$N_p$ vector whose entries are size-$n$ vectors. This ``vector-of-vectors'' notation allows us to write schemes more compactly in what follows. We further denote $\bm{M},\bm{Q}\in\mathbb{R}^{N_p\times N_p}$ as the element-local diagonal mass and nodal differential matrices, such that $\bm{M}^{-1}\bm{Q}$ approximates the first derivative on a reference interval. In this setup, some algebraic operations on size-$N_p$ objects can be applied to our vector-of-vectors setup. For example, $\bm{M} \bm{u}_k$ is another vector of vectors, i.e., a size-$N_p$ vector whose individual entries are size-$n$ vectors, and the $j$th entry of $\bm{M} \bm{u}_k$ is the size-$n$ vector,
\begin{align*}
  \left(\bm{M} \bm{u}_k\right)_j = \sum_{q=1}^{N_p} (M)_{j,q} \bm{u}_{q,k}^T \in \R^n.
\end{align*}
We will make considerable use of such algebraic operations on $\R^n$-valued objects.

We note for later that the operator $\bm{Q}$ satisfies an SBP property if 
\begin{equation}
  \bm{Q} + \bm{Q}^T = \bm{B} \coloneqq 
  \begin{bmatrix}
        -1 & & & &\\
        & 0 & & &\\
        & & \ddots & &\\
        & & & 0 & \\
        & & & & 1
    \end{bmatrix}. 
    \label{eq:DG_SBP}
\end{equation}

\subsection{Entropy conservative fluxes}

\par Because a DG ansatz presumes no continuity of the solution at element interfaces, numerical flux functions are introduced at element interfaces to ensure accurate propagation of information across elements. Entropy conservative numerical schemes are typically constructed based on special numerical fluxes which satisfy an entropy conservation condition. Let $\bm{u}_L$, $\bm{u}_R$ denote left and right solution states (i.e., the values of the solution at an interface from the element on the left and right). Then the two-point flux $\bm{f}_{EC}(\bm{u}_L,\bm{u}_R)$ is entropy-conservative \cite{Tadmor87} if it satisfies the following properties:
\begin{align}\nonumber
  \bm{f}_{EC}(\bm{u}, \bm{u}) &= \bm{f}(\bm{u}) && \text{(consistency)}  \\\label{eq:fluxCondition}
  \bm{f}_{EC}(\bm{u}_L,\bm{u}_R) &= \bm{f}_{EC}(\bm{u}_R,\bm{u}_L) && \text{(symmetry)} \\\nonumber
  (\bm{v}_L-\bm{v}_R)^T\bm{f}_{EC}(\bm{u}_L,\bm{u}_R) &= \bm{\psi}(\bm{u}_L)-\bm{\psi}(\bm{u}_R). && \text{(entropy conservation)}
\end{align}
The choice of $\bm{f}_{EC}(\bm{u}_R,\bm{u}_L)$ is not unique since \eqref{eq:fluxCondition} is underdetermined, and in general can be derived through path integration \cite{Tadmor03}. However, for many systems, explicit formulas for entropy conservative fluxes have been derived, which are generally more efficient to evaluate than path integrals \cite{chandrashekar2013kinetic, ranocha2018comparison, winters2019entropy}.

\subsection{Local formulation}
In 1D, each element $D^k$ is an interval whose size is denoted as $h_k$. Then, $J_k=h_k/2$ is the Jacobian of the mapping from the reference element/interval $[-1,1]$ to the physical element $D^k$. Let $\bm{F}_k$ denote a square size-$N_p$ matrix of conservative flux values,
\begin{align*}
  \bm{F}_k &= \begin{bmatrix}
          \bm{F}_k^1 & & \\
          & \ddots & \\ 
          & & \bm{F}_k^n 
  \end{bmatrix},  & 
    \left(\bm{F}_k\right)_{i,j} &= \bm{f}_{EC}\left(\bm{u}_{i,k}, \bm{u}_{j,k}\right), 
\end{align*}
The individual entries of the size-$N_p$ matrix $\bm{F}_k$ are size-$n$ vectors. 
A local entropy stable DG formulation on $D^k$ is then formed through a collocation approximation of a weak form of \eqref{eq:NCL} that results in a boundary term the couples values at neighboring elements. The scheme takes the form, 
\begin{equation}
\begin{gathered}
J_k\bm{M}\frac{\rmd\bm{u}_k}{\rmd\rmt} + 2\LRp{\bm{Q}\circ\bm{F}_k}\bm{1} + \bm{B}(\bm{f}^*-\bm{f}(\bm{u}_k))= \bm{0}, \\
\bm{f}^* = \begin{bmatrix}
  \bm{f}_{EC}(\bm{u}_{1,k}^+,\bm{u}_{1,k})^T & \bm{0}^T & \hdots & \bm{0}^T  &\bm{f}_{EC}(\bm{u}_{N_p,k}^+,\bm{u}_{N_p,k})^T
\end{bmatrix}^T, \\
  \bm{f}(\bm{u}_k) = \begin{bmatrix} \bm{f}(\bm{u}_{1,k})^T & \ldots & \bm{f}(\bm{u}_{N_p, k})^T \end{bmatrix}^T,
\end{gathered}
    \label{eq:DG_1D_FOM}
\end{equation}
where $\bm{u}_{1,k}^+$, $\bm{u}_{N_p,k}^+$ denote the exterior values of $\bm{u}_{1,k}$, $\bm{u}_{N_p,k}$ on neighboring elements. More specifically, when elements are ordered from left to right in ascending order, for interior element indices $1<k<K$, these values are defined as follows: $$\bm{u}_{1,k}^+ = \bm{u}_{N_p,k-1}\csp \bm{u}_{N_p,k}^+ = \bm{u}_{1,k+1}.$$ The values of $\bm{u}^+$ at boundary nodes ($k = 1, K$) are determined by the imposed boundary conditions, which we will discuss later. In \eqref{eq:DG_1D_FOM}, we have again exercised linear algebra constructions where each element in $\bm{F}$, $\bm{u}_k$, $\bm{f}^\ast$, and $\bm{f}(\bm{u}_k)$ is a size-$n$ vector.

Using the SBP property on $\bm{Q}$ and the consistency property $\bm{f}_{EC}(\bm{u},\bm{u}) = \bm{u}$ on fluxes, we obtain the skew-symmetric form of \eqref{eq:DG_1D_FOM}
\begin{equation}
\label{eq:DG_1D_FOM2}
    J_k\bm{M}\frac{\rmd\bm{u}_k}{\rmd\rmt} + \LRp{\LRp{\bm{Q}-\bm{Q}^T}\circ\bm{F}_k}\bm{1} + \bm{B}\bm{f}^*=\bm{0}.
\end{equation}
This semi-discrete form can be fully discretized through a time integration method.

\subsection{1D periodic domains}
\subsubsection{Full order model}
We review the global DG FOM construction for 1D periodic domains in \cite{Chan22SBP} since ROMs are easier to construct from global solution vectors. We start with the definition of a global solution vector-of-vectors $\bm{u}_\Go=[\bm{u}^T_1,\bm{u}^T_2,...,\bm{u}^T_K]^T \in (\R^n)^{K N_p}$ where we have concatenated element-local solution vectors together.
Then we define a global flux matrix as the $K \times K$ block matrix,
\begin{align}
  \bm{F} &= 
    \begin{bmatrix}
        \bm{F}_{1,1} & \hdots & \bm{F}_{1,K} \\
        \vdots & \ddots & \vdots \\
        \bm{F}_{K,1}  & \hdots & \bm{F}_{K,K} 
    \end{bmatrix},  &
    \LRp{\bm{F}_{k_1,k_2}}_{ij} &= \bm{f}_{EC}\LRp{\bm{u}_{i,k_1},\bm{u}_{j,k_2}},
    \label{eq:F-block}
\end{align}
where the blocks of the matrix $\bm{F}$ capture flux interactions between solution values across elements. For periodic boundary conditions on 1D domains, we set $$\bm{u}_{1,1}^+ = \bm{u}_{N_p,K}\csp \bm{u}_{N_p,K}^+ = \bm{u}_{1,1}.$$
We can now concatenate all local formulations into a single skew-symmetric global formulation involving the global matrix $\bm{Q}_\Go$ \cite{Chan22SBP}
\begin{equation}
    \bm{M}_\Go\f{\rmd\bm{u}_\Go}{\rmd\rmt} + 2\LRp{\bm{Q}_\Go\circ\bm{F}}\bm{1} = \bm{0},
    \label{eq:DG_1D_FOM_PBC}
\end{equation}
where $\bm{M}_\Go$ is a block-diagonal matrix with blocks $J_k\bm{M}$, and $\bm{Q}_\Omega$ is a sparse block-Toeplitz matrix,
\begin{equation}
\label{eq:DG_1D_Qglobal}
    \bm{Q}_\Go = \f{1}{2}\begin{bmatrix}
        \bm{S} & \bm{B}_R & & -\bm{B}_L \\
        -\bm{B}_L & \bm{S} & \bm{B}_R & \\
        & -\bm{B}_L & \ddots & \bm{B}_R \\
        \bm{B}_R & & -\bm{B}_L & \bm{S} 
    \end{bmatrix}
    \csp \bm{S} = \LRp{\bm{Q} - \bm{Q}^T},
\end{equation}
where the matrices $\bm{B}_L, \bm{B}_R \in \R^{n N_p \times n N_p}$ are formed from blocks that are all zeros except for a single entry
\begin{align}\label{eq:BLBRT}
  \bm{B}_L &= \begin{bmatrix}
        & & 1\\
        & \udots & \\
        0 & & 
    \end{bmatrix}, & 
    \bm{B}_R &= \bm{B}_L^T = \begin{bmatrix}
        & & 0\\
        & \udots & \\
        1 & & 
    \end{bmatrix}. 
\end{align}
We can now show some useful properties of the quantities we've introduced.
\begin{theorem}
\label{thm:DG_1D_Qglobal}
Let $\bm{Q}_\Go$ be given by \eqref{eq:DG_1D_Qglobal}. $\bm{Q}_\Go$ is skew-symmetric and has zero row sums, i.e. $$\bm{Q}_\Go + \bm{Q}_\Go^T = \bm{0}\csp\bm{Q}_\Go\bm{1} = \bm{0}.$$
\end{theorem}

\begin{proof}
  By \eqref{eq:DG_1D_Qglobal}, \eqref{eq:BLBRT}, $\bm{Q}_\Go$ is skew-symmetric. To prove the zero row sum property, observe from \eqref{eq:DG_SBP} that $\bm{B}\bm{1} = \left(\bm{B}_R - \bm{B}_L\right)\bm{1}$, and also that $\bm{Q} \bm{1} = \bm{0}$ because of the local differentiation property of $\bm{Q}$. Then each block row sum of $\bm{Q}_\Go$ as defined in \eqref{eq:DG_1D_Qglobal} reduces to 

\[
\left(\bm{Q}-\bm{Q}^T\right)\bm{1} + \bm{B}_R\bm{1} - \bm{B}_L\bm{1} 
  = \left( \bm{Q}-\bm{Q}^T + \bm{B} \right) \bm{1} \stackrel{\eqref{eq:DG_SBP}}{=} 2 \bm{Q} \bm{1} = \bm{0}.
\]

\end{proof}

\begin{theorem}
\label{thm:DG_1D_PBC_EC}
    The solution $\bm{u}_\Go$ of formulation \eqref{eq:DG_1D_FOM_PBC} obeys a semi-discrete entropy conservation condition $$\bm{1}^T\bm{M}_\Go\f{\rmd S(\bm{u}_\Go)}{\rmd\rmt} = 0.$$
\end{theorem}
  We omit the proof, as it is identical to the proof of Theorem 1 in \cite{Chan20ROM}. We note that this proof uses only the fact that $\bm{M}_\Go$ is diagonal and that $\bm{Q}_\Go$ is skew-symmetric and has zero row sums.
We note that a diagonal mass matrix is not strictly necessary, and that one can prove Theorem~\ref{thm:DG_1D_PBC_EC} using a more general ``modal'' entropy stable formulation as well \cite{Chan19Colloc}.

\subsubsection{Reduced order model}\label{sssec:ROM}

We begin by describing the construction of our ROM reduced basis. We employ proper orthogonal decomposition (POD) \cite{Chatterjee00} to construct a subspaced spanned by basis elements $\phi_j$. More specifically, we collect solution snapshots from the FOM at equally spaced times $[t_1,...,t_q]$. With $\bm{x}$ a $K N_p \times d$ matrix containing the $K N_p$ nodal locations, let ${u}_i(\bm{x},t) \in \R^{K N_p}$ denote the $i$th component of the solution at all interpolation nodes and at time $t$. We then form the snapshot matrix:
$$
    \bm{V}_{\text{snap}} = \begin{bmatrix}
     u_1(\bm{x},t_1) & \ldots u_n(\bm{x},t_1) & \cdots & {u}_1(\bm{x},t_q) & \ldots 
     & {u}_n(\bm{x},t_q)
    \end{bmatrix} \in \mathbb{R}^{KN_p\times q n }.
$$
A reduced basis can be obtained by computing a singular value decomposition (SVD) to weighted snapshots $$\sqrt{\bm{M}_\Go}\bm{V}_{\text{snap}} = \bm{U} \bm{\Sigma} \bm{V}^T$$ such that each column of the resulting matrix $\LRp{\sqrt{\bm{M}_\Go}}^{-1}\bm{U}$ represents a POD mode \cite{Liang01}. We now take the first $N$ left singular vectors modes as our reduced basis, i.e., as spatial discretizations of basis elements that we use to approximate each solution component. Let $\bm{V}_{N,1}$ 
denote the generalized Vandermonde matrix, given by the first $N$ columns of $\LRp{\sqrt{\bm{M}_\Go}}^{-1}\bm{U}$, i.e., where column $j$ corresponds to evaluating a conceptual basis function $\phi_j$ at the $K N_p$ spatial locations $\bm{x}$, 
$$\LRp{\bm{V}_{N,1}}_{:,j} =\phi_j(\bm{x}).$$

\par We seek a numerical solution from the span of the POD basis vectors $\bm{u}_\Go\approx\bm{V}_N \bm{u}_N$, where $\bm{u}_N$ are modal coefficients in the reduced basis that we must identify. As before, $\bm{u}_N$ is a vector-of-vectors, i.e., a size-$N$ vector whose individual entries are size-$n$ vectors.
We now impose the Galerkin formulation \cite{Rowley04} to identify our ROM: We plug $\bm{u}_\Go = \bm{V}_N \bm{u}_N$ into the FOM \eqref{eq:DG_1D_FOM_PBC} and force the residual to be orthogonal to $\mathrm{range}(\bm{V}_N)$
yielding the following reduced system
\begin{equation}
    \bm{V}_N^T\bm{M}_\Go\bm{V}_N\f{\rmd\bm{u}_N}{\rmd\rmt}+2\bm{V}_N^T\LRp{\bm{Q}_\Go\circ\bm{F}}\bm{1} = \bm{0},
    \label{eq:FV_ROM}
\end{equation}
  where $\bm{F}$ is again in $K \times K$ block form as in \eqref{eq:F-block}, but the blocks are now given by 
\begin{align*}
  \LRp{\bm{F}_{k_1,k_2}}_{ij} &= \bm{f}_{EC}\LRp{\left(\bm{V}_N \bm{u}_N\right)_{i,k_1},\left(\bm{V}_N \bm{u}_N\right)_{j,k_2}},
\end{align*}

Unfortunately, as observed in \cite{Chan20ROM}, the formulation \eqref{eq:FV_ROM} is no longer entropy conservative. To remedy this, we introduce the entropy-projection $\bm{u}\LRp{\bm{V}_N\bm{V}_N^\dagger\bm{v}(\bm{V}_N\bm{u}_N)}$, where $\bm{u} \mapsto \bm{v}(\bm{u})$ and $\bm{v} \mapsto \bm{u}(\bm{v})$ correspond to the $\R^n \mapsto \R^n$ invertible entropy variables transformation in \eqref{eq:S-def}. I.e., we evaluate the conservative variables in terms of the projection of the entropy variables onto $\mathrm{range}(\bm{V}_N)$, and these transformations are pointwise in space, e.g., $\bm{v}(\bm{V}_N \bm{u}_N)$ means that the map $\bm{v}$ is applied componentwise to each element in the vector-of-vectors $\bm{V}_N \bm{u}_N$. \footnote{We note the entropy projection does not enforce positivity for the conservative variables, and that to the best of our knowledge, this is an open issue for schemes which utilize the entropy projection \cite{KLEIN2025, Parsani16, Chan18, PETHRICK2025}.} Combining the entropy projection with Galerkin projection yields the following formulation:
\begin{equation}
\begin{split}
    & \bm{M}_N\f{\rmd\bm{u}_N}{\rmd\rmt} + 2\bm{V}_N^T\LRp{\bm{Q}_\Go\circ\bm{F}}\bm{1} = \bm{0},\\
    &  \bm{M}_N = \bm{V}_N^T\bm{M}_\Go\bm{V}_N\csp \Tilde{\bm{u}} = \bm{u}\LRp{\bm{V}_N\bm{V}_N^\dagger\bm{v}(\bm{V}_N\bm{u}_N)}, \\
    &\textrm{$\bm{F}$ and $\bm{F}_{k_1, k_2}$ as in \eqref{eq:F-block}, but } \LRp{\bm{F}_{k_1,k_2}}_{ij} = \bm{f}_{EC}\LRp{\Tilde{\bm{u}}_{i,k_1},\Tilde{\bm{u}}_{j,k_2}},
\end{split}
    \label{eq:DG_ROM_ES}
\end{equation}
This new formulation results in a restoration of entropy conservation.
\begin{theorem}
\label{thm:DG_1D_ROM_ES}
    The solution $\bm{u}_N$ in \eqref{eq:DG_ROM_ES} satisfies semi-discrete entropy conservation $$\bm{1}^T\bm{M}_\Go\f{\rmd S(\bm{V}_N\bm{u}_N)}{\rmd\rmt} = 0.$$
    
\end{theorem}
We again omit the proof as it is identical to proof of Theorem 2 in \cite{Chan20ROM}.

\par While the use of the entropy projection preserves entropy conservation, it is equally important to ensure that the accuracy of the solution is maintained after the projection step. To compute the entropy projection, we project the evaluation of the entropy variables onto our reduced basis. Because the entropy projected conservative variables are then evaluated using this projection of the entropy variables, it is important to ensure that the reduced basis accurately captures the solution when represented in both the conservative variables and entropy variables such that both $\tilde{\bm{u}}=\bm{u}\left(\bm{V}_N\bm{V}_N^\dagger\bm{v}(\bm{V}_N\bm{u}_N)\right)$ and $\bm{V}_N\bm{u}_N$ yield accurate approximations of the solution. In practice, we augment the snapshot matrix $\bm{V}_\text{snap}$ with snapshots of the entropy variables:

\[
\begin{bmatrix}
    
     v_1(\bm{u}(\bm{x},t_1)) & \ldots v_n(\bm{u}(\bm{x},t_1)) & \cdots & {v}_1(\bm{u}(\bm{x},t_q)) & \ldots 
     & {v}_n(\bm{u}(\bm{x},t_q))
    \end{bmatrix} \in \mathbb{R}^{KN_p\times q n }.
\]

\subsection{Hyper-reduction}

Additional hyper-reduction steps are necessary to reduce the computational cost associated with the ROM \eqref{eq:DG_ROM_ES}. In particular, we seek a way to reduce the cost of evaluating the flux term $\bm{V}_N^T\LRp{\bm{Q}\circ\bm{F}}\bm{1}$. The hyper-reduction applied to the volume term is similar to the one used in \cite{Chan20ROM}. Here, we provide a summary of the main steps while emphasizing certain differences due to presence of non-uniform quadrature weights in high order DG methods. 

\par  We employ a sampling and weighting strategy \cite{Farhat15,Chapman16} where we determine $\mathcal{I}$, a subset of node indices, along with corresponding positive weights $\bm{w}$, to form a reduced quadrature rule that also allows us to retain entropy conservation. We will use this hyper-reduced quadrature to construct a smaller hyper-reduced matrix $\widebar{\bm{Q}}$ and approximate $$\bm{V}_N^T\LRp{\bm{Q}\circ\bm{F}} \approx \widebar{\bm{V}}_N^T\bm{W}\LRp{\widebar{\bm{Q}}\circ\bm{F}}\csp \widebar{\bm{V}}_N = \bm{V}_N\LRp{\mathcal{I},:}\csp \bm{W} = \text{diag}(\bm{w}),$$
where the notation $\bm{V}_N\LRp{\mathcal{I},:}$ refers to the matrix consisting of rows $i$ of $\bm{V}_N$ for all $i \in \mathcal{I}$.

\par We utilize the greedy algorithm for computing empirical cubature points and weights \cite{Hernandez16}. Since we assume a fixed reduced basis in time, we compute a single set of empirical cubature points and use them throughout the simulation process. Let $\bm{V}_{\text{target}}$ represent the target matrix, whose columns span the desired function space to be integrated, and $\bm{w}_{\text{target}}$ be the initial reference weights. The algorithm produces an index set $\mathcal{I}$ and corresponding new weights $\bm{w}$ such that: $$\bm{V}_{\text{target}}^T\bm{w}_{\text{target}} \approx \bm{V}_{\text{target}}(\mathcal{I},:)^T\bm{w},$$ 

\par For DG, we set $\bm{w}_{\text{target}}=\bm{V}_{\text{target}}^T \bm{w}_\Go$, where $\bm{w}_\Go$ is a global vector of DG quadrature weights on each physical element. Then, we set $\bm{V}_{\text{target}}$ to be the span of products of $\phi_i(\bm{x})\phi_j(\bm{x})$ such that $$\mathcal{R}(\bm{V}_{\text{target}}) = \text{span}\LRc{\bm{V}_N(:,i)\circ\bm{V}_N(:,j),\quad i,j=1,...,N},$$ 
where $\bm{V}_(:,i)$ refers to the $i$th column of $\bm{V}_N$. This choice of target space ensures that the mass matrix $\bm{M}_N=\bm{V}_N^T\bm{M}\bm{V}_N$ is accurately approximated. We then use the SVD to remove linearly dependent vectors in $\bm{V}_\text{target}$ to produce a full rank matrix. For convenience, we use $*$ to denote the operations used to construct $\bm{V}_{\text{target}}$ (e.g., forming products of columns and removing linearly dependent vectors)
$$\bm{V}_{\text{target}} = \bm{V}_N * \bm{V}_N,$$ to which we refer as ``targeting $\bm{V}_N * \bm{V}_N$".

 \par We are able to extend the proof of entropy conservation in \autoref{thm:DG_1D_ROM_ES} as long as the hyper-reduced matrix $\widebar{\bm{Q}}$ maintains the two properties (skew-symmetry and zero row sums). To ensure this, we use a two-step hyper-reduction. We first compute a compressed test matrix $$\hat{\bm{Q}}_t = \bm{V}_t^T\bm{Q}\bm{V}_t$$ from some test basis $\bm{V}_t$, and then construct the hyper-reduced nodal differentiation matrix $$\widebar{\bm{Q}} = \bm{P}_t^T\hat{\bm{Q}}_t\bm{P}_t$$ using the orthogonal projection $\bm{P}_t = \bm{M}_t^{-1}\widebar{\bm{V}}_t^T\bm{W}$ onto the test basis.

\subsubsection{Choice of test basis}

 In previous work \cite{Chan20ROM}, the test basis was chosen such that 
\begin{equation}
    \mathcal{R}\LRp{\bm{1}},\: \mathcal{R}\LRp{\bm{V}_N},\: \mathcal{R}\LRp{\bm{Q}_\Go\bm{V}_N} \subset \mathcal{R}\LRp{\bm{V}_t}.
    \label{eq:DG_testbasis1}
\end{equation} 
For finite volume FOMs, enriching the test basis to include ${\bm{Q}_\Go\bm{V}_N}$ ensures that, in addition to  maintaining both skew-symmetry and zero row sums,
$\widebar{\bm{Q}}$ accurately samples the derivative matrix $\bm{Q}_\Go\bm{V}_N$. Initial numerical experiments, however, revealed this choice of enrichment to be inaccurate for DG FOMs.

For DG FOMs, we augment the test basis with $\bm{M}_\Go^{-1}\bm{Q}_\Go^T\bm{V}_N$ instead of $
\bm{Q}_{\Go}\bm{V}_N$, so that
\begin{equation}
   \mathcal{R}\LRp{\bm{1}},\: \mathcal{R}\LRp{\bm{V}_N},\:  \mathcal{R}\LRp{\bm{M}_\Go^{-1}\bm{Q}_\Go^T\bm{V}_N}\subset \mathcal{R}\LRp{\bm{V}_t}.
    \label{eq:DG_testbasis2}
\end{equation}
We note that, when $\bm{M}_\Go$ is not a scalar multiple of the identity, the initial choice of test basis \eqref{eq:DG_testbasis1} no longer accurately samples the action of $\bm{Q}_\Go$. Introducing \eqref{eq:DG_testbasis2} fixes this. Note that for finite volume FOMs, $\bm{M}_\Go = h\bm{I}$, and this enrichment reduces \eqref{eq:DG_testbasis2} to the original test basis \eqref{eq:DG_testbasis1} used in \cite{Chan20ROM}.

\par This choice of test basis enrichment ensures that, at least under ``ideal'' hyper-reduction (when the hyper-reduced quadrature is taken to be the entire set of FOM nodes), the resulting hyper-reduced matrix is accurate in the following sense:
\begin{theorem}
\label{thm: DG_1D_HR_testbasis}
Under ``ideal'' hyper-reduction (using all FOM nodes for the hyper-reduced quadrature), the test basis choice \eqref{eq:DG_testbasis2} ensures the error between $\bm{Q}$ and the hyper-reduced operator $\widebar{\bm{Q}}$ is orthogonal to the basis matrix $\bm{V}_N$, i.e.
$$\bm{V}_N^T\LRp{\bm{Q} - \widebar{\bm{Q}}} = \bm{V}_N^T\LRp{\bm{Q} - \bm{P}_t^T\bm{V}_t^T\bm{Q}\bm{V}_t\bm{P}_t} = \bm{0}.$$
\end{theorem}
\begin{proof}
    Since $\mathcal{R}\LRp{\bm{V}_N} \subset \mathcal{R}\LRp{\bm{V}_t}$, $\bm{V}_t\bm{P}_t$ acts as an identity projector for $\bm{V}_N$ such that
    $$\bm{V}_t\bm{P}_t\bm{V}_N = \bm{V}_N.$$
    Therefore, 
    \begin{equation*}
        \begin{split}
            & \bm{V}_N^T\LRp{\bm{Q} - \bm{P}_t^T\bm{V}_t^T\bm{Q}\bm{V}_t\bm{P}_t} = \bm{V}_N^T\bm{Q} - (\bm{V}_t\bm{P}_t\bm{V}_N)^T\bm{Q}\bm{V}_t\bm{P}_t  \\
            & = \bm{V}_N^T\bm{Q}(\bm{I}-\bm{V}_t\bm{P}_t)  = \LRp{\bm{M}_\Go^{-1}\bm{Q}^T\bm{V}_N}^T\bm{M}_\Go(\bm{I}-\bm{V}_t\bm{P}_t) = \bm{0},
        \end{split}
    \end{equation*}
    which follows since  $\mathcal{R}\LRp{\bm{M}_\Go^{-1}\bm{Q}_\Go^T\bm{V}_N}\subset \mathcal{R}\LRp{\bm{V}_t}$  and $\bm{V}_t\bm{P}_t$ is an orthogonal projection onto $\mathcal{R}\LRp{\bm{V}_t}$ with respect to the inner product induced by $\bm{M}_\Go$.
\end{proof}

\par To illustrate the difference in the accuracy between using the original test basis enrichment $\bm{Q}_\Go\bm{V}_N$ and new test basis enrichment $\bm{M}_\Go^{-1}\bm{Q}_\Go\bm{V}_N$, we present an example for the 1D viscous Burgers' equation. We utilize 256 elements with interpolation polynomial degree 3 over 1D periodic domain $[-1,1]$. We use initial condition $u_0=0.5-\sin(\pi x),$ set the viscosity coefficient to $\epsilon=1\times10^{-2}$, and simulate until time $T=1.0$. The setting of this numerical experiment (e.g., physical parameters) is identical to the one in Sec \ref{sec:burgers}, and details of the numerical setting (e.g., time integration scheme) can be found in Sec \ref{sec:preliminary}. Figure~\ref{fig:testbasis} shows the FOM and ROM solutions at the final time, where the blue line represents the FOM solution, the dashed orange line represents the ROM solution (without hyper-reduction), and the green line represents the hyper-reduced ROM solution. The ROM without hyper-reduction closely matches the FOM solution; however, using the original test basis results in large errors for the hyper-reduced DG ROM. Adopting the test basis significantly improves accuracy of the hyper-reduced DG ROM.

\begin{figure}
\centering
\noindent
\subfloat[(a) $N=20$, FVM test basis]{\includegraphics[width=.49\textwidth]{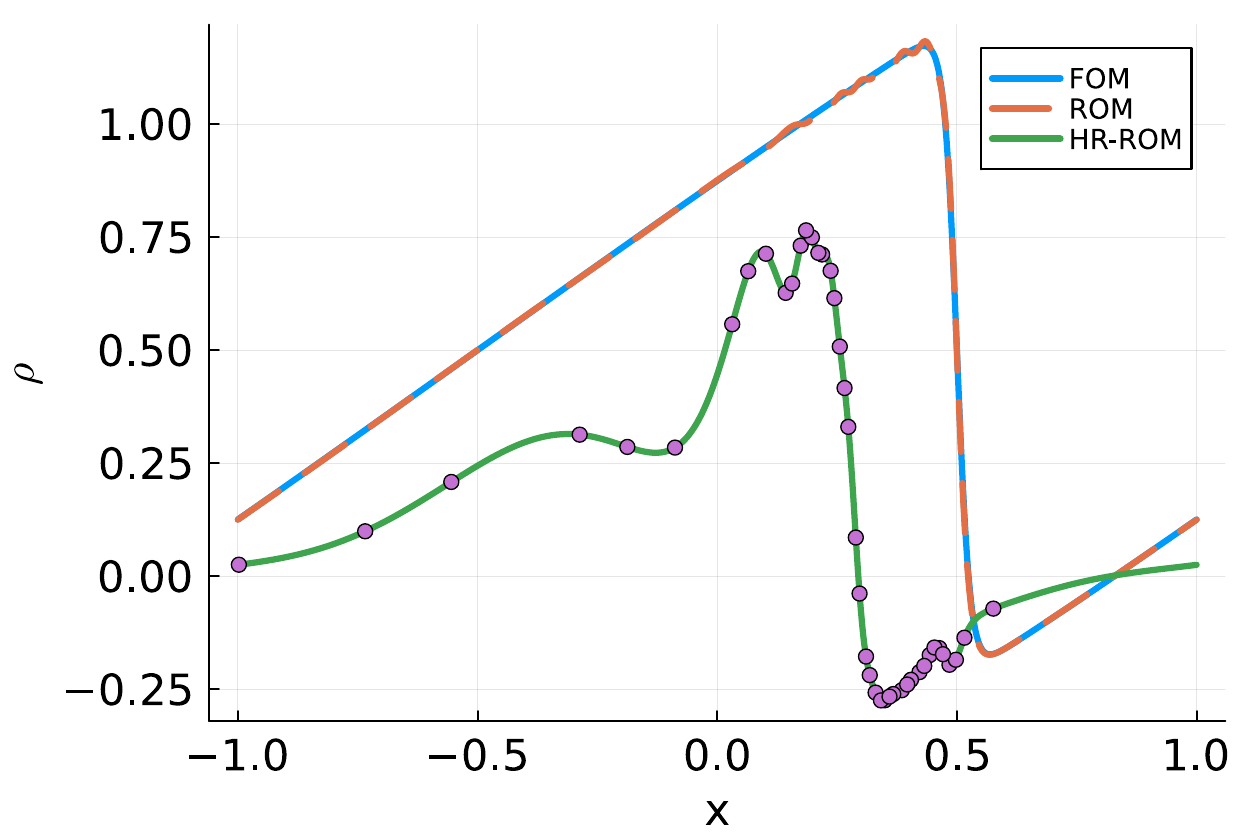}}
\subfloat[(b) $N=20$, new test basis]{\includegraphics[width=.49\textwidth]{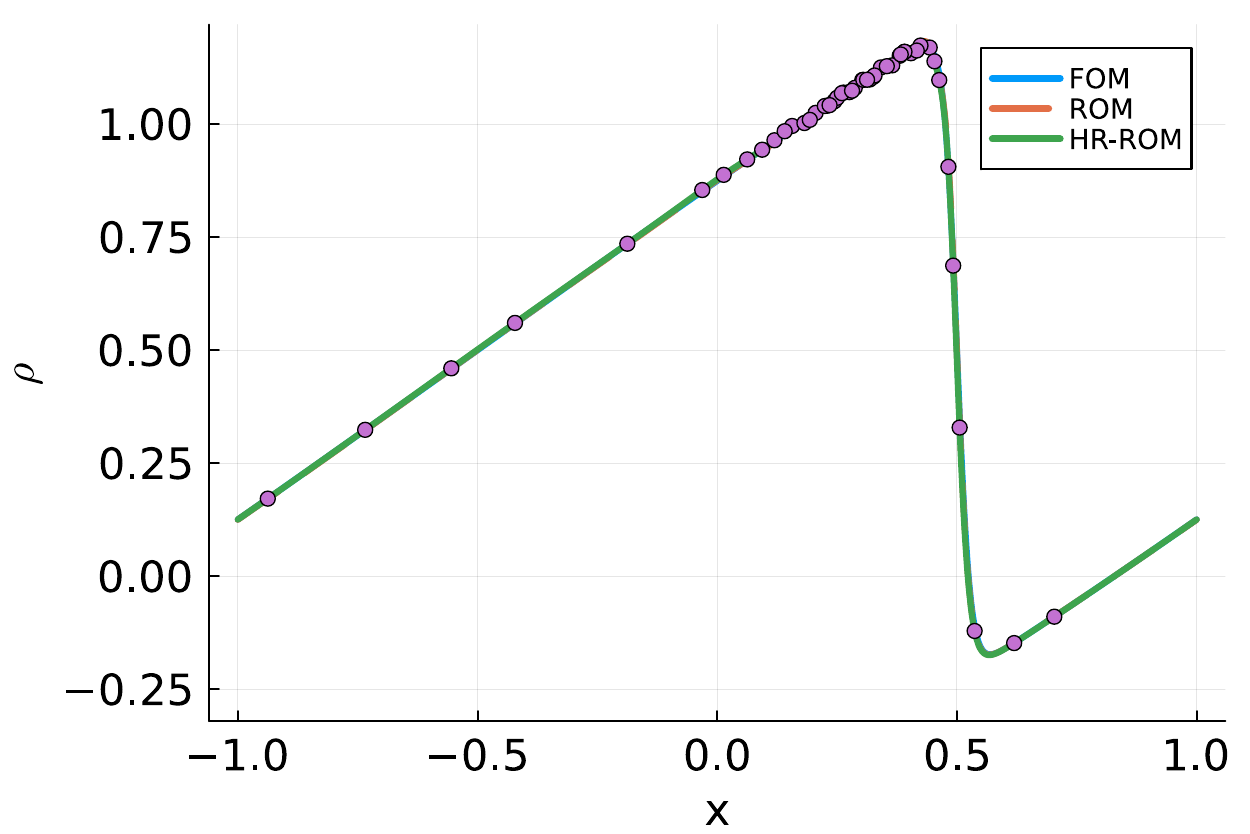}}
\caption{Density plots of FOM, ROM and hyper-reduced ROM solutions with dotted hyper-reduced nodes. Left: original test basis \eqref{eq:DG_testbasis1}; Right: new test basis  \eqref{eq:DG_testbasis2}.}
\label{fig:testbasis}
\end{figure}


\begin{remark}
    In \cite{Chan20ROM}, it is recommended to only target $\bm{V}_N * \bm{V}_N$, based on the observation that targeting $\bm{V}_N * \bm{V}_t$ or $\bm{V}_t * \bm{V}_t$ requires a significantly larger number of nodes without substantially enhancing accuracy. However, as detailed in \cite{Chan20ROM}, targeting only $\bm{V}_N * \bm{V}_N$ may result in a singular test mass matrix. This can be addressed by adding ``stabilizing" hyper-reduced quadrature nodes, and is done for all experiments in this paper using the approach specified in Section 4.3 of \cite{Chan20ROM}. 
\end{remark}

\subsection{Weakly-imposed boundary conditions}

In this section, we extend our model to accommodate weakly-imposed boundary conditions. In DG methods, the boundary condition is usually applied by specifying the relevant exterior boundary states $\bm{u}_{i,k}^+$ during the flux evaluation. On a 1D domain, this only entails specifying the exterior values for the two boundary nodes, namely $\bm{u}_{1,1}^+$ and $\bm{u}_{N_p,K}^+$. The resulting global formulation is derived from \eqref{eq:DG_1D_FOM2}
\begin{equation}
    \bm{M}_\Go\f{\rmd\bm{u}_\Go}{\rmd\rmt} + 2\LRp{\bm{Q}_\Go\circ\bm{F}}\bm{1} + \bm{B}_\Go\bm{f}_\Go^* = \bm{0},
    \label{eq:DG_1D_FOM_weakBC}
\end{equation}
where $\bm{M}_\Go$ remains unchanged from \eqref{eq:DG_1D_FOM_PBC}, but $\bm{Q}_\Go$, $\bm{B}_\Go$, and $\bm{f}_\Go^*$ account for boundary terms, and are now 
\begin{gather*}
      \bm{Q}_\Go = \f{1}{2}\begin{bmatrix}
        \bm{S} & \bm{B}_R & &  \\
        -\bm{B}_L & \bm{S} & \bm{B}_R & \\
        & -\bm{B}_L & \ddots & \bm{B}_R \\
         & & -\bm{B}_L & \bm{S} 
    \end{bmatrix} +\f{1}{2}\bm{B}_\Go, \\ 
    \bm{B}_\Go = \begin{bmatrix}
        -1 & & & &\\
        & 0 & & &\\
        & & \ddots & &\\
        & & & 0 & \\
        & & & & 1
    \end{bmatrix}\csp \bm{f}_\Go^* = \begin{bmatrix}
        \bm{f}_{EC}(\bm{u}_{1,1},\bm{u}_{1,1}^+)  \\
        0 \\
        \vdots \\
        0 \\
         \bm{f}_{EC}(\bm{u}_{N_p,K},\bm{u}_{N_p,K}^+)
    \end{bmatrix}.
\end{gather*}
Here, $\bm{B}_\Go$ and $\bm{f}_\Go^*$ are the global version of $\bm{B}$ and $\bm{f}^*$ to match the global matrix size, and $\bm{Q}_\Go$ obeys a global SBP property
\begin{equation}
    \bm{Q}_\Go + \bm{Q}_\Go^T = \bm{B}_\Go.
    \label{eq:DG_1D_Qglobal_SBP_weakBC}
\end{equation} Following \cite{Chen17}, if we employ the entropy-conservative flux at the boundary, this model formulation retains entropy conservation (similar to the periodic boundary case) under appropriate boundary conditions (for example, under ``mirror state'' imposition of reflective wall boundary conditions for the compressible Euler equations \cite{Chen17}). 

\par Extending \eqref{eq:DG_1D_FOM_weakBC} to the ROM setting requires extra steps compared with the periodic case. First, the hyper-reduced differentiation matrix $\widebar{\bm{Q}}$ now obeys a ``generalized" SBP property \cite{Fernandez14SBP}
\begin{equation}
\label{eq:gen_SBP}
    \widebar{\bm{Q}} + \widebar{\bm{Q}}^T = \bm{E}^T\bm{B}_b\bm{E},
\end{equation}
where $\widebar{\bm{Q}} = \bm{P}_t^T\bm{V}_t^T\bm{Q}_\Go\bm{V}_t\bm{P}_t$ and the matrices $\bm{B}_b$ and $\bm{E}$ encode scalings by outward unit normals and interpolation at boundary nodes, respectively. This complicates the imposition of weak boundary conditions and inter-element fluxes for high order DG methods \cite{Chan18}. To address this, we introduce a ``hybridized'' SBP operator which simplifies the imposition of weak boundary conditions for ROMs.


For 1D ROMs, the boundary normal and boundary interpolation matrices $\bm{B}_b$ and $\bm{E}$ are given by 
\begin{equation}
    \bm{B}_b = \begin{bmatrix}
        -1 & \\ & 1
    \end{bmatrix} \csp
    \bm{E} = \bm{V}_{bt}\bm{P}_t \csp \bm{V}_{bt} = \begin{bmatrix}
        \bm{V}_t(1,:) \\ \bm{V}_t(KN_p,:)
    \end{bmatrix},
    \label{eq:DG_1D_generalSBP}
\end{equation}
where $\bm{V}_{bt}$ interpolates test basis functions to boundary points (e.g.\ the leftmost and rightmost nodes on a 1D domain). These operators can be used to construct hybridized SBP operators \cite{Chan18}, which enable the imposition of weak boundary conditions for entropy stable methods built on generalized SBP operators:
\begin{equation}
    \widebar{\bm{Q}}_h = \f{1}{2}\begin{bmatrix}
        \widebar{\bm{Q}} -\widebar{\bm{Q}}^T & \bm{E}^T\bm{B}_b \\
        -\bm{B}_b\bm{E} & \bm{B}_b
    \end{bmatrix}.
    \label{eq:DG_1D_Qh}
\end{equation}
By construction, $\widebar{\bm{Q}}_h$ satisfies a block SBP property
\begin{equation}
\label{eq:block_SBP}
    \widebar{\bm{Q}}_h + \widebar{\bm{Q}}_h^T = \begin{bmatrix}
        \bm{0} & \\ & \bm{B}_b
    \end{bmatrix} := \bm{B}_h.
\end{equation}
Moreover, one can show that $\widebar{\bm{Q}}_h$ also has zero row sums \cite{Chan20ROM}.

\par Now, define $\bm{V}_b$ as the matrix evaluating the reduced basis at boundary nodes, $\widebar{\bm{V}}_h$ as the matrix evaluating the reduced basis at both boundary and hyper-reduced volume nodes, and $\bm{f}_b^*$ as the vector of boundary flux evaluations:
\begin{equation*}
    \bm{V}_b = \begin{bmatrix}
            \bm{V}_N(1,:) \\ \bm{V}_N(KN_p,:)
    \end{bmatrix} \csp
    \widebar{\bm{V}}_h =
    \begin{bmatrix}
        \widebar{\bm{V}}_N \\ \bm{V}_b
    \end{bmatrix}\csp \bm{f}_b^* = \begin{bmatrix}
        \bm{f}_{EC}(\bm{u}_{1,1},\bm{u}_{1,1}^+)  \\
         \bm{f}_{EC}(\bm{u}_{N_p,K},\bm{u}_{N_p,K}^+)
    \end{bmatrix}
\end{equation*}
Then, we can construct the following hyper-reduced entropy conservative ROM
\begin{equation}
\begin{split} 
   & \widebar{\bm{M}}_N\f{\rmd\bm{u}_N}{\rmd\rmt} + 2\widebar{\bm{V}}_h^T\LRp{\widebar{\bm{Q}}_h\circ\bm{F}}\bm{1} + \bm{V}_b^T\bm{B}_b\LRp{\bm{f}_b^*-\bm{f}(\Tilde{\bm{u}}_b)} =\bm{0},\\
   & \widebar{\bm{V}}_N = \bm{V}_N(\mathcal{I},:)\csp \widebar{\bm{M}}_N = \widebar{\bm{V}}_N^T\bm{W}\widebar{\bm{V}}_N,\csp \bm{P}_N = \widebar{\bm{M}}_N^{-1}\widebar{\bm{V}}_N^T\bm{W}, \\ & \bm{v}_N = \bm{P}_N\bm{v}\LRp{\widebar{\bm{V}}_N\bm{u}_N}\csp \Tilde{\bm{v}}=\widebar{\bm{V}}_h\bm{v}_N\csp \Tilde{\bm{v}}_b=\bm{V}_b\bm{v}_N, \\ &\Tilde{\bm{u}}=\bm{u}(\Tilde{\bm{v}})\csp \Tilde{\bm{u}}_b=\bm{u}(\Tilde{\bm{v}}_b)\csp \bm{F}_{ij} =\bm{f}_{EC}(\Tilde{\bm{u}}_i,\Tilde{\bm{u}}_j).
\end{split}
    \label{eq:DG_ROM_HR_weakBC1}
\end{equation}
Using the block SBP property of $\widebar{\bm{Q}}_h$ \eqref{eq:block_SBP}, the formulation \eqref{eq:DG_ROM_HR_weakBC1} is equivalent to
\begin{equation}
   \widebar{\bm{M}}_N\f{\rmd\bm{u}_N}{\rmd\rmt} + \widebar{\bm{V}}_h^T\LRp{\LRp{\widebar{\bm{Q}}_h-\widebar{\bm{Q}}_h^T}\circ\bm{F}}\bm{1} + \bm{V}_b^T\bm{B}_b\bm{f}_b^* =\bm{0},
    \label{eq:DG_ROM_HR_weakBC2}
\end{equation}
and we have the following statement of conservation of entropy:
\begin{theorem}
\label{thm:weakBC_stability}
    Formulations \eqref{eq:DG_ROM_HR_weakBC1} and \eqref{eq:DG_ROM_HR_weakBC2} admit a semi-discrete conservation of entropy such that $$\bm{1}^T\bm{W}\f{\rmd S(\bm{V}_N\bm{u}_N)}{\rmd\rmt} - \bm{1}^T\bm{B}_b\LRp{\bm{\psi}^T(\Tilde{\bm{u}}_b)-\Tilde{\bm{v}}_b^T\bm{f}_b^*} = 0.$$
\end{theorem}

\begin{proof}
    The proof is identical to the weakly-imposed boundary condition hyper-reduced ROM in \cite{Chan20ROM} by testing \eqref{eq:DG_ROM_HR_weakBC2} with $\bm{v}_N$ and using $\widebar{\bm{Q}}_h\bm{1}=\bm{0}$ and the block SBP property of $\widebar{\bm{Q}}_h$. 

\begin{remark}
  If the boundary condition imposition is entropy stable such that $\bm{1}^T\bm{B}_b\LRp{\bm{\psi}^T(\Tilde{\bm{u}}_b)-\Tilde{\bm{v}}_b^T\bm{f}_b^*} \leq 0$, then the ROM is entropy stable such that 
  \begin{equation}
  \bm{1}^T\bm{W}\f{\rmd S(\bm{V}_N\bm{u}_N)}{\rmd\rmt}\leq0.
  \label{eq:entropy_dissip}
  \end{equation}
  Moreover, if the boundary condition imposition is entropy conservative such that $\bm{\psi}^T(\Tilde{\bm{u}}_b) - \Tilde{\bm{v}}_b^T\bm{f}_b^* = 0$, the inequality \eqref{eq:entropy_dissip} becomes and equality.
\end{remark}

    \end{proof}

\section{Entropy conservative DG scheme on high-dimensional domains}
\label{sec:HDdomains}
For nonlinear conservation laws \eqref{eq:NCL}, extending the DG FOM from 1D to higher dimensions, $d>1$, involves adding contributions from each additional dimension. We focus on the generailized formulation with weakly-imposed boundary conditions in this section.
\subsection{Full order model}


For multi-dimensional domains $\Go \subset \mathbb{R}^d$, we can extend the FOM from \eqref{eq:DG_1D_FOM_weakBC} to 
\begin{equation}
    \bm{M}_\Go\f{\rmd\bm{u}_\Go}{\rmd\rmt} + \sumid\LRp{2\LRp{\bm{Q}^i_\Go\circ\bm{F}^i}\bm{1} + \bm{B}^i_\Go\bm{f}^{i,*}_\Go} = \bm{0},
    \label{eq:DG_HD_FOM_weakBC}
\end{equation}
where now, for the $i$th coordinate, $\bm{Q}_\Go^i$, $\bm{F}^i$ are the global differentiation and flux matrices and $\bm{f}^{i,*}_\Go$ is a corresponding global boundary flux matrix. Because explicit matrix expressions for $\bm{Q}_\Go^i$ and $\bm{B}^i_\Go$ are more cumbersome, we instead prove matrix properties by appealing to properties of the DG weak formulation. 

Suppose we have two arbitrary functions $u,v$ in the DG approximation space with coefficients $\bm{u}$ and $\bm{v}$. Define the exterior values of $u$ on a neighboring element across an element interface as $u^+$. On the boundary $\partial \Omega$, we use the convention that $u^+ = u$. Finally, we define the jump and average as follows:
\[
\avg{u} = \frac{1}{2}\LRp{u^+ + u}, \qquad \jump{u} = u^+ - u.
\]
We can now define $\bm{Q}_\Go^i$ as the matrix which satisfies:
\begin{equation}
\bm{v}^T\bm{Q}_\Go^i\bm{u} = \sum_k\LRp{\int_{D^k}\f{\p u}{\p x^i}v + \int_{\p D^k}\f{1}{2}\llbracket{u}\rrbracket{n}^i{v}}.
\label{eq:DG_HD_Q}
\end{equation}
Similarly, $\bm{B}^i_\Go$ is defined as the matrix which satisfies:
\begin{equation}
    \bm{v}^T\bm{B}_\Go^i\bm{u}=\int_{\partial \Go} v u n^i,
    \label{eq:DG_HD_B}
\end{equation} where $\bm{n}^i$ is the outward normal vector along $i$-th coordinate and all non-bold symbols represent the scalar values of the corresponding vectors over each element $D^k$. In practice, we can construct $\bm{Q}_\Go^i$ by applying $\bm{Q}_\Go^i$ to canonical basis vectors and computing each entry $$\LRp{\bm{Q}^i_\Go}_{j,k} = \bm{e}_j^T\bm{Q}_\Go^i\bm{e}_k.$$

\begin{theorem}
\label{thm:DG_HD_SBP}
Let $\bm{Q}^i_\Go$ satisfy \eqref{eq:DG_HD_Q} and $\bm{B}^i_\Go$ satisfy \eqref{eq:DG_HD_B}. Then $\bm{Q}^i_\Go\bm{1} = \bm{0}$, and $\bm{Q}^i_\Go$ satisfies the SBP property $\bm{Q}_\Go^i + \LRp{\bm{Q}_\Go^i}^T = \bm{B}^i_\Go$. Moreover, $\bm{Q}^i_\Go$ is skew-symmetric for periodic domains.
\end{theorem}
\begin{proof}
    To see $\bm{Q}^i_\Go\bm{1} = \bm{0}$, we look at each row 
$$ (\bm{Q}^i_\Go)_j\bm{1} = \bm{e}_j^T\bm{Q}^i_\Go\bm{1}=\sum_k\LRp{\int_{D^k}\f{\p {1}}{\p {x}^i}{e}_j + \int_{\p D^k}\f{1}{2}\llbracket{1}\rrbracket{n}^i{e}_j} = 0.$$

To prove the SBP property, we use integration by parts to manipulate \eqref{eq:DG_HD_Q}:
\begin{equation*}
    \begin{split}
        \bm{v}^T\bm{Q}_\Go^i\bm{u} & = \sum_k\LRp{\int_{D^k}\f{\p {u}}{\p {x}^i}{v} + \int_{\p D^k}\f{1}{2}\llbracket{u}\rrbracket{n}^i{v}} \\
        & = \sum_k\LRp{-\int_{D^k}{u}\f{\p {v}}{\p {x}^i} + \int_{\p D^k}{u}{n}^i{v}+\f{1}{2}\llbracket{u}\rrbracket{n}^i{v}} \\
        & = \sum_k\LRp{-\int_{D^k}{u}\f{\p {v}}{\p {x}^i} + \int_{\p D^k}\f{1}{2}\avg{u}{n}^i{v}}.
    \end{split}
\end{equation*}
Since the sum is over all elements, each inter-element interface is associated with two surface integrals. Let $f$ denote an inter-element interface between $D^k$ and a neighbor $D^{k,+}$. Then,
\begin{align*}
        \sum_k \int_{\p D^k}\avg{u}{n}^i{v} &= \sum_f\LRp{ \int_{f \cap D^k}\avg{u}n^iv + \int_{f \cap D^{k,+}}\avg{u} n^{i,+} v^+} + \int_{\partial \Omega} u n^i v,
\end{align*}
where the sum is now taken over all inter-element interfaces and boundary faces $f$. Noting that $n^{i,+} = -n^i$ and exchanging terms yields that the sum over interior faces is equivalent to
\begin{align*}
\sum_f \int_{f \cap D^k}\avg{u}n^iv + \int_{f \cap D^{k,+}}\avg{u} n^{i,+} v^+ &= \sum_f\int_{f \cap D^k} u n^i\frac{1}{2}\LRp{v^+-v} + \int_{f \cap D^{k,+}}\avg{u} n^{i,+} \frac{1}{2}\LRp{v - v^+}.
\end{align*}
Thus, the sum of surface contributions yields 
\[
\sum_k \int_{\p D^k}\avg{u}{n}^i{v} = \sum_k \int_{\partial D^k} un^i \jump{v} + \int_{\partial \Omega} u n^i v.
\]
This implies that
$$\bm{v}^T\bm{Q}_\Go^i\bm{u} = \sum_k\LRp{-\int_{D^k}{u}\f{\p {v}}{\p {x}_i} -\int_{\p D^k} \f{1}{2} \llbracket{v}\rrbracket{n}^i{u} } + \int_{\partial \Omega} u n^i v = -\bm{u}^T\LRp{\bm{Q}_\Go^i + \bm{B}^i_\Go}\bm{v}.$$
For periodic domains, the integral over $\partial \Omega$ vanishes since boundary faces are also treated as inter-element interfaces, yielding a skew-symmetric operator $\bm{Q}^i_\Go$.

\label{pf:DG_HD_Q}
\end{proof}

Again, following \cite{Chen17}, the formulation \eqref{eq:DG_HD_FOM_weakBC} is entropy conservative if we use entropy conservative flux for boundary terms under appropriate boundary conditions.

\begin{remark}
The proof of Theorem~\ref{thm:DG_HD_SBP} assumes that integration by parts holds along each dimension. This is satisfied even under inexact $(N+1)$ point Gauss-Lobatto quadrature \cite{kopriva2010quadrature}, which is used in the FOMs considered in this work.
\end{remark}

\subsection{Volume hyper-reduction}

\par The volume flux terms of \eqref{eq:DG_HD_FOM_weakBC} after applying Galerkin projection are
\begin{equation}
\sum_{i=1}^{d}\LRp{2\bm{V}_N^T\LRp{\bm{Q}_\Go^i\circ\bm{F}^i}\bm{1}} \csp \bm{F}^i_{jk} = \bm{f}_{EC}^i(\Tilde{\bm{u}}_j,\Tilde{\bm{u}}_k)\csp \Tilde{\bm{u}} = \bm{u}\LRp{\bm{V}_N\bm{V}_N^\dagger\bm{v}(\bm{V}_N\bm{u}_N)}. 
    \label{eq:DG_HD_vol}
\end{equation}
Our volume hyper-reduction approach is to hyper-reduce each coordinate separately, akin to the 1D case, using the same quadrature rule in each coordinate direction. This quadrature rule is generated by targeting $\bm{V}_N*\bm{V}_N$ to yield hyper-reduced nodes $\mathcal{I}$ and weights $\bm{w}$.
\par Now, given the basis matrix $\bm{V}_N$, we can introduce a test basis $\bm{V}_t^i$ for each dimension. We assume that the spans of $\bm{V}_t^i$ include the ranges of $\bm{V}_N$ and $\bm{M}_\Go^{-1}\LRp{\bm{Q}_\Go^i}^T\bm{V}_N$ by imposing the condition,
\begin{equation}
  \mathcal{R}\LRp{\bm{V}_t^i} = \mathcal{R}\LRp{\begin{bmatrix}
           \bm{1} &  \bm{V}_N & \bm{M}_\Go^{-1}\LRp{\bm{Q}_\Go^i}^T\bm{V}_N
    \end{bmatrix}} \csp i=1,...,d.
    \label{eq:test_basis_2d}
\end{equation}
We note that, while the test basis varies depending on the coordinate dimension $i$, the rank of the test basis is at most $2N+1$. 

Next, we compute projection matrices for each dimension $$\widebar{\bm{V}}_t^i = \bm{V}_t^i\LRp{\mathcal{I},:}\csp\bm{W} = \text{diag}\LRp{\bm{w}}\csp\bm{P}_t^i = \LRp{\widebar{\bm{V}}^i_t\bm{W}\widebar{\bm{V}}^i_t}^{-1}\LRp{\widebar{\bm{V}}^i_t}^T\bm{W}.$$ These matrices can be used to construct hyper-reduced nodal differentiation matrices  $$\widebar{\bm{Q}}^i = \LRp{\bm{P}_t^i}^T\LRp{\bm{V}_t^i}^T\bm{Q}_\Go^i\bm{V}_t^i\bm{P}_t^i,$$ where each $\widebar{\bm{Q}}^i$ is skew-symmetric and has zero row sums. After this, the flux terms in \eqref{eq:DG_HD_FOM_weakBC} are evaluated using entropy-projected conservative variables interpolated independently for each coordinate.



We note that, if we target other space like $\bm{V}_N*\bm{V}_t$ or $\bm{V}_t * \bm{V}_t$ for hyper-reduction on each dimension, then it results in distinct positive quadrature rules for each coordinate with different sets of hyper-reduced nodes $\mathcal{I}^i$ and weights $\bm{w}^i$. We do not consider this strategy in this paper, as this strategy results in an increased number of hyper-reduced nodes but not a significant improvement in accuracy in our experiments. 

\subsection{Boundary hyper-reduction}

In 2D and 3D, hyper-reduction is required not just for volume terms, but also for boundary terms that appear in \eqref{eq:DG_HD_FOM_weakBC}. However, as discussed in \cite{Chan20ROM}, the boundary hyper-reduction must preserve certain properties involving $\bm{Q}_\Go^i$ and $\bm{B}_\Go^i$ to guarantee entropy conservation or stability. We note that it is possible to apply boundary hyper-reduction independently of the volume hyper-reduction. 

\par We first describe how to construct an entropy conservative or entropy stable ROM formulation in the presence of boundary hyper-reduction, which is not necessary in 1D. Denote $\bm{V}_b$ and $\bm{V}_{bt}$ as the boundary sub-matrices of the POD basis and test basis respectively, such that their columns correspond to values of basis functions at boundary nodes. Let $\mathcal{I}_b$ be a sub-sampled set of boundary nodes, $\bm{B}_b^i$ be the diagonal matrix whose entries consist of coordinate values of outward normals on the boundary scaled by quadrature weights, and let $\bm{E}^i$ be the matrix which extrapolates from hyper-reduced volume points to hyper-reduced boundary points for $i$th coordinate:
\begin{equation}
\widebar{\bm{B}}_b^i=\text{diag}(\bar{\bm{n}}_b^i \circ \bar{\bm{w}}_b)\csp \widebar{\bm{V}}_{b}^i = \bm{V}_b^i(\mathcal{I}_b,:)\csp \widebar{\bm{V}}_{bt}^i = \bm{V}_{bt}^i(\mathcal{I}_b,:)\csp \widebar{\bm{E}}^i=\widebar{\bm{V}}_{bt}^i\bm{P}_t^i,
\label{eq:boundary_HR_ops}
\end{equation}
where $\circ$ is the Hadamard (elementwise) product of two vectors, $\bar{\bm{n}}_b^i$ is the vector of values of the $i$th coordinate of the outward normal at hyper-reduced boundary nodes, $\bar{\bm{w}}_b$ and $\bm{V}_{bt}^i$ is a sub-matrix of $\bm{V}_t^i$ at boundary nodes. Note that $\bm{E}^i$ is defined in terms of the projection $\bm{P}_t^i$ onto the $i$th test basis.

\par We introduce the hybridized SBP operator \cite{Chan18}  along the $i$th coordinate:
\begin{equation*}
    \widebar{\bm{Q}}_h^i = \begin{bmatrix}
        \widebar{\bm{Q}}^i-\LRp{\widebar{\bm{Q}}^i}^T & \LRp{\widebar{\bm{E}}^i}^T\widebar{\bm{B}}_b^i \\
        -\widebar{\bm{B}}_b^i\widebar{\bm{E}}^i & \widebar{\bm{B}}_b^i
    \end{bmatrix},
\end{equation*}
which can be used to construct a hyper-reduced ROM in higher dimensions following \cite{Chan20ROM}:
\begin{equation}
\begin{split}
   & \widebar{\bm{M}}_N\f{\rmd\bm{u}_N}{\rmd\rmt} + \sumid\LRp{2\LRp{\widebar{\bm{V}}_h}^T\LRp{\widebar{\bm{Q}}_h^i\circ\bm{F}^i}\bm{1} + \LRp{\widebar{\bm{V}}_{b}}^T\widebar{\bm{B}}_b^i\LRp{\bm{f}_b^{i,*}-\bm{f}^i(\Tilde{\bm{u}}_b)}} =\bm{0}, \\
     &\widebar{\bm{V}}_N = \bm{V}_N\LRp{\mathcal{I},:}\csp \widebar{\bm{V}}_h = \begin{bmatrix}
       \widebar{\bm{V}}_N  \\ \bm{V}_b
   \end{bmatrix}\csp \widebar{\bm{V}}_{b} = \bm{V}_b(\mathcal{I}_b,:),\\
      & \widebar{\bm{M}}_N = \widebar{\bm{V}}_N^T\bm{W}\widebar{\bm{V}}_N\csp \bm{P}_N = \widebar{\bm{M}}_N^{-1}\widebar{\bm{V}}_N^T\bm{W}\csp \bm{v}_N = \bm{P}_N\bm{v}\LRp{\widebar{\bm{V}}_N\bm{u}_N},\\
   & \Tilde{\bm{v}}=\widebar{\bm{V}}_N\bm{v}_N\csp \Tilde{\bm{u}}=\bm{u}\LRp{\Tilde{\bm{v}}}\csp\LRp{\bm{F}^i}_{jk}=\bm{f}_{EC}^i\LRp{\Tilde{\bm{u}}_j,\Tilde{\bm{u}}_k}. 
\end{split}
    \label{eq:DG_ROM_HD_HR_weakBC1}
\end{equation}
By the block SBP property of $\widebar{\bm{Q}}_h^i$, we can show that formulation \eqref{eq:DG_ROM_HD_HR_weakBC1} is equivalent to
\begin{equation}
    \widebar{\bm{M}}_N\f{\rmd\bm{u}_N}{\rmd\rmt} + \sumid\LRp{\LRp{\widebar{\bm{V}}_h^i}^T\LRp{\LRp{\widebar{\bm{Q}}_h^i-\LRp{\widebar{\bm{Q}}_h}^T}\circ\bm{F}^i}\bm{1} + \LRp{\widebar{\bm{V}}_{b}}^T\widebar{\bm{B}}_b^i\bm{f}_b^{i,*}} =\bm{0},
    \label{eq:DG_ROM_HD_HR_weakBC2}
\end{equation}
which is often more convenient to implement in practice. 

\subsection{Sufficient conditions for entropy conservation}

In \cite{Chan20ROM}, it was shown that the proof of entropy conservation in \autoref{thm:weakBC_stability} could be extended to formulation \eqref{eq:DG_ROM_HD_HR_weakBC2} if $\widebar{\bm{Q}}_h^i\bm{1}=\bm{0}$. Expanding out this expression yields: 
\begin{equation}
    \widebar{\bm{Q}}_h^i\bm{1} = \begin{bmatrix}
        \widebar{\bm{Q}}^i\bm{1}-\LRp{\widebar{\bm{Q}}^i}^T\bm{1}+\LRp{\widebar{\bm{E}}^i}^T\widebar{\bm{B}}_b^i\bm{1} \\ \bm{0}
    \end{bmatrix} = \begin{bmatrix}
        -\LRp{\widebar{\bm{Q}}^i}^T\bm{1}+\LRp{\widebar{\bm{E}}^i}^T\widebar{\bm{B}}_b^i\bm{1} \\ \bm{0}
    \end{bmatrix}. 
    \label{eq:conservation_nonperiodic}
\end{equation}
Since we wish to enforce that the right hand side of \eqref{eq:conservation_nonperiodic} reduces to $\bm{0}$,  when implementing hyper-reduction for the boundary terms, we will impose constraints on the boundary weights $\bm{w}_b$ such that $$\LRp{\widebar{\bm{Q}}^i}^T\bm{1} = \LRp{\widebar{\bm{E}}^i}^T\widebar{\bm{B}}_b^i\bm{1}.$$ 
Recall in \eqref{eq:boundary_HR_ops} that the boundary weights appear through $\widebar{\bm{B}}_b^i = \text{diag}(\bm{n}_b^i \circ \bm{w}_b)$. We can eliminate the dependence of \eqref{eq:boundary_HR_ops} on the volume hyper-reduced matrix $\widebar{\bm{Q}}^i$ by observing that:
\begin{equation*}
    \begin{split}
        &\LRp{\widebar{\bm{Q}}^i}^T\bm{1} = \LRp{\widebar{\bm{E}}^i}^T\widebar{\bm{B}}_b^i\bm{1} \\
       & \iff\bm{1}^T\widebar{\bm{Q}}^i = \bm{1}^T\widebar{\bm{B}}_b^i\widebar{\bm{E}}^i \\
        &\iff \bm{1}^T(\bm{V}_t^i\bm{P}_t^i)^T\bm{Q}_\Go^i\bm{V}_t^i\bm{P}_t^i = \bm{1}^T\widebar{\bm{B}}_b^i\widebar{\bm{V}}_{bt}^i\bm{P}_t^i.
    \end{split}
\end{equation*}
Using the fact that $\bm{V}_t^i\bm{P}_t^i=\bm{1}$ then yields that the following conditions are equivalent:
\[
\LRp{\widebar{\bm{Q}}^i}^T\bm{1} = \LRp{\widebar{\bm{E}}^i}^T\widebar{\bm{B}}_b^i\bm{1}  \iff \bm{1}^T\bm{Q}_\Go^i\bm{V}_t^i  = \bm{1}^T\widebar{\bm{B}}_b^i\widebar{\bm{V}}_{bt}^i.
\]
Hence, enforcing 
\begin{equation}
\bm{1}^T\widebar{\bm{B}}_b^i\widebar{\bm{V}}_{bt}^i = \bm{1}^T\bm{Q}_\Go^i\bm{V}_t^i
\label{eq:boundary_HR_condition}
\end{equation}
is sufficient to preserve entropy conservation.

\subsection{Carath\'eodory pruning}
\par To enforce \eqref{eq:boundary_HR_condition}, we introduce a method which we refer to as Carath\'eodory pruning which allows us to hyper-reduce the boundary terms while preserving both positivity of the boundary weights and condition \eqref{eq:boundary_HR_condition}. Moreover, unlike the LP-based hyper-reduction algorithm used in \cite{Chan20ROM, Yano19}, if $N$ moment conditions are enforced, then Carath\'eodory pruning yields exactly $N$ hyper-reduced nodes. 

Carath\'eodory's theorem in convex analysis can be used to conclude that, for any $M$-point positive quadrature rule which is exact for a space spanned by $\text{span}\{v_1,...,v_N\}$, we can always generate a new $N$-point interpolatory positive rule to preserve all moments \cite{vandenBos20}. More precisely, suppose $M\geq N$ and for all $n=1,...,N:$ $$m_n := \int v_n(x)dx = \sum_{m=1}^Mw_mv_n(x_m).$$
This is equivalent to 
\begin{equation*}
    \bm{V}^T\bm{w} = \begin{bmatrix}
        v_1(x_1) & v_1(x_2) & \hdots & v_1(x_M) \\
        v_2(x_1) & v_2(x_2) & \hdots & v_2(x_M) \\
        \vdots & \vdots & \ddots & \vdots \\
        v_N(x_1) & v_N(x_2) & \hdots & v_N(x_M)
    \end{bmatrix} \begin{bmatrix}
        w_1 \\ w_1 \\ \vdots \\ w_M
    \end{bmatrix} = \bm{m},
\end{equation*}
with $\bm{w}\geq0$, which implies that $\bm{m}$ is in the convex hull of $\bm{0}$ and the $M$ columns of $\bm{V}^T$. Then, Carath\'eodory's theorem states that $\bm{m}$ also lies in the convex hull of a subset of $N$ columns of $\bm{V}^T$. This subset can be computed in practice using Algorithm~\ref{alg:cara_pruning} from \cite{vandenBos20} to construct a new positive quadrature rule.

\begin{algorithm}
\caption{Carath\'eodory pruning (Algorithm 1 in \cite{vandenBos20})}\label{alg:cara_pruning}
\hspace*{\algorithmicindent} \textbf{Input}: $\bm{V}^T$, $\bm{w}$.\\
\hspace*{\algorithmicindent} \textbf{Output}: $\bm{w}$, $\mathcal{I}$.
\begin{algorithmic}[1]
\State Asset error if any entry of $\bm{w}<0.$
\State Set $M,N = \text{size}(\bm{V}^T)$.
\State Return $\bm{V}^T$ and $\bm{w}$ if $M \leq N$.
\State Set $m = M-N$ and $\mathcal{I} = 1:M$.
\For{q = 1:m}
\State Determine null vector $\bm{c}$ using pivoted QR decomposition 
$$\bm{c} \gets \text{qr}(\bm{V}^T(1:N+1,:))\csp \bm{c} \gets \bm{c}(:,\text{end}).$$
\State Find the indices for all positive components 
$$\text{id} \gets \text{findall}(\bm{c}>0).$$
\State Select $\alpha$, $k$ such that $\bm{w}-\alpha\bm{c} \geq 0$ and $\bm{w}(k) = \alpha\bm{c}(k)$
$$\alpha \gets \text{min}(\bm{w}(\text{id}) ./ \bm{c}(\text{id})) \csp k \gets \argmin(\bm{w}(\text{id}) ./ \bm{c}(\text{id}))\csp k \gets \text{id}(k).$$ 
\State Remove $k$-th entry in $\bm{w}$ and $\mathcal{I}$, and remove $k$-th row in $\bm{V}^T$.
\EndFor
\State Return $\bm{w}$, $\mathcal{I}$.
\end{algorithmic}
\end{algorithm}

\par Now, we consider our setting where $N$ is the number of POD modes. Let $\phi^i_{t, j}$ denote a test basis function for the $i$th coordinate, where the test basis spans the approximation space specified in \eqref{eq:test_basis_2d} and of dimension at most $(2N+1)$. Since $\phi^i_{t, j}$ restricted to a surface face is in the DG FOM approximation space, and since the surface quadrature weights $\bm{w}_b$ correspond to a high order composite Gauss-Lobatto rule, the following exactness conditions hold in the absence of hyper-reduction: 
\begin{equation}
\bm{1}^T\bm{B}_b^i\bm{V}_{bt}^i = \int_{\partial \Omega} \phi_{t,j}^i (\bm{x}_k)\bm{n}_b^i = \sum_{k}\bm{w}_{b,j}\bm{n}^i_{b,j}\phi_{t,j}(\bm{x}_k).
\label{eq:boundary_moments}
\end{equation}
However, while we have $d$ test bases in $d$ dimensions, it is more computationally convenient to compute a single set of hyper-reduced boundary nodes and indices for all coordinate directions. We can do so by concatenating matrices for each dimension together and setting 
\begin{equation*}
     \bm{V} = \begin{bmatrix}
        \text{diag}(\bm{n}_b^1)\bm{V}_{bt}^1 & ... & \text{diag}(\bm{n}_b^d)\bm{V}_{bt}^d
    \end{bmatrix}^T.
\end{equation*}
The moment conditions \eqref{eq:boundary_moments} then can be rewritten as
\begin{equation}
\bm{w}_b^T\bm{V} = \left[ \bm{1}^T\bm{Q}^1_\Omega\bm{V}^1_t \quad \ldots \quad \bm{1}^T\bm{Q}^d_\Omega\bm{V}^d_t \right].
\label{eq:concatenated_moment_conditions}
\end{equation}
After transposing, we can apply Carath\'eodory pruning to \eqref{eq:concatenated_moment_conditions} via Algorithm~\ref{alg:cara_pruning} to construct a single set of hyper-reduced boundary weights $\bm{w}_b$ and node indices $\mathcal{I}_b$. The new boundary quadrature weights and index can subsequently be utilized to construct the requisite matrices for the hyper-reduced ROM \eqref{eq:DG_ROM_HD_HR_weakBC2} while preserving entropy conservation.

\section{Discretization of artificial viscosity}
\label{sec:viscosity}
Although entropy conservative high-order methods provide a statement of stability, such formulations produce spurious oscillations in the presence of shocks, sharp gradients, or under-resolved features. We thus implement an artificial viscosity to damp such spurious oscillations. We note that artificial viscosity is not needed to achieve stabilization, and will present numerical experiments in a later section demonstrating that an entropy conservative ROM remains stable in the presence of shocks. 


\par The artificial viscosity used in this work \eqref{eq:NCL} is a simple scalar Laplacian viscosity:
\begin{equation}
    \f{\p \bm{u}}{\p \rmt} + \sum_{i=1}^{\rmd} \f{\p \bm{f}^i(\bm{u})}{\p \bm{x}^i} = \eps \Gd \bm{u} \csp \eps \geq 0.
    \label{eq:NCL_viscous}
\end{equation} 
Our DG FOM utilizes the Bassi-Rebay-1 (BR-1) scheme \cite{bassi1997high} to discretize the Laplacian. In one dimension, the BR-1 scheme defines on each element $\bm{\sigma}  
\approx \f{\partial \bm{u}}{\partial x}$ as $$\bm{\sigma}=(J_k\bm{M})^{-1}\LRp{\bm{Q}\bm{u}+\bm{B}\avg{\bm{u}}}\csp \frac{\partial^2\bm{u}}{\partial x^2} = \pd{\bm{\sigma}}{x} \approx \LRp{J_k\bm{M}}^{-1}\LRp{\bm{Q}\bm{\sigma}+\bm{B}\avg{\bm{\sigma}}}. $$
 where $\avg{\bm{u}}, \avg{\bm{\sigma}}$ are vectors of values corresponding to the central fluxes. 

At a global level, the BR-1 discretization approximates both $\pd{\bm{u}}{x}$ and $\pd{\bm{\sigma}}{x}$ using a DG discretization with central fluxes. In other words, the coefficients for $\bm{\sigma}$ can be computed by applying the global differentiation matrix $\bm{Q}_\Go$ defined in \eqref{eq:DG_1D_FOM_weakBC} to $\bm{u}$ where $\avg{\bm{u}}$ and $\avg{\bm{\sigma}}$ can be imposed by the SBP property \eqref{eq:DG_1D_Qglobal_SBP_weakBC}. Let $\bm{\sigma}_\Go, \bm{u}_\Go$ denote the vector of coefficients for $\bm{\sigma}, \bm{u}$, respectively. Then, 
\begin{equation}
\label{eq:visc_SBP}
  \bm{\sigma}_\Go = \bm{M}_\Go^{-1}\LRp{-\bm{Q}_\Go^T\bm{u}_\Go + \bm{B}_\Go \avg{\bm{u}}}\csp \frac{\partial^2 \bm{u}_\Go}{\partial x^2} = \pd{\bm{\sigma}_\Go}{x} \approx  \bm{M}_\Go^{-1}\LRp{-\bm{Q}_\Go^T\bm{\sigma}_\Go + \bm{B}_\Go \avg{\bm{\sigma}}}.  
\end{equation}

For non-periodic domains, we impose Neumann boundary conditions on $\bm{\sigma}$ weakly by setting $\avg{\bm{u}}= \bm{u}$ and $\avg{\bm{\sigma}} = \bm{0}$ on domain boundaries. In this setting, we recover the global evaluations from \eqref{eq:visc_SBP} such that 
\[
\bm{\sigma}_\Go = \bm{M}_\Go^{-1}\LRp{-\bm{Q}_\Go^T\bm{u}_\Go + \bm{B}_\Go\bm{u}_\Go}= \bm{M}_\Go^{-1}\bm{Q}_\Go\bm{u}_\Go\csp \pd{\bm{\sigma}_\Go}{x} \approx  -\bm{M}_\Go^{-1}\bm{Q}_\Go^T\bm{\sigma}_\Go.
\] 
After multiplying through by the mass matrix, one can derive that the viscous contribution to the entropy conservative formulation \eqref{eq:DG_1D_FOM_weakBC} is as follows: 
\begin{equation}
 -\epsilon\bm{Q}^T_\Go\bm{M}_\Go^{-1}\bm{Q}_\Go \bm{u}_\Go.
    \label{eq:DG_1D_FOM_visc}
\end{equation}
After Galerkin projection, the viscous contribution to 1D ROMs is $  -\epsilon\bm{K}_N\bm{u}_N$, where the matrix $\bm{K}_N$ is defined as
\begin{equation}
\begin{split}
  \bm{K}_N = \bm{V}_N^T\bm{Q}^T_\Go\bm{M}_\Go^{-1}\bm{Q}_\Go\bm{V}_N. 
\end{split}
    \label{eq:DG_ROM_ES_visc}
\end{equation}
For higher dimensions,  
the FOM viscous term in terms of global DG matrices now have contributions from each coordinate dimension:
\begin{equation}
 - \sumid \LRp{ \epsilon \LRp{\bm{Q}_\Go^i}^T\bm{M}_\Go^{-1}\bm{Q}_\Go^i \bm{u}_\Go}.
    \label{eq:DG_HD_FOM_PBC_visc}
\end{equation}
and the corresponding viscous ROM term is $-\epsilon\bm{K}_N\bm{u}_N$, where $\bm{K}_N$ is now defined as
\begin{equation}
    \bm{K}_N = \sumid\bm{V}_N^T\LRp{\bm{Q}_\Go^i}^T\bm{M}_\Go^{-1}\bm{Q}_\Go^i\bm{V}_N.
    \label{eq:DG_HD_HR_PBC_visc}
\end{equation}

We note that, for a straightforward BR-1 discretization of the Laplacian, we are unable to prove entropy dissipation for either the FOM or ROM formulations. It is possible to construct provably entropy stable formulations of the Laplacian artificial viscosity by symmetrizing under entropy variables \cite{Chan20ROM, chan2022entropy}. Specifically, let $\bm{D}_\Go = \bm{M}_\Go^{-1} \bm{Q}_\Go$ be the global strong DG derivative operator. Then, for 1D ROM formulation, the viscous contribution can be written as:
\begin{equation}
        - \epsilon \bm{V}_N^T  \bm{D}_\Go^T \bm{H}_\Go \bm{M}_\Go \bm{D}_\Go \bm{V}_N\bm{v}_N = - \epsilon \bm{V}_N^T  \bm{Q}_\Go^T \bm{M}_\Go^{-1}\bm{H}_\Go \bm{Q}_\Go\bm{V}_N \bm{v}_N,
        \label{eq:ES_visc}
\end{equation}
where $\LRp{\bm{H}_\Go}_{ii} = \left.\pd{\bm{u}}{\bm{v}} \right |_{\bm{u} = \tilde{\bm{u}}_i} > \bm{0}$ for convex entropy function $S(\bm{u})$. This allows provably entropy dissipation from the viscosity if we test with $\bm{v}_N$
\[
-\bm{v}_N^T \bm{V}_N^T \bm{Q}_\Go^T  \bm{M}_\Go^{-1} \bm{H}_\Go \bm{Q}_\Go\bm{V}_N \bm{v}_N = -\sum_i \frac{\LRp{\bm{H}_\Go}_{ii}}{\LRp{\bm{M}_\Go}_{ii}}\LRp{\bm{Q}_\Go\bm{V}_N\bm{v}_N}_{i}^2 \leq 0.
\]
Formulation \eqref{eq:ES_visc} also extends naturally to multi-dimensional ROMs, and a hyper-reduction treatment can be applied using the same strategy as for the flux terms \cite{Chan20ROM}.

However, it was observed in \cite{Chan20ROM} that using a provably entropy stable approximation of the artificial viscosity did not result in significant differences compared to a more standard naive BR-1 discretization. Moreover, for all numerical experiments, we observe numerically that our ROM discretization \eqref{eq:DG_HD_HR_PBC_visc} is entropy dissipative, so we use the simpler approximation $\bm{K}_N$ for the remainder of this work.

\begin{remark}
We note that, while a naive discretization of the artificial viscosity is sufficient for the problems considered in this paper, it has been observed that entropy stable viscous discretizations significantly improve robustness for problems with viscous no-slip boundary conditions \cite{ chan2022entropy, parsani2015entropy, dalcin2019conservative}. These will be considered in future work. 
\end{remark}

\section{Numerical experiments}
\label{sec:NumExp}
In this section, we first compare the behavior of entropy stable ROMs based on FVM and DG FOMs on 1D linear advection, Burgers', and the compressible Euler equations. We then analyze the behavior of entropy stable DG-ROMs on benchmark problems for the compressible Euler equations in both 1D and 2D.

\subsection{Preliminaries}
\label{sec:preliminary}
All numerical errors reported in this chapter are relative $L^2$ errors between FOM and ROM solutions defined as 
$$\f{||\bm{u}_\text{FOM} - \bm{u}_\text{ROM}||_{L^2}}{||\bm{u}_\text{FOM}||_{L^2}},$$ where $$||\bm{u}_\text{FOM}||^2_{L^2} = \sum_{i=1}^d ||\bm{u}^i_\text{FOM} ||^2_{L^2} \csp ||\bm{u}_\text{FOM} - \bm{u}_\text{ROM}||^2_{L^2} = \sum_{i=1}^d ||\bm{u}^i_\text{FOM} - \bm{u}^i_\text{ROM}||^2_{L^2}.$$ The norms are computed using the associated FOM quadrature rule (e.g., Gauss-Lobatto for DG). The artificial viscous term is implemented exactly as described in Section \ref{sec:viscosity}. 

All experiments are implemented using the Julia programming language. Unless otherwise specified, we save 400 solution frames for constructing $\bm{V}_\text{snap}$ to prevent the introduction of errors due to an insufficient number of frames. We also fix the target space for hyper-reduction to be $\bm{V}_N * \bm{V}_N$. Finally, the numerical solution is evolved in time using an explicit 5-stage 4th order Runge-Kutta (RK4), 6-stage 5th order Runge-Kutta (Tsit5) \cite{TSITOURAS2011}, or a stabilized Runge-Kutta scheme (ROCK4) \cite{Abdulle2001}. All simulation codes are available at: \url{https://github.com/rayqu1126/ES_DG_ROM}.

\subsection{Comparison between FVM and DG ROMs}
In this section, we present numerical experiments on a periodic 1D domain $[-1,1]$ to compare FVM and DG ROMs under hyper-reduction. Because FVM and higher order DG methods have different numbers of degrees of freedom per element, we fix the total dimension of our FOMs to be 1024, so that higher polynomial degrees result in FOMs with coarser meshes.

\subsubsection{Linear advection equation}

To begin, we compare FVM and DG-based ROMs for the simplest conservation law: the scalar linear advection equation. In 1D, the linear advection equation is given by 
$$\f{\partial u}{\partial \rmt}+\f{\partial u}{\partial x} = 0.$$
We utilize the square entropy $S(u) = u^2/2$, such that entropy variables are the same as the conservative variables. The corresponding two-point entropy conservative flux is simply the central flux $${f}_\text{EC}(u_L,u_R) = \f{1}{2}(u_L+u_R).$$ 



We test our FOM and ROM using the initial condition $$u_0=e^{-50x^2}$$ without adding any artificial viscosity. We run the FOM until final time $T=1.0$. This setting results in a relatively fast decay of the singular values in $\bm{V}_{\text{snap}}$ (\autoref{fig_sval2}). However, we note that the singular values in general decay more slowly for general transport-dominated equations. This is also known as the large Kolmogorov $n$-width problem \cite{Ohlberger16}.

\autoref{tb:Adv1} shows the FOM error to the analytical solutions and the hyper-reduction error varying the number of nodes for polynomial degrees $p=0$ (1024 elements), $p=3$ (256 elements), and $p=7$ (128 elements). We note that $p=0$ corresponds to a FVM FOM. We observe a larger error in the FVM FOM when compared to the analytical solution. However, when using the FOM as the reference, all hyper-reduced ROMs achieve similar accuracy given a fixed number of POD modes $N$. Moreover, more hyper-reduced nodes are required for DG methods, particularly for a larger number of modes.
\begin{figure}
\centering
\noindent
\subfloat[$p=0$ (FVM)]{\includegraphics[width=.33\textwidth]{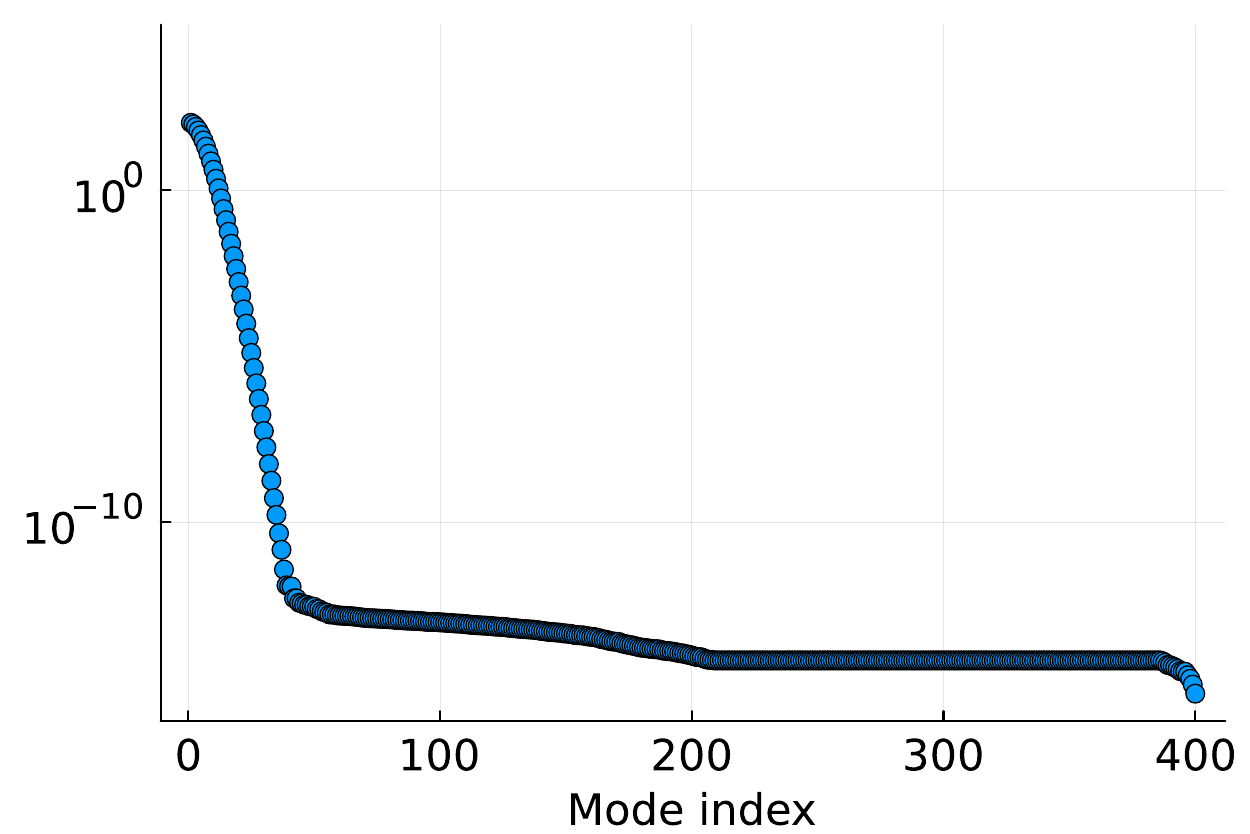}}
\subfloat[$p=3$]{\includegraphics[width=.33\textwidth]{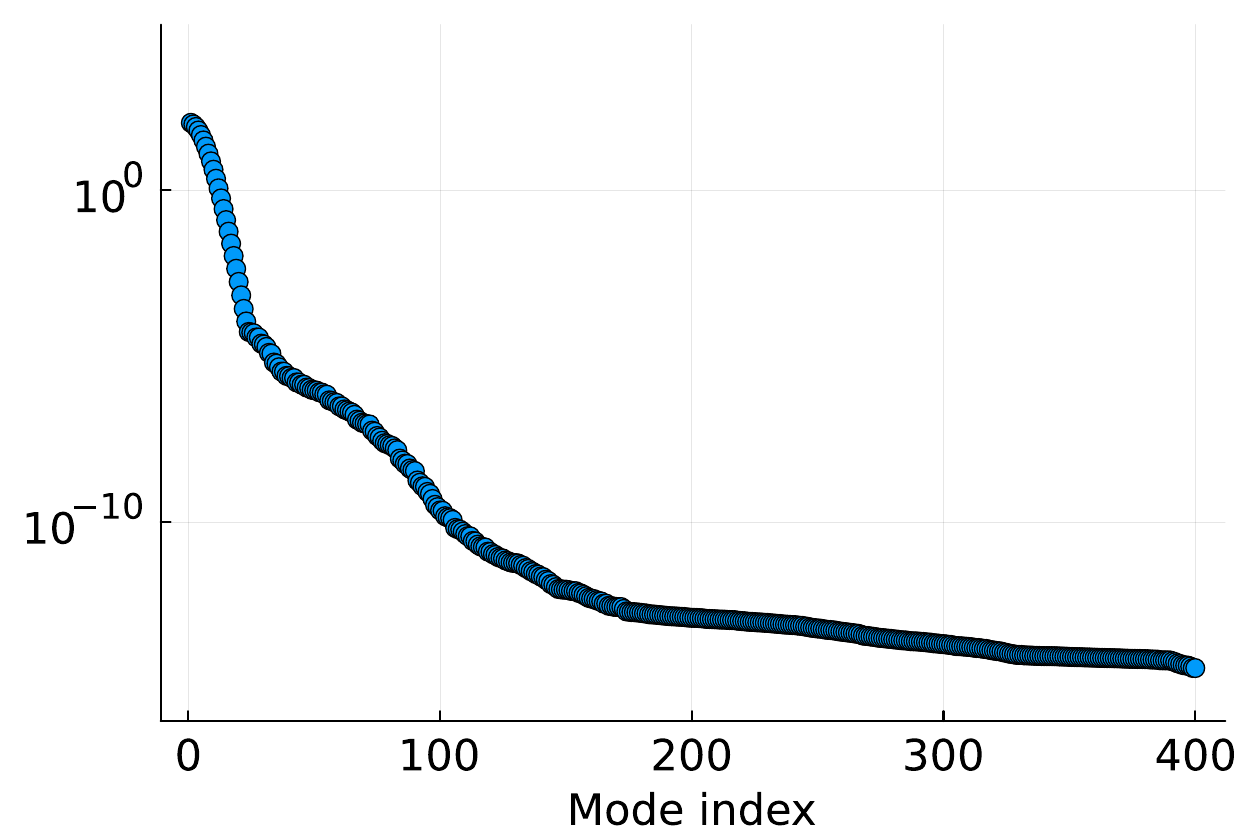}}
\subfloat[$p=7$]{\includegraphics[width=.33\textwidth]{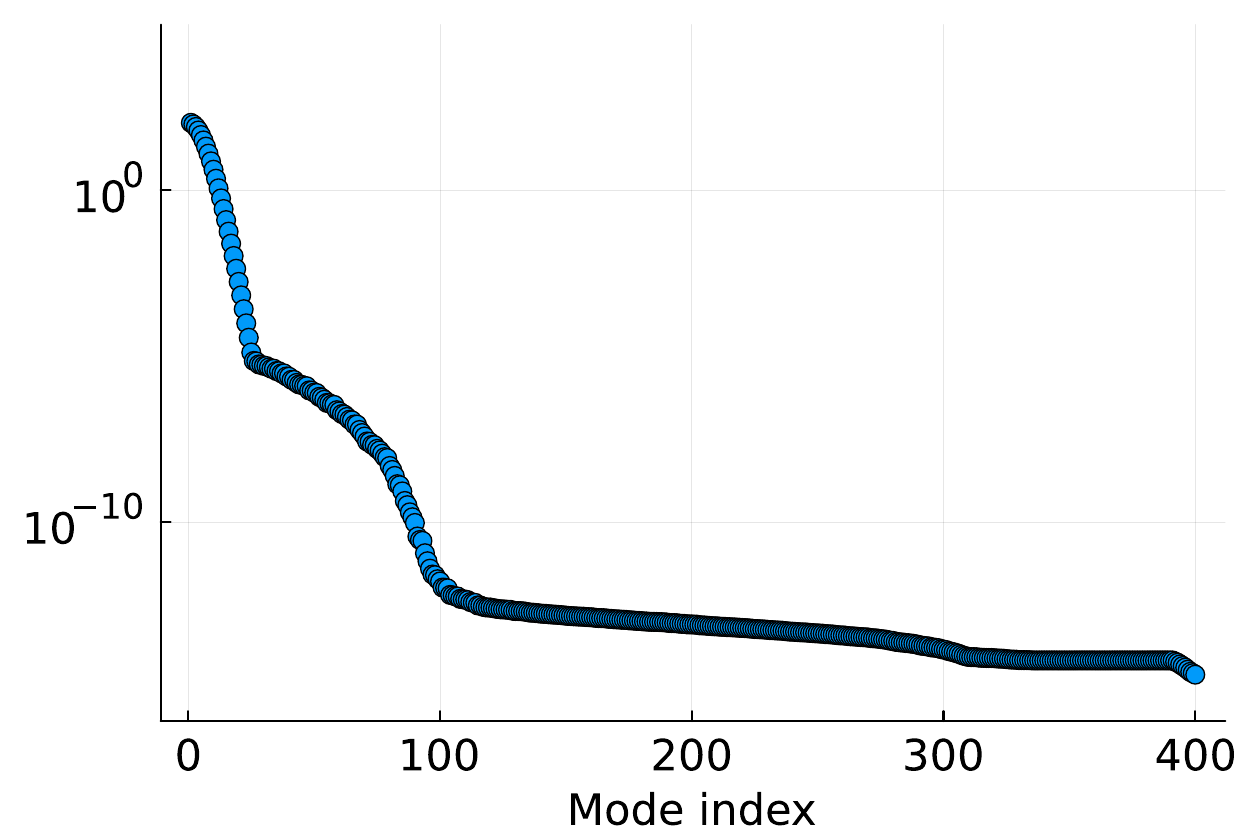}}
\caption{Singular values of $\bm{V}_\text{snap}$ in linear advection equation.}
\label{fig_sval2}
\end{figure}

\begin{table}
\begin{center}
\begin{tabular}{ | l | c | c | c | c | c | c |}
\hline
   Error/nodes& $p=0$ & $p=3$ & $p=7$ \\ \hline
  FOM (to analytical) & 8.71e-4/ 1024 &  1.00e-7/ 1024 & 5.68e-8/ 1024 \\ \hline
   $N=15$ &  1.28e-3/ 31 & 1.29e-3/ 32  &  1.29e-3/ 31 \\   \hline
  $N=20$ &  1.52e-5/ 44 & 1.55e-5/ 81 & 1.55e-5/ 68\\   \hline
  $N=25$ &  9.12e-8/ 58 &  5.55e-7/ 164 & 1.02e-7 /155 \\ \hline
\end{tabular}
\caption{Error and number of hyper-reduced nodes in linear advection equation}
\label{tb:Adv1}
\end{center}
\end{table}

\subsubsection{Burgers' equation}
\label{sec:burgers}
The next equation we investigate is Burgers' equation $$\f{\partial u}{\partial \rmt}+u\f{\partial u}{\partial x} = 0.$$ 
As with the linear-advection equation, we use the square entropy $S(u) = u^2/2$, such that the entropy variables are identical to the conservative variables $u$. The entropy conservative two-point flux is $$f_\text{EC}(u_L,u_R) = \f{1}{6}\LRp{u_L^2+u_Lu_R+u_R^2}.$$
The initial condition is set as $$u_0(x) = 0.5-\sin(\pi x),$$
which forms a traveling shock some time between $t=0.2$ and $t=0.4$. The viscosity coefficient is set as $\epsilon=1\times 10^{-2}$. We simulate until final time $T=1.0$ when the shock forms and has started propagating. The solution plot with $N=20$ can be found in \autoref{fig:testbasis}. \autoref{fig_sval3} shows the singular values, which decay more slowly in this case. We report results in \autoref{tb:Burgers1}. 

\begin{figure}
\centering
\noindent
\subfloat[$p=0$ (FVM)]{\includegraphics[width=.33\textwidth]{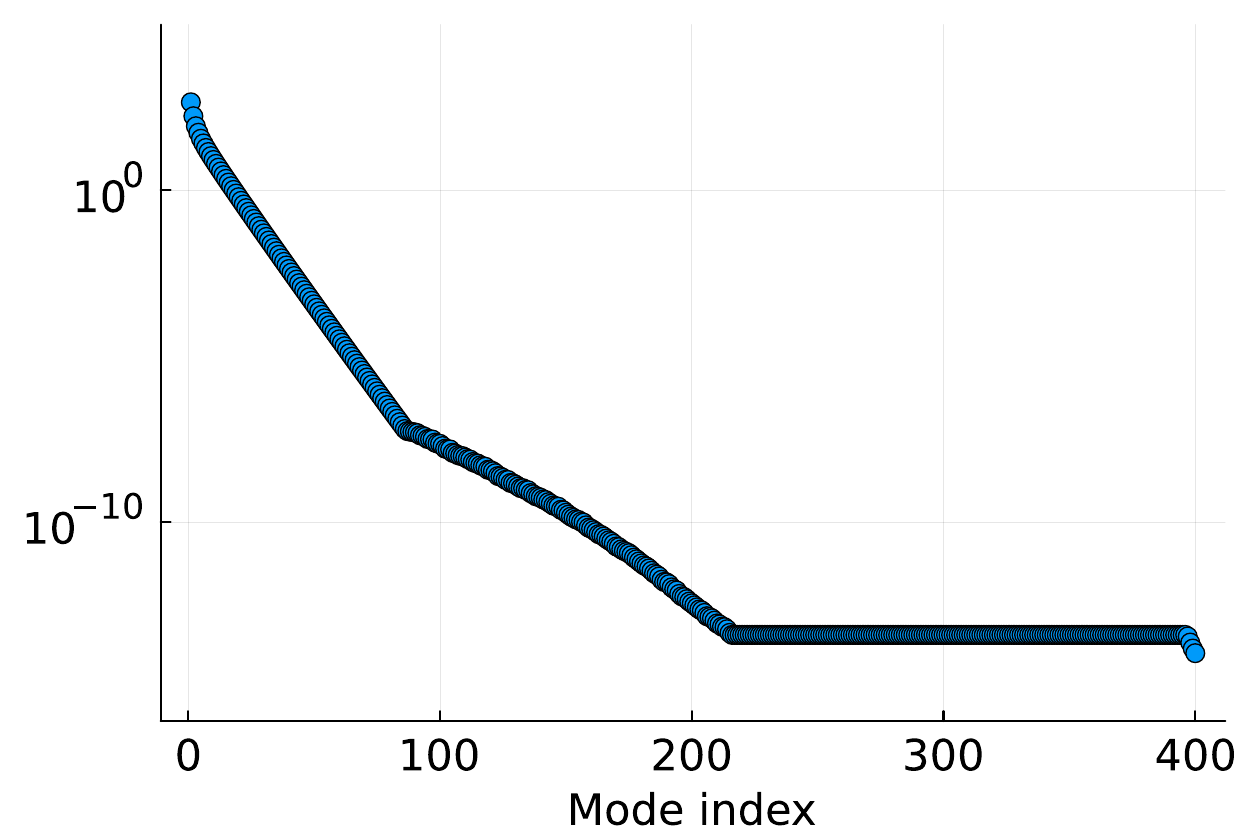}}
\subfloat[$p=3$]{\includegraphics[width=.33\textwidth]{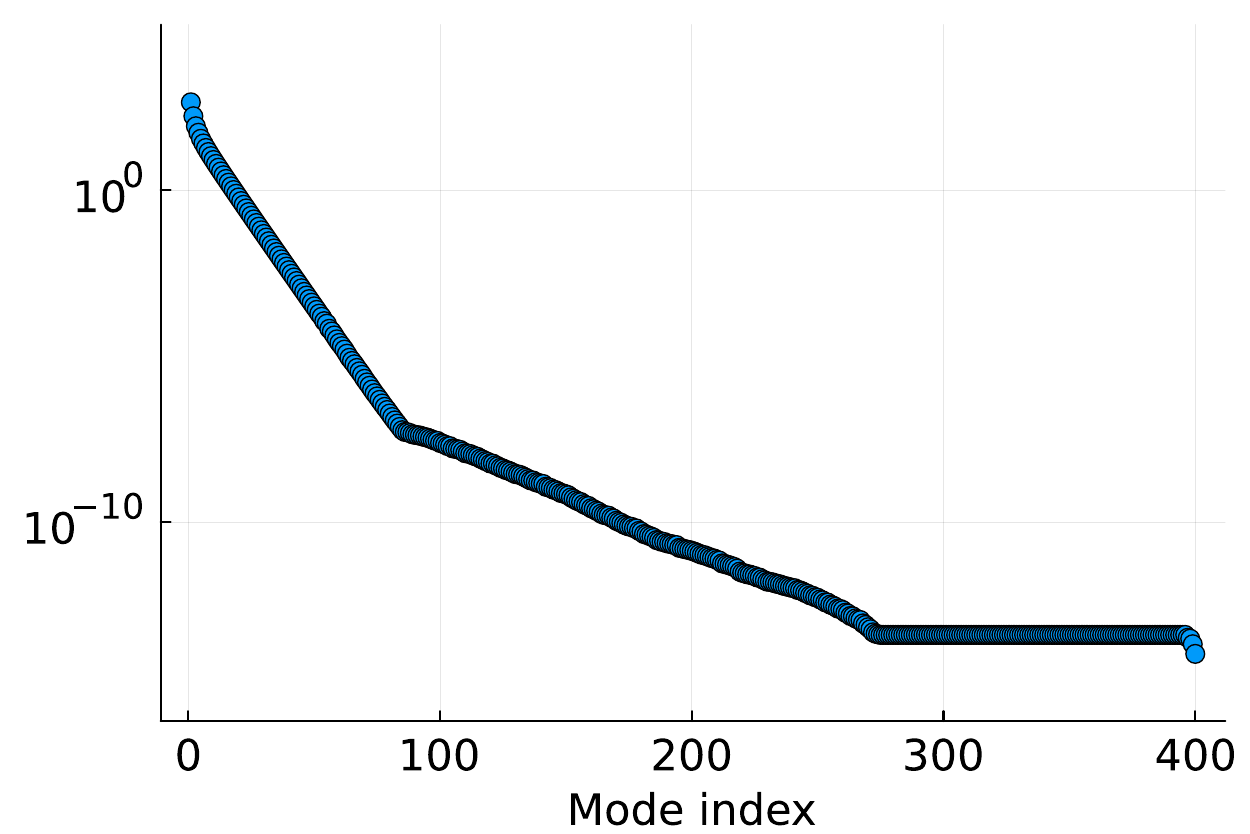}}
\subfloat[$p=7$]{\includegraphics[width=.33\textwidth]{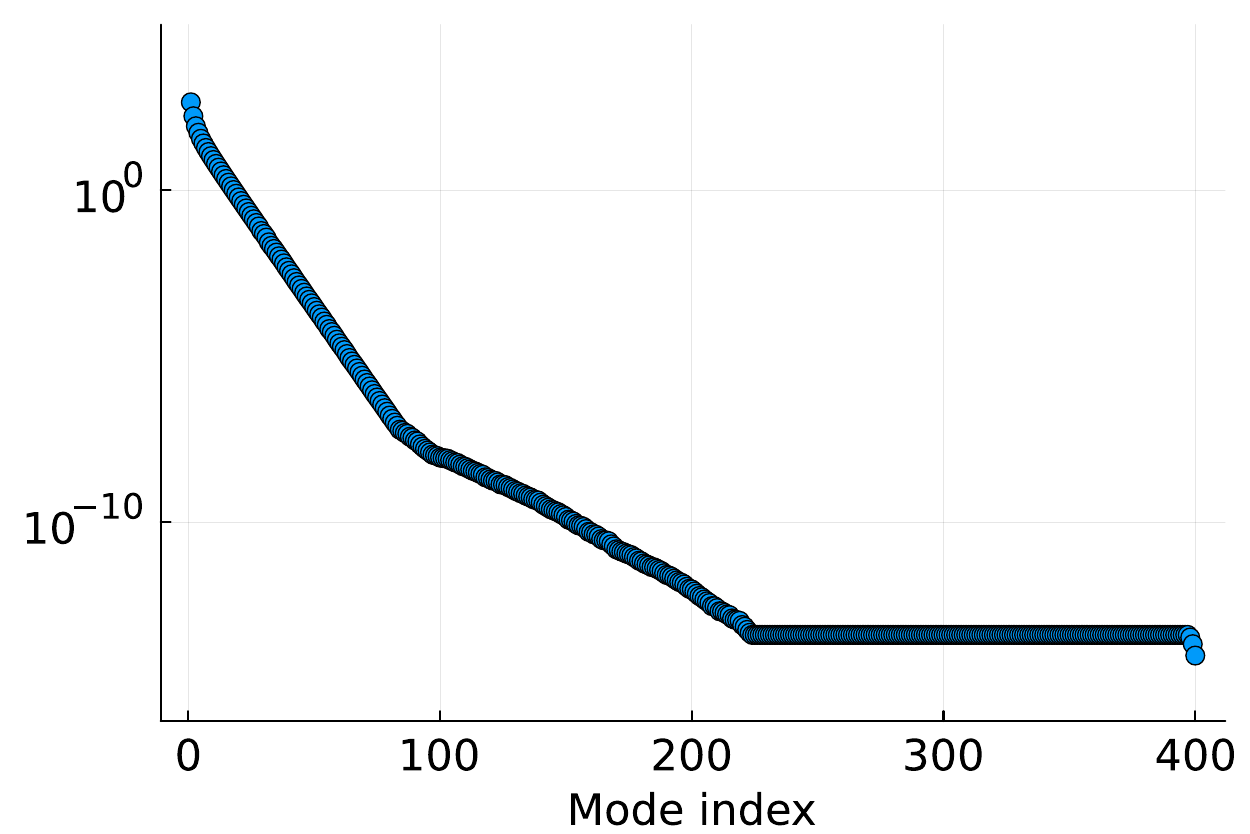}}
\caption{Singular values of $\bm{V}_\text{snap}$ in Burgers' equation}
\label{fig_sval3}
\end{figure}

\begin{table}
\begin{center}
\begin{tabular}{ | l | c | c | c | c | c | c |}
\hline
   Error/nodes& $p=0$ & $p=3$ & $p=7$\\ \hline
   $N=30$ & 1.38e-3/ 68 & 4.35e-4/ 125 & 6.18e-4/ 69  \\   \hline
  $N=40$ & 3.35e-5/ 94  & 6.93e-5/ 214  & 5.16e-5/ 145 \\   \hline
  $N=50$ & 2.43e-5/ 128 &  8.64e-6/ 283&  1.81e-5/ 271\\ \hline
\end{tabular}
\caption{Error and number of hyper-reduced nodes in Burgers' equation}
\label{tb:Burgers1}
\end{center}
\end{table}
\par We observe that the ROM errors are of similar magnitude for all values of $p$ tested. However, as with the advection equation, DG methods yield more hyper-reduced nodes in general.

\subsubsection{Compressible Euler equations}\label{sssec:compressible-Euler}
The compressible Euler equations for $d$ dimensions are given by
\begin{equation}
    \begin{split}
        \f{\p\rho}{\p\rmt} + &\sumjd \f{\p u_j}{\p \bm{x}_j} = \bm{0} \\
       \f{\p\rho u_i}{\p\rmt} + &\sumjd \f{\p (\rho u_i u_j + p\delta_{ij})}{\p \bm{x}_j} = \bm{0}\csp i=1,...,d \\
       \f{\p E}{\p \rmt} + &\sumjd\f{\p (u_j(E+p))}{\p \bm{x}_j} = \bm{0},
    \end{split}
    \label{eq:num_CEuler}
\end{equation}
where $\rho$ is density, $u_i$ is the $i$-th component of velocity, and $E$ is the total energy. The pressure $p$ and internal energy $\rho e$ are given by
$$p=(\gamma-1)\LRp{E-\f{1}{2}\rho|\bm{u}|^2}\csp \rho e = E-\f{1}{2}\rho|\bm{u}|^2\csp |\bm{u}|^2=\sumjd u_j^2.$$
The compressible Euler equations correspond to the compressible Navier-Stokes equations with zero thermal conductivity. There is a unique entropy $S(\bm{u})$ symmetrizing the viscous heat conduction term in Navier-Stokes equation \cite{Hughes86}
$$S(\bm{u}) = -\rho s,$$
where $s=\log(p/\rho^\gamma)$ is the physical specific entropy, and we set $\gamma=1.4$ for all numerical experiments. The entropy variables in $d$ dimensions are given by 
$$v_1=\f{\rho e(\gamma+1-s)-E}{\rho e}\csp v_{d+2}=-\f{\rho}{\rho e}\csp v_{1+i}=\f{\rho u_i}{\rho e}\csp i=1,...,d,$$
and the corresponding conservative variables in terms of the entropy variables are 
$$\rho = -(\rho e)v_{d+2}\csp E=(\rho e)\LRp{1-\f{\sumjd v_{1+j}^2}{2v_{d+2}}}\csp \rho u_i=(\rho e)v_{1+i}\csp i=1,...,d,$$
where
$$\rho e= \LRp{\f{\gamma-1}{(-v_{d+2})^\gamma}}^{1/(\gamma-1)} e^{-s/(\gamma-1)} \csp s = \gamma-v_1+\f{\sumjd v_{1+j}^2}{2v_{d+2}}.$$
Explicit entropy conservative flux expressions for the compressible Euler equation are given in \cite{Ranocha18, Ranocha20, Ranocha22}.



We test both DG and FVM ROMs for an isentropic Gaussian wave initial condition such that 
$$\rho = 1+0.1e^{-25x^2}\csp u=0.1\sin(\pi x)\csp p=\rho^\gamma.$$ The artificial viscosity coefficient is set to $\epsilon = 5\times10^{-4}$, and we run until final time $T=1.0$. \autoref{fig_sval5} shows the singular values of the snapshot matrix. In \autoref{tb:Euler4}, we observe the same trend that DG methods at $p=3$ require more hyper-reduced nodes for similar accuracy.

\begin{figure}
\centering
\noindent
\subfloat[$p=0$ (FVM)]{\includegraphics[width=.33\textwidth]{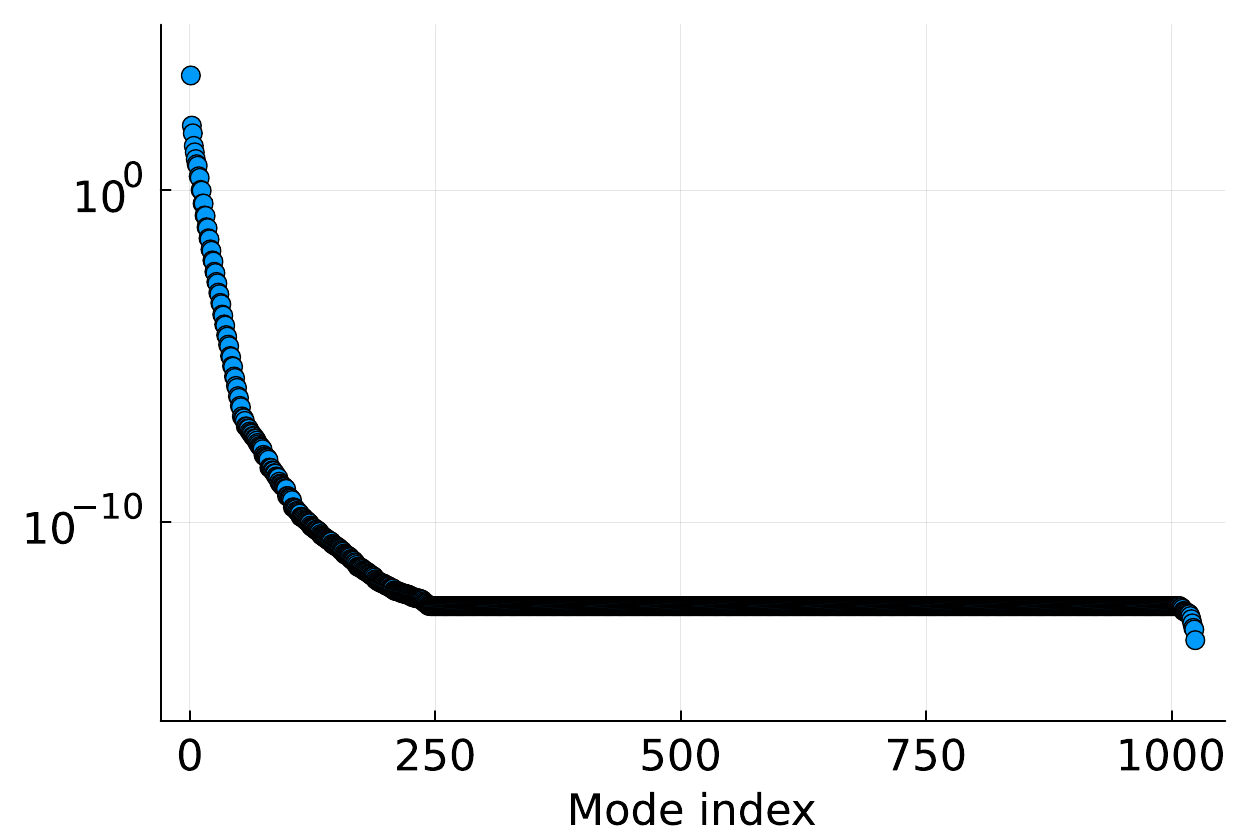}}
\subfloat[$p=3$]{\includegraphics[width=.33\textwidth]{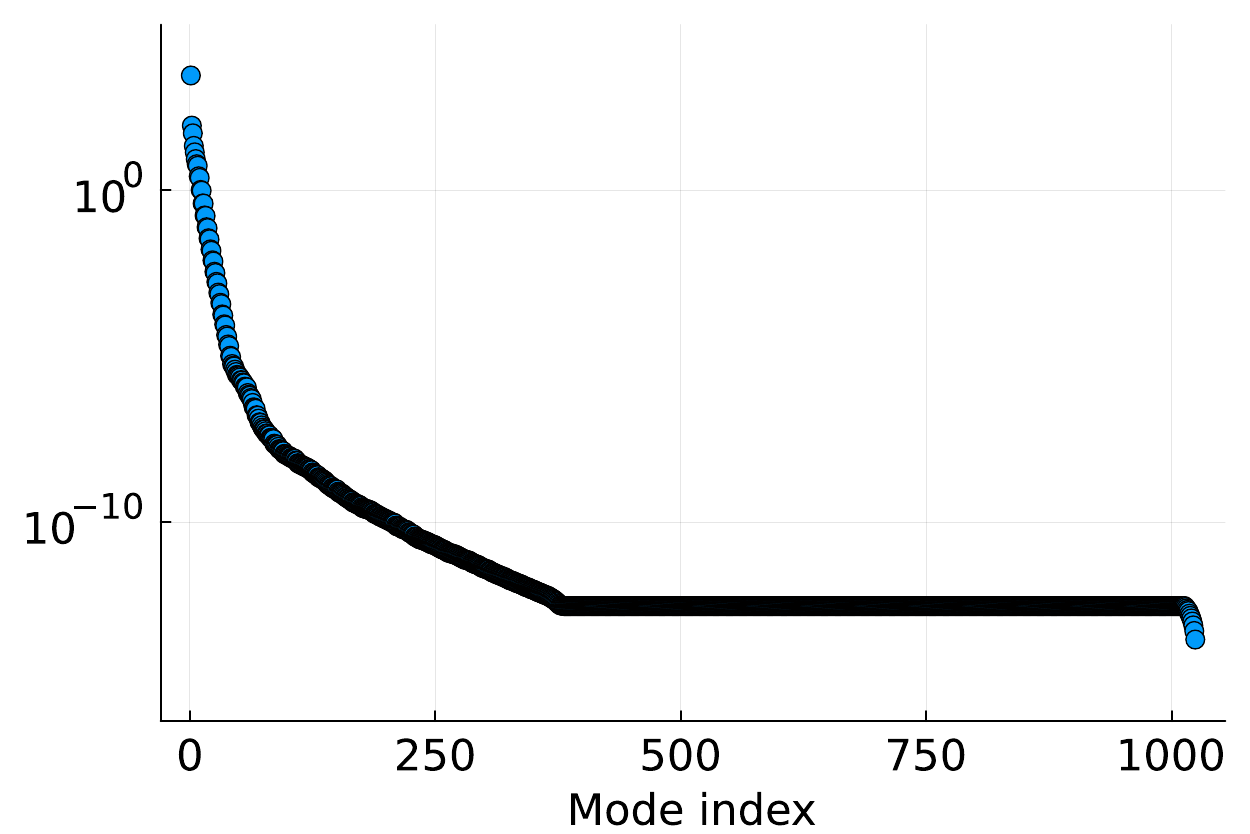}}
\subfloat[$p=7$]{\includegraphics[width=.33\textwidth]{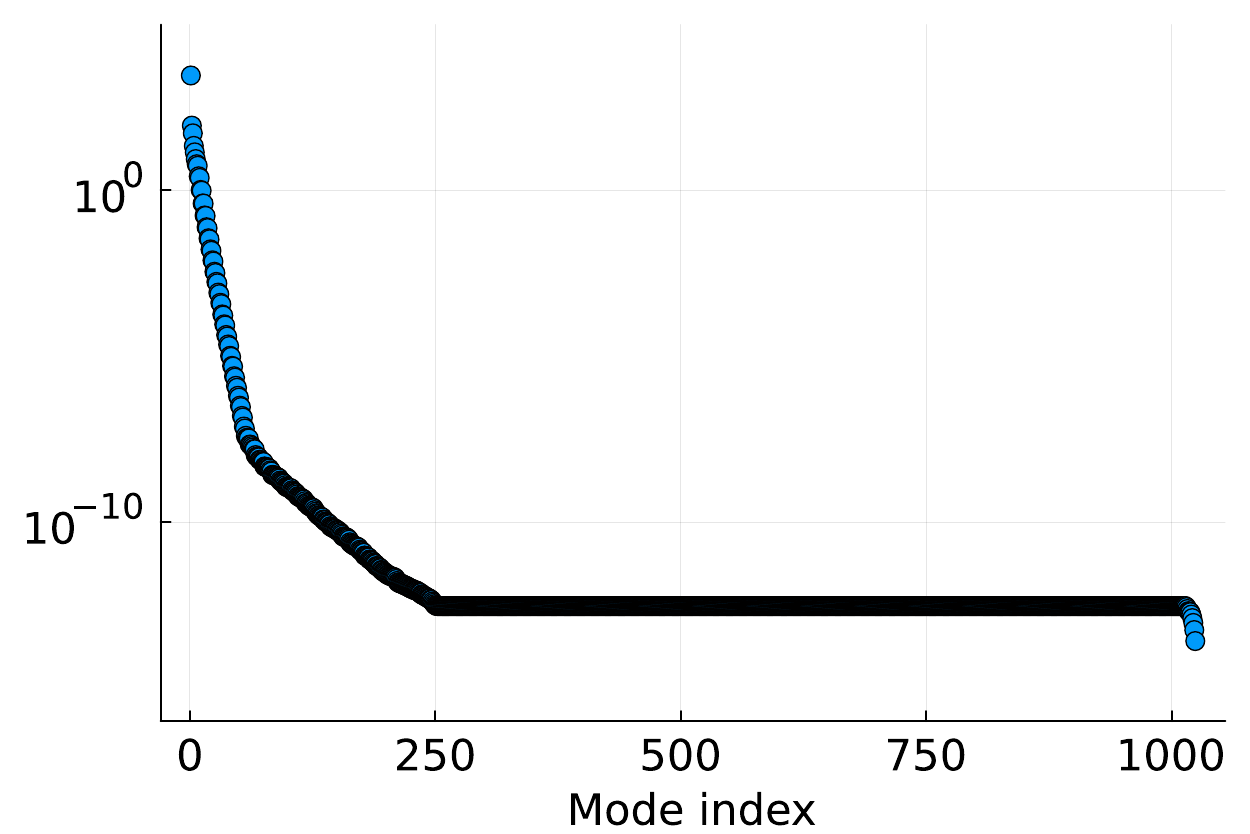}}
\caption{Singular values of $\bm{V}_\text{snap}$ in 1D compressible Euler equations}
\label{fig_sval5}
\end{figure}

\begin{table}
\begin{center}
\begin{tabular}{ | l | c | c | c | c | c | c |}
\hline
   Error/nodes& $p=0$ & $p=3$ & $p=7$\\ \hline
   $N=20$ & 1.00e-4/ 53  & 1.00e-4/ 77  & 1.01e-4/ 53  \\  \hline
  $N=30$ &  4.98-6/ 93 &  5.77e-6/ 174& 5.10e-6/ 80 \\   \hline
  $N=40$ &  1.09e-7/ 285 & 1.14e-7/ 415 & 1.10-7/ 281 \\   \hline
\end{tabular}
\caption{ Error and number of hyper-reduced nodes in 1D compressible Euler equations}
\label{tb:Euler4}
\end{center}
\end{table}

In conclusion, for all equations tested, ROMs based on DG or FVM FOMs yield similar errors, though DG ROMs appear to yield a larger number of hyper-reduced nodes compared to FVM ROMs.

\subsection{1D experiments for DG ROMs}
\subsubsection{Reflective wall boundary condition}
\label{sec:1DEuler_weak}
We continue with the 1D Euler equations on a non-periodic domain $[0,1]$. We implement reflective wall boundary conditions using a ``mirror state" \cite{Chen17, Svard14} at domain boundaries. This results in an evaluation of $\bm{f}^* = \bm{f}_{EC}(\bm{u}^+,\bm{u})$ at boundary nodes with exterior state $\bm{u}^+$ defined as
$$\rho^+=\rho\csp u^+=-u\csp p^+ = p.$$
The initial condition is set as 
$$\rho = 2+0.5e^{-100(x-0.5)^2}\csp u=0.1e^{-100(x-0.5)^2}\csp p=\rho^\gamma,$$
and the system exhibits a shock some time after $T=0.25$ \cite{Chan20ROM}.

 For the FOM, we employ 512 elements, each with an interpolation degree of 3. The total dimension of the FOM is thus 2048. The artificial viscosity coefficient is set to $\epsilon = 2 \times 10^{-4}$. We run the FOM until $T=0.75$.

\begin{figure}
\centering
\noindent
\subfloat[(a) Singular values of snapshots]{\includegraphics[width=.49\textwidth]{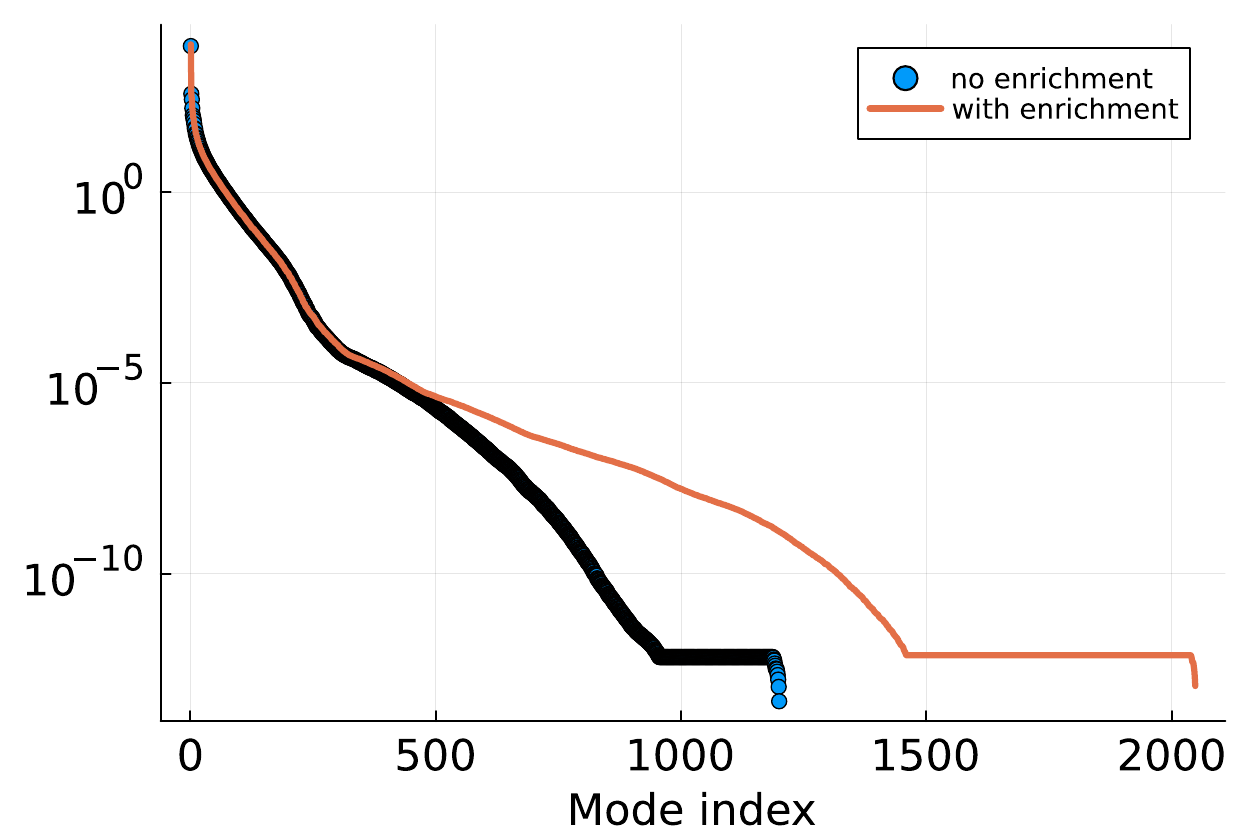}
\label{fig:DG1_sval}}
\subfloat[(b) Entropy projection error ($N=20$)]{\includegraphics[width=.49\textwidth]{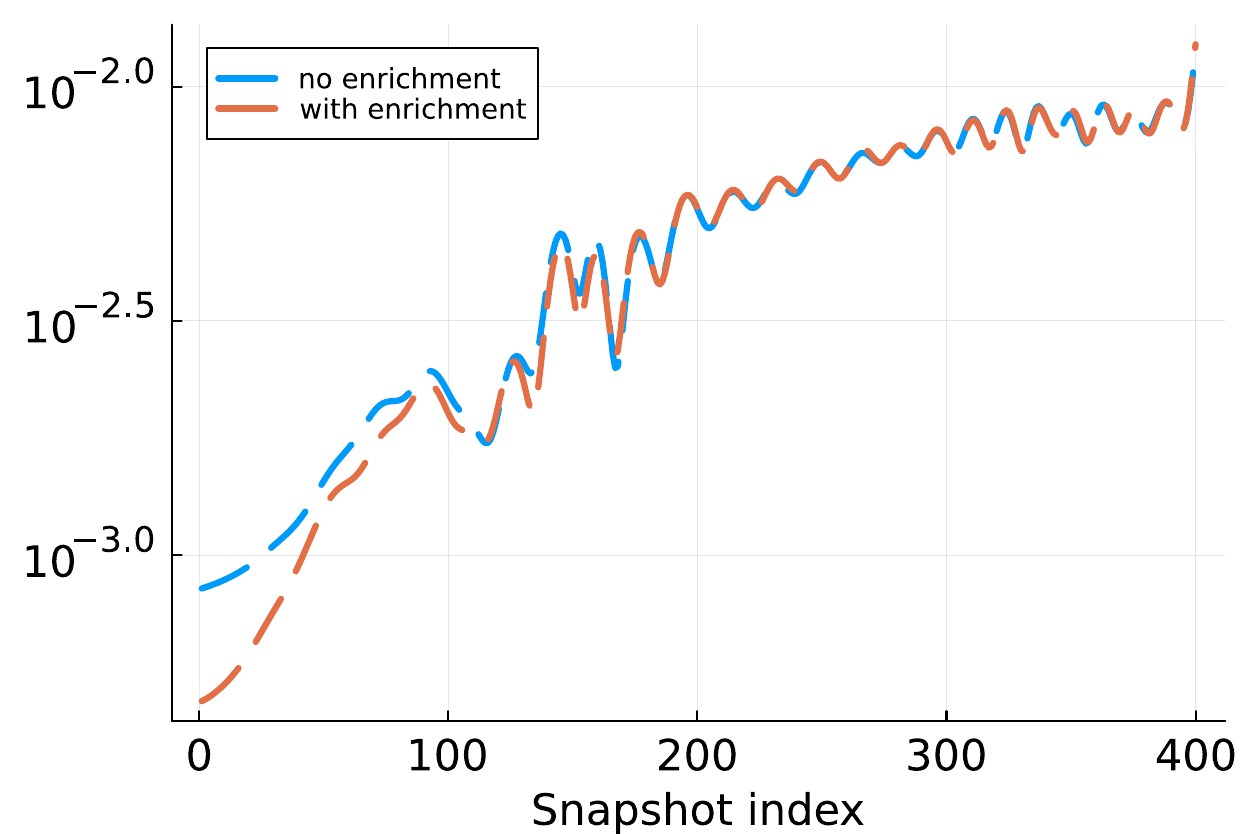}
\label{fig:DG1_proj_er}}

\caption{Singular values and entropy projection error with and without entropy enrichment
}
\end{figure}

\begin{figure}
\centering
\noindent
\subfloat[(a) ROM error (with and without hyper-reduction) and their difference]{\includegraphics[width=.49\textwidth]{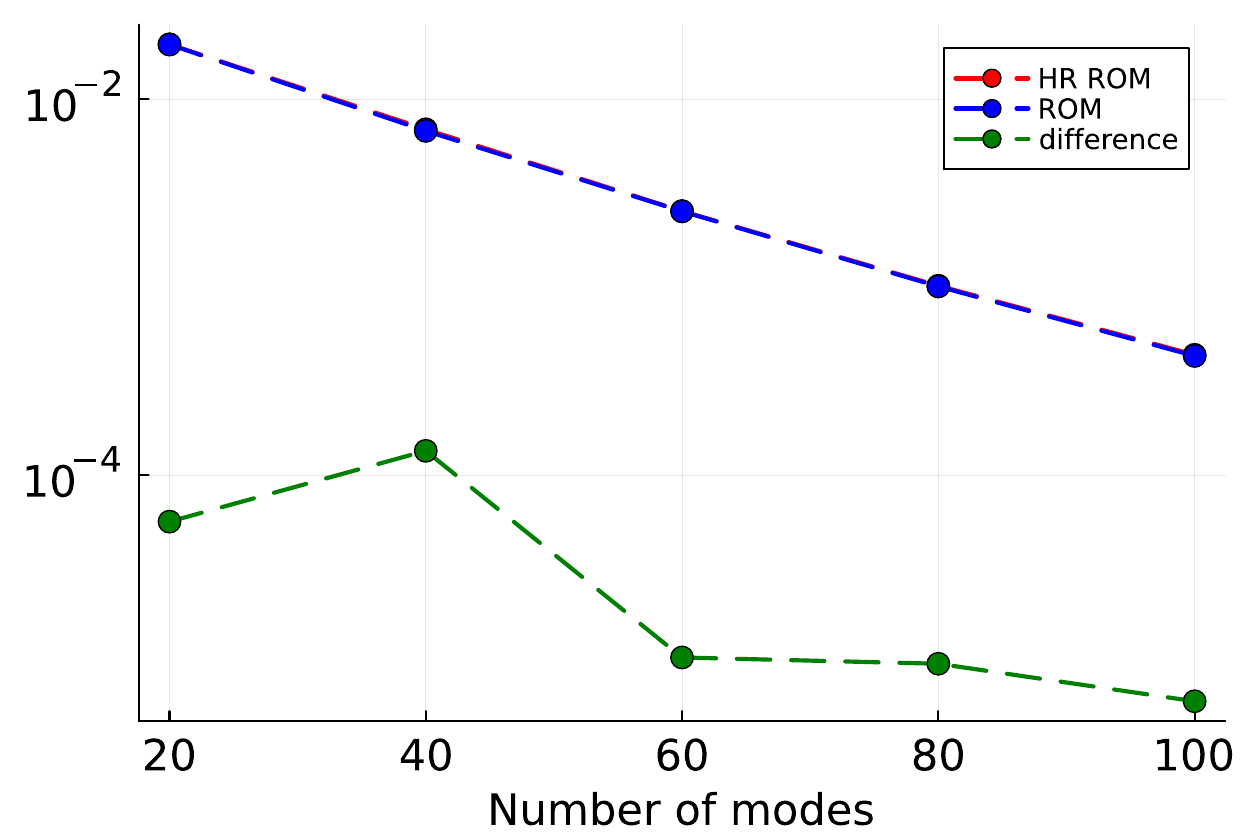}}
\subfloat[(b) Number of hyper-reduced nodes]{\includegraphics[width=.49\textwidth]{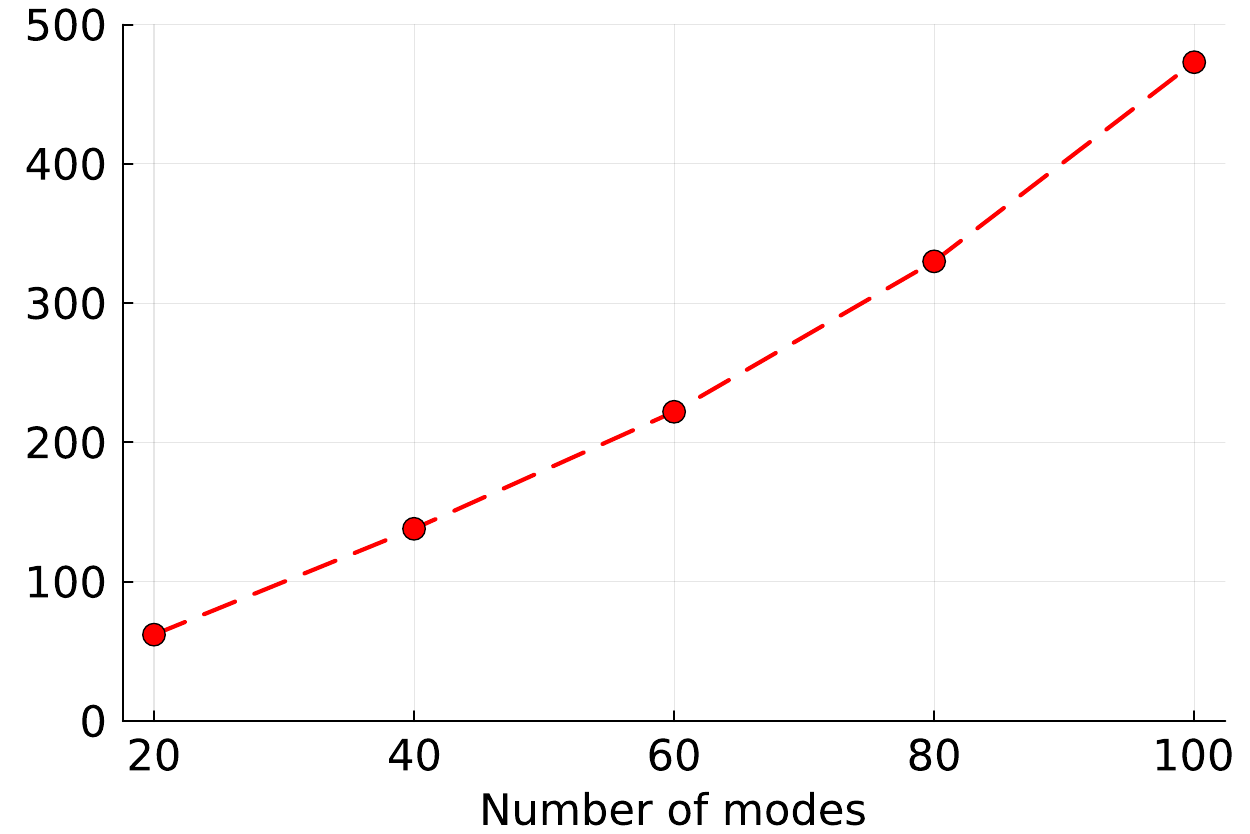}}
\caption{ROM and HR-ROM performance (error and number of HR nodes).}
\label{fig1.2}
\end{figure}

\begin{figure}
\centering
\noindent
\subfloat[(a) $N=20$, $T=.25$]{\includegraphics[width=.49\textwidth]{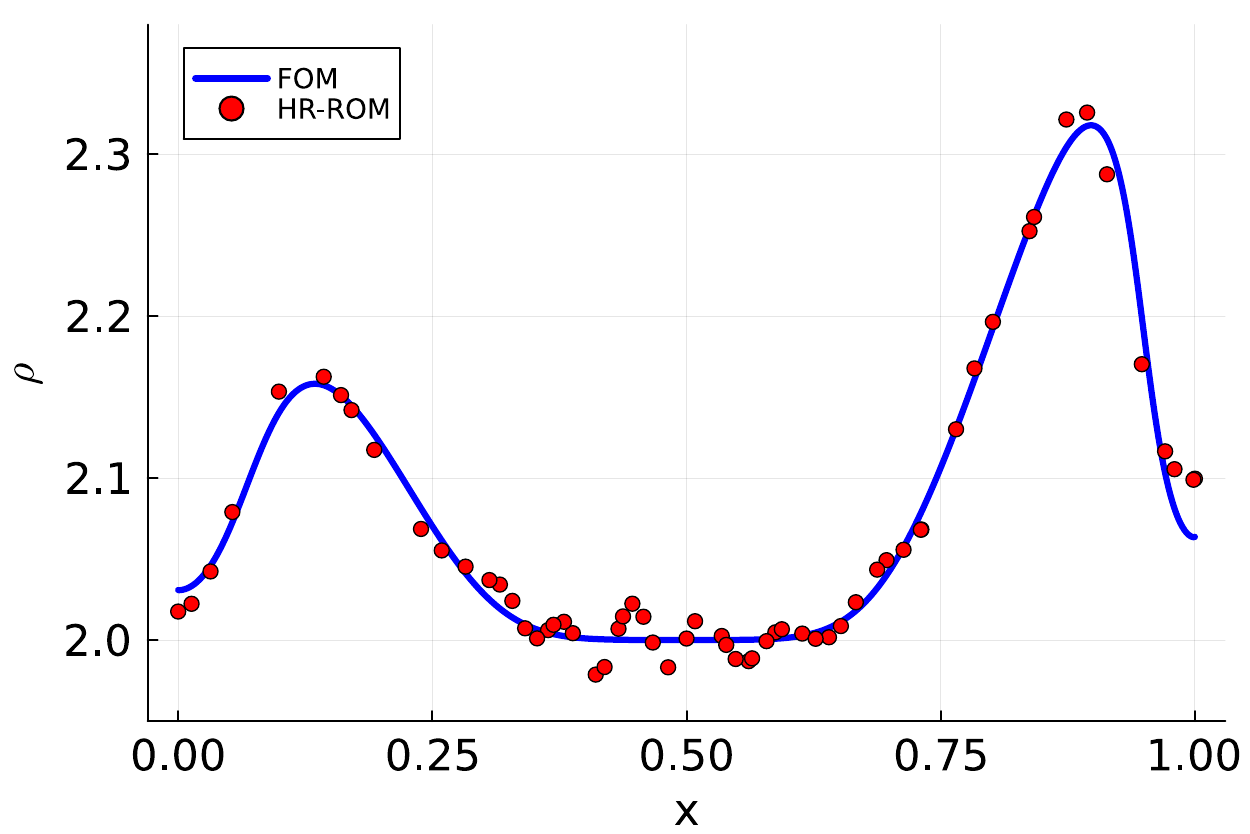}}
\subfloat[(b) $N=60$, $T=.25$]{\includegraphics[width=.49\textwidth]{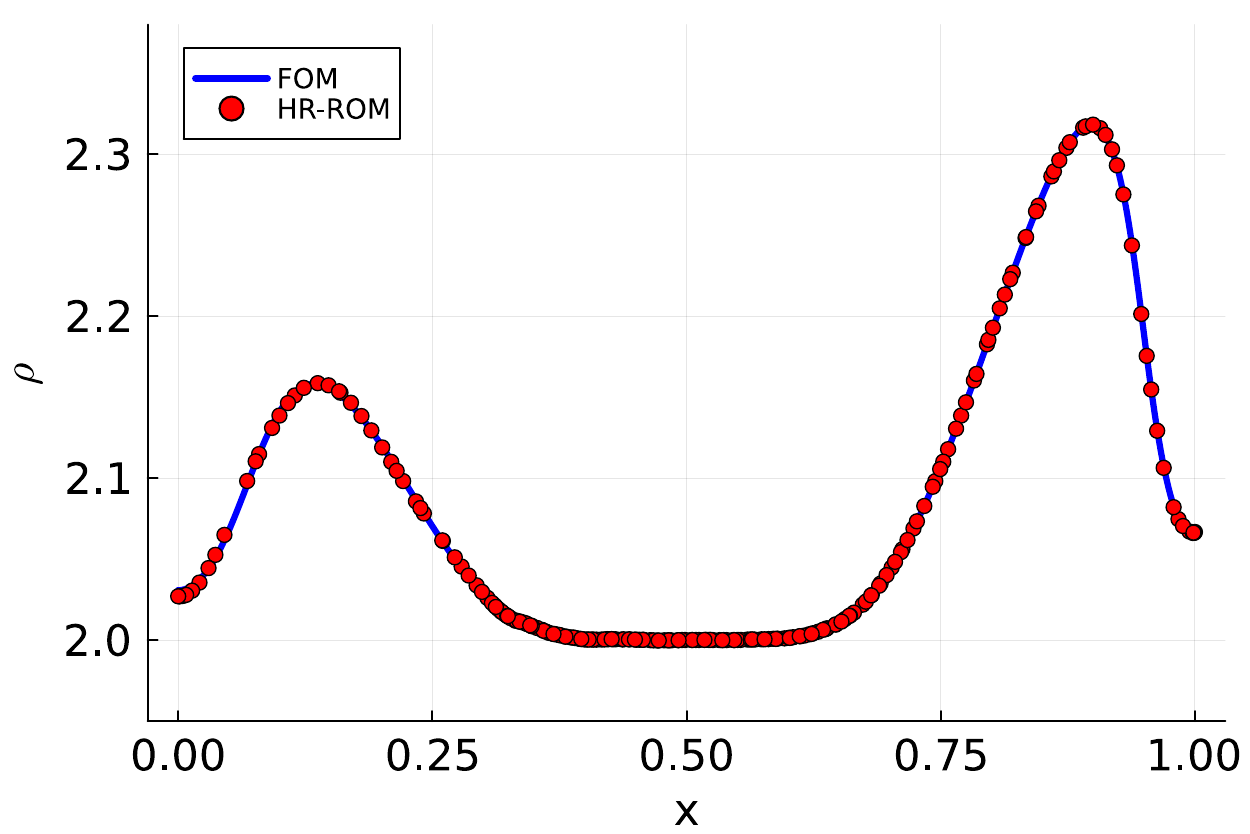}}\\
\subfloat[(c) $N=40$, $T=.75$]{\includegraphics[width=.49\textwidth]{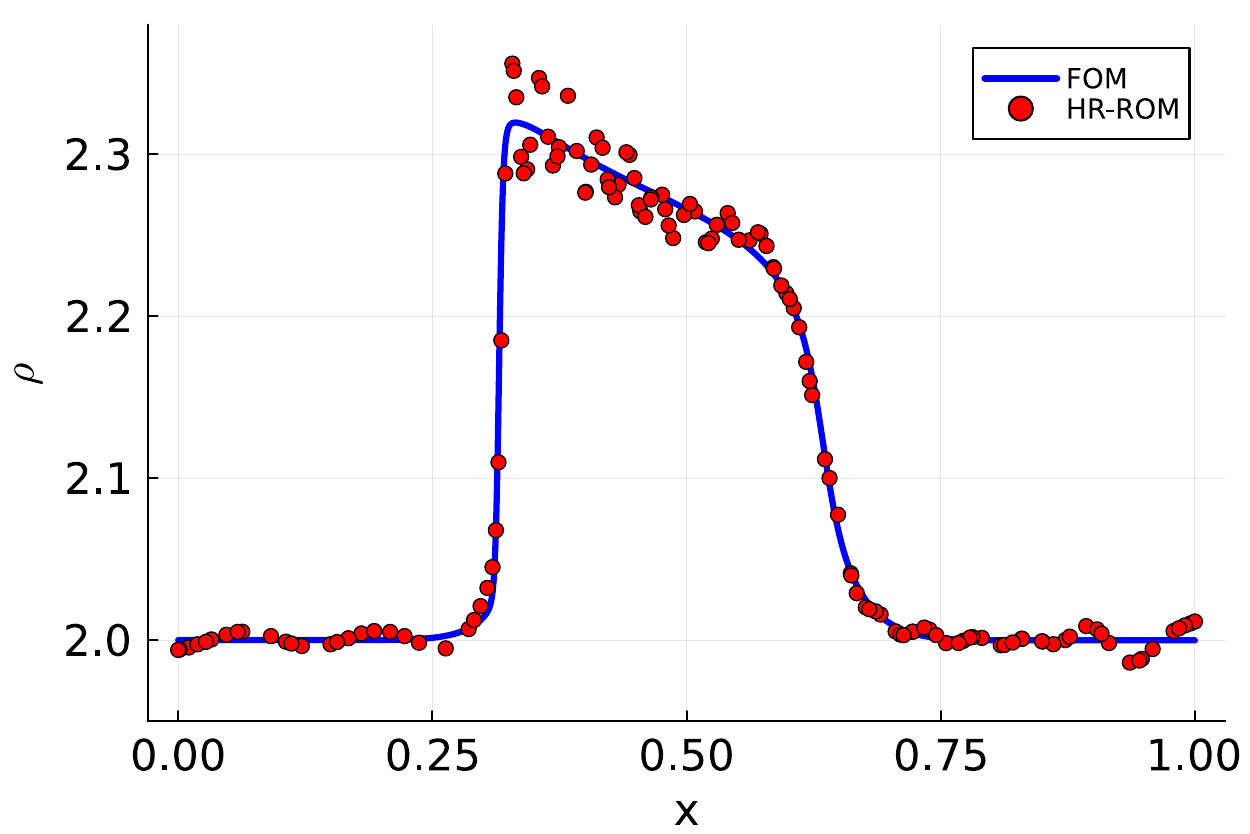}}
\subfloat[(d) $N=100$, $T=.75$]{\includegraphics[width=.49\textwidth]{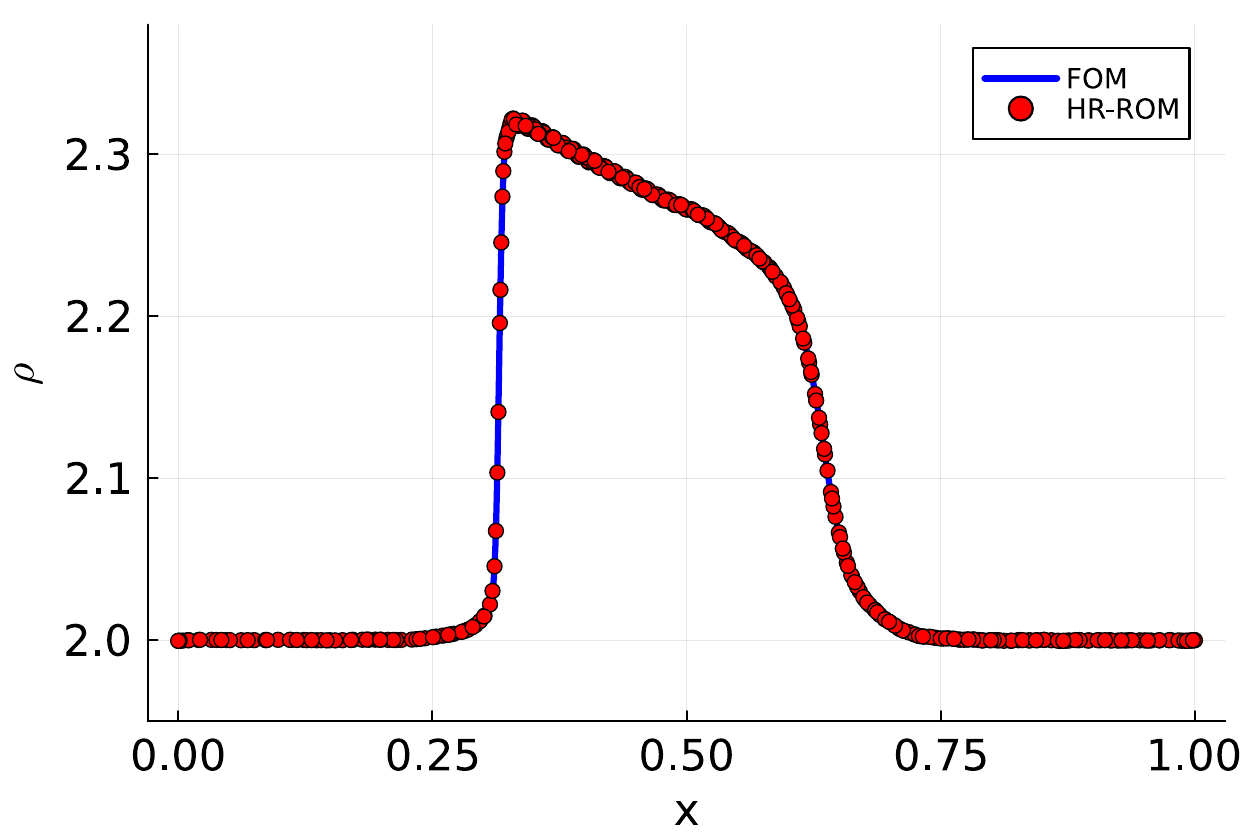}}
\caption{Density $\rho$ plots. Blue line plots display the FOM solutions, while red markers indicate the ROM solutions on hyper-reduced nodes.
}
\label{fig1.3}
\end{figure}

\begin{figure}
\centering
\noindent
\subfloat[(a) Convective entropy for $N=20$ HR-ROM]{\includegraphics[width=.49\textwidth]{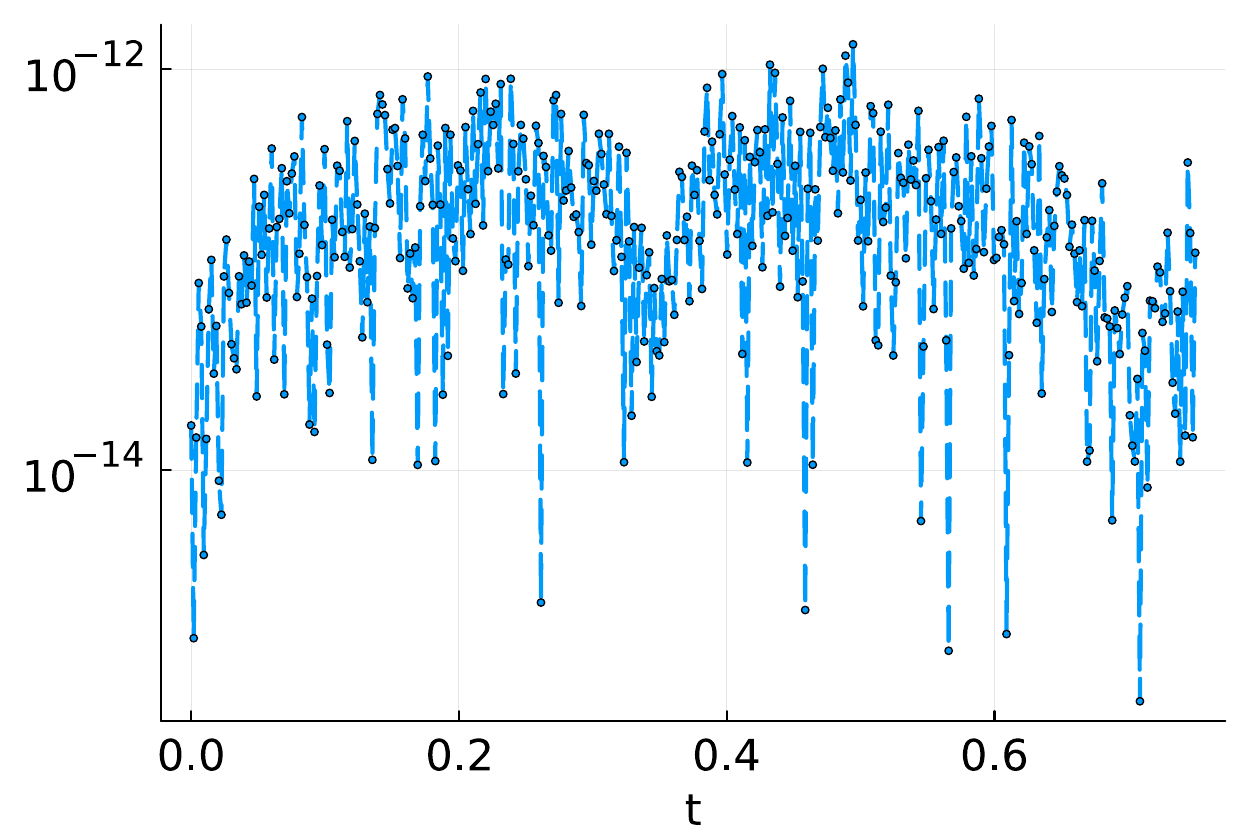}
\label{fig:DG1_conv_entro}}
\subfloat[(b) Entropy dissipation from viscosity for HR-ROMs]{\includegraphics[width=.49\textwidth]{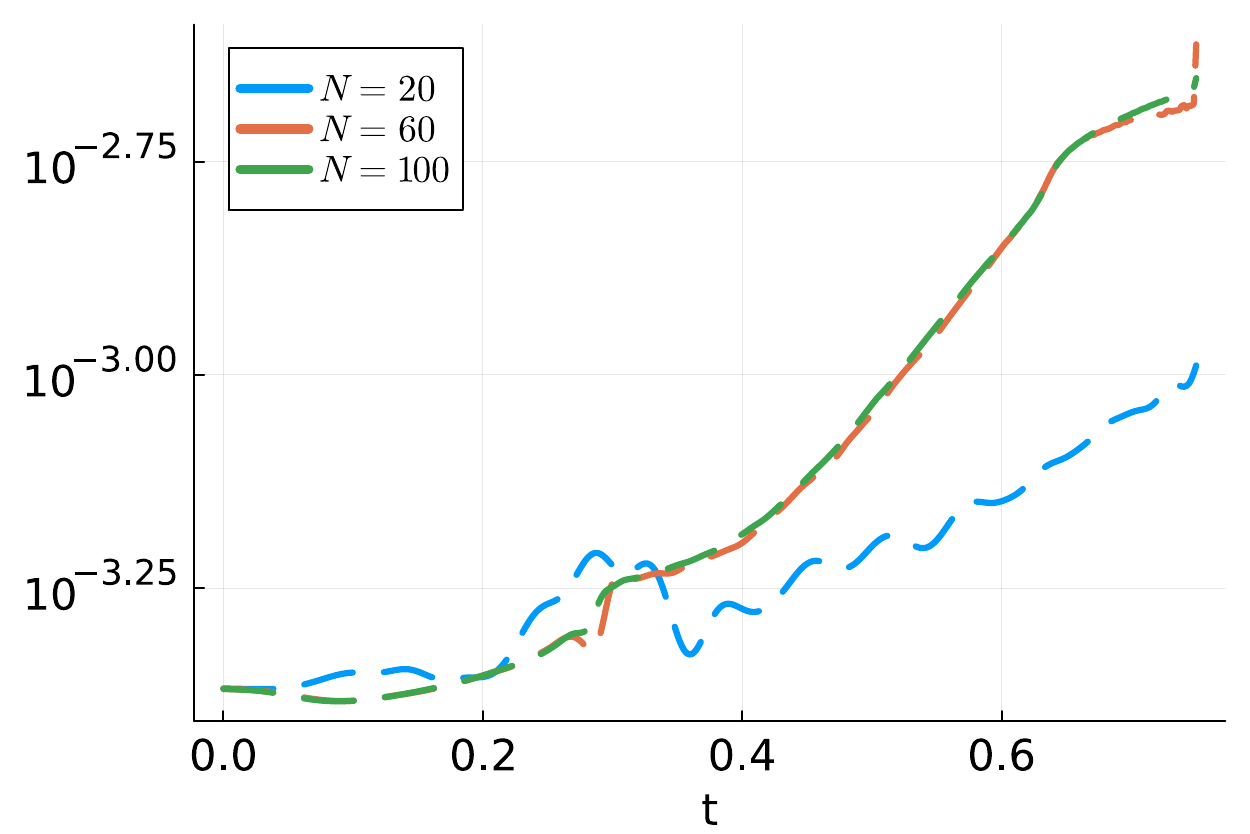}
\label{fig:DG1_entro_diss}}
\caption{Convective entropy and entropy dissipation for HR-ROMs.}
\end{figure}

\begin{figure}
\centering
\noindent
\subfloat[(a) Relative difference of solutions]{\includegraphics[width=.49\textwidth]{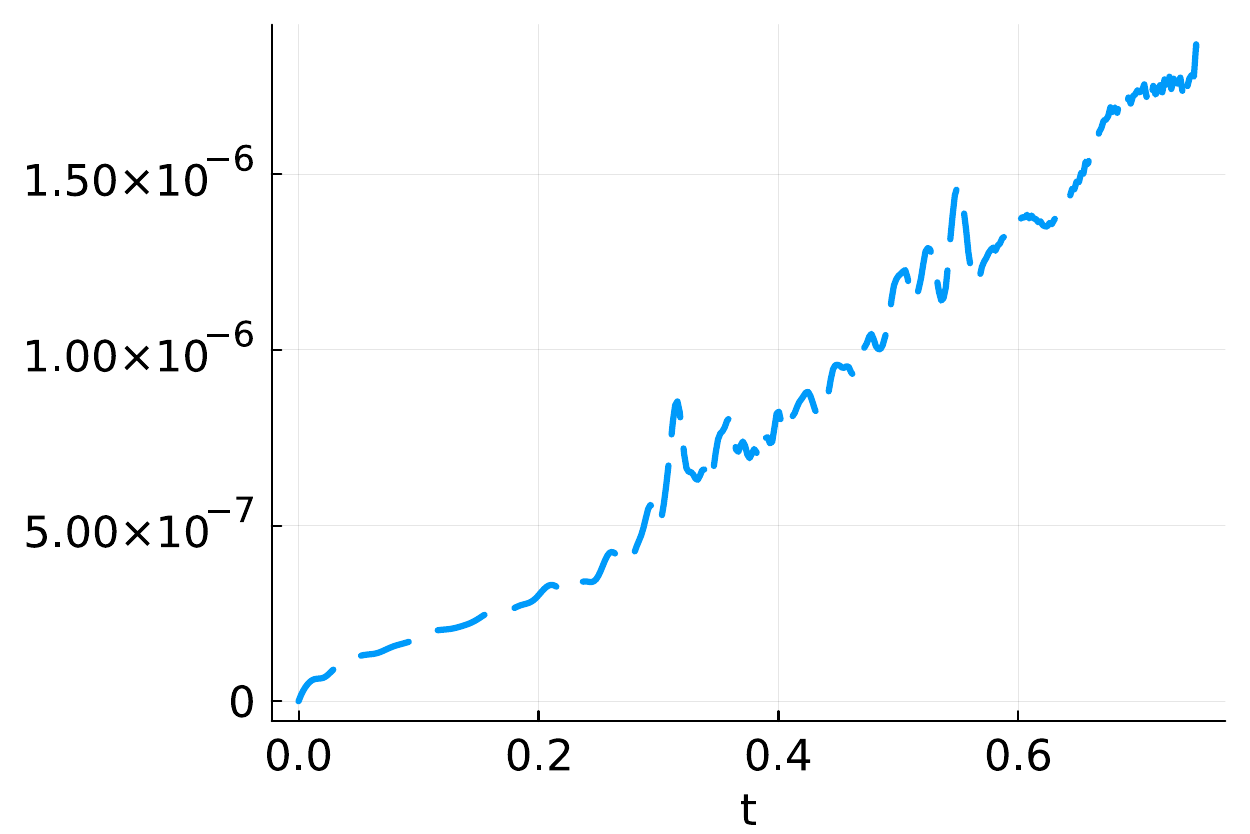}
\label{fig:DG1_sol_diff}}
\subfloat[(b) Absolute difference of entropy dissipation from viscosity]{\includegraphics[width=.49\textwidth]{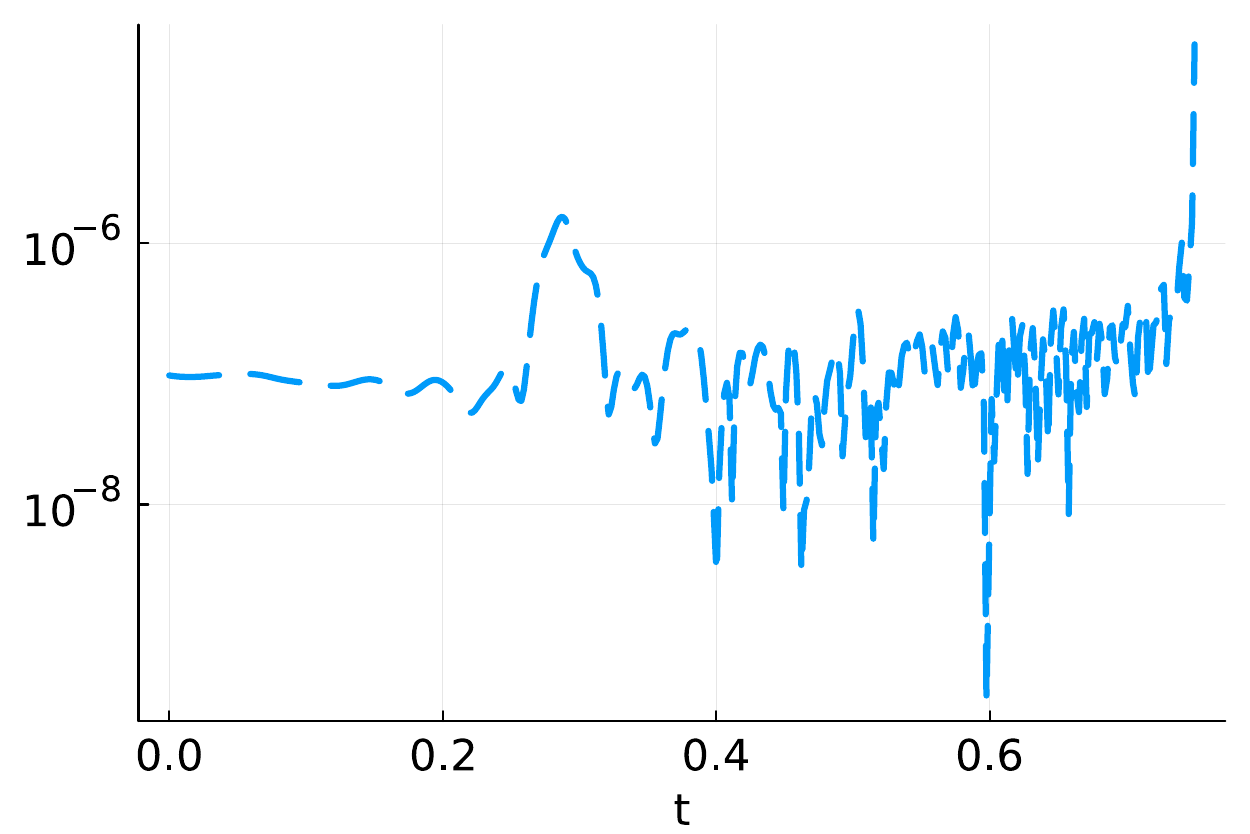}
\label{fig:DG1_dissipation_diff}}
\caption{Relative difference of solution and absolute difference of entropy dissipation from viscosity for $N=100$ HR-ROM using entropy stable and BR-1 viscosity discretization.}
\end{figure}

Initially, we explore the impact of incorporating entropy variables into the snapshots. Fig.~\ref{fig:DG1_sval} displays the singular values computed from the snapshot matrices $\bm{V}_\text{snap}$ with and without entropy variable enrichment. We observe that including snapshots of entropy variables results in a slower decay rate of the singular values, especially at higher modes. In Fig.~\ref{fig:DG1_proj_er}, we assess the relative error between snapshots of the FOM solution $\bm{u}_\text{FOM}$ and the \textit{entropy projected} solution $\bm{u}\LRp{\bm{V}_N\bm{P}_N\bm{v}\LRp{\bm{V}_N\bm{u}_N}}$ using 20 modes. We define the energy residual of the first $i$ modes as $$E_i = \sqrt{\sum_{j=i+1}^{N_t}\mu_j^2 / \sum_{j=1}^{N_t}\mu_j^2}.$$ Although singular values of enriched snapshots decay slower, the energy residuals $E_{20}$ are similar ($5.80\times 10^{-3}$ without enrichment and $5.25\times 10^{-3}$ with enrichment). We observe that incorporating entropy variables leads to slightly lower errors for this 1D example.

\par Fig.~\ref{fig1.2} illustrates the relative $L^2$ error and the difference in error for conservative variables between FOM and ROM solutions, alongside the count of hyper-reduced nodes, for mode counts from 20 to 100 in 20-mode increments. In this example, the ROM solutions with or without hyper-reduction have very similar performance, suggesting that the hyper-reduction step does not introduce significant errors.

\par Next, we take a look at the actual solutions by plotting the conservative variable, density $\rho$, in Fig.~\ref{fig1.3}. At $T=0.25$, the system has not yet developed any shocks. With 20 modes, the ROM exhibits some oscillations around the smooth FOM solution, but these oscillations diminish when we increase the number of modes to 60. At $T=0.75$, a shock develops, causing the solution with 40 modes to oscillate due to the discontinuity. However, these oscillations disappear when we further increase the number of modes to 100. Even with a small number of modes and oscillatory solution approximations, the solution remains stable. 

\par We then numerically assess entropy stability. In Fig.~\ref{fig:DG1_conv_entro}, we show for a ROM with 20 modes the convective entropy contribution 
\begin{equation*}
\LRb{\bm{v}_N^T\LRp{\widebar{\bm{V}}_h^T\LRp{\LRp{\widebar{\bm{Q}}_h-\widebar{\bm{Q}}_h^T}\circ\bm{F}}\bm{1} + \bm{V}_b^T\bm{B}_b\bm{f}_b^*}},
\end{equation*}
where the variables are defined in \eqref{eq:DG_ROM_HR_weakBC1}. \autoref{thm:weakBC_stability} implies this quantity should be exactly zero to ensure entropy conservation. In practice, its magnitude is close to machine precision. In Fig.~\ref{fig:DG1_entro_diss}, we illustrate the entropy dissipation from the viscous term \begin{equation*}
\epsilon\bm{v}_N^T\bm{K}_N\bm{u}_N,
\end{equation*}
with different numbers of modes. The entropy dissipation resulting from the artificial viscosity consistently remains positive over simulation time, confirming that entropy dissipation is observed.

Finally, we implement the provably entropy stable viscosity discretization \eqref{eq:ES_visc} for the same example. We fix $N=100$, run to final time $T=.75,$ and plot the relative difference of solutions (using previous BR-1 solution as a reference) in \autoref{fig:DG1_sol_diff} together with the absolute difference of entropy dissipation induced by the viscosity in \autoref{fig:DG1_dissipation_diff}. We observe that the entropy stable discretization exhibits no significant differences in terms of HR-ROM accuracy and entropy dissipation from the viscosity.

\subsubsection{Sod shock tube}
\label{sec:sod}
We now consider another example: the Sod shock tube \cite{Sod78}. The domain is now $[-\f{1}{2},\f{1}{2}]$ with the exact solutions imposed at boundaries, and the initial condition is set as
\begin{equation*}
    \rho = 0.125 + \f{0.875}{1+e^{100x}}\csp u = 0\csp p = 0.1 + \f{0.9}{1+e^{100x}},
\end{equation*}
which corresponds to a smoothing of the standard discontinuous initial condition. For the FOM, we still employ 512 elements with an interpolation degree of 3. The artificial viscosity coefficient is set to $5 \times 10^{-4}$. We execute the model until $T=0.25$.

\begin{remark}
The solution still remains stable if we do not smooth the initial condition. However, we observe that when using a strictly discontinuous initial condition, the largest eigenvalue of the reduced diffusion matrix \eqref{eq:DG_ROM_ES_visc} is significantly larger than for the smoothed initial condition case. As a result, not smoothing the initial condition results requires a much smaller time-step for stability under explicit time-stepping. 
\end{remark}

\begin{figure}
\centering
\noindent
\subfloat[(a) Singular values of snapshots]{\includegraphics[width=.49\textwidth]{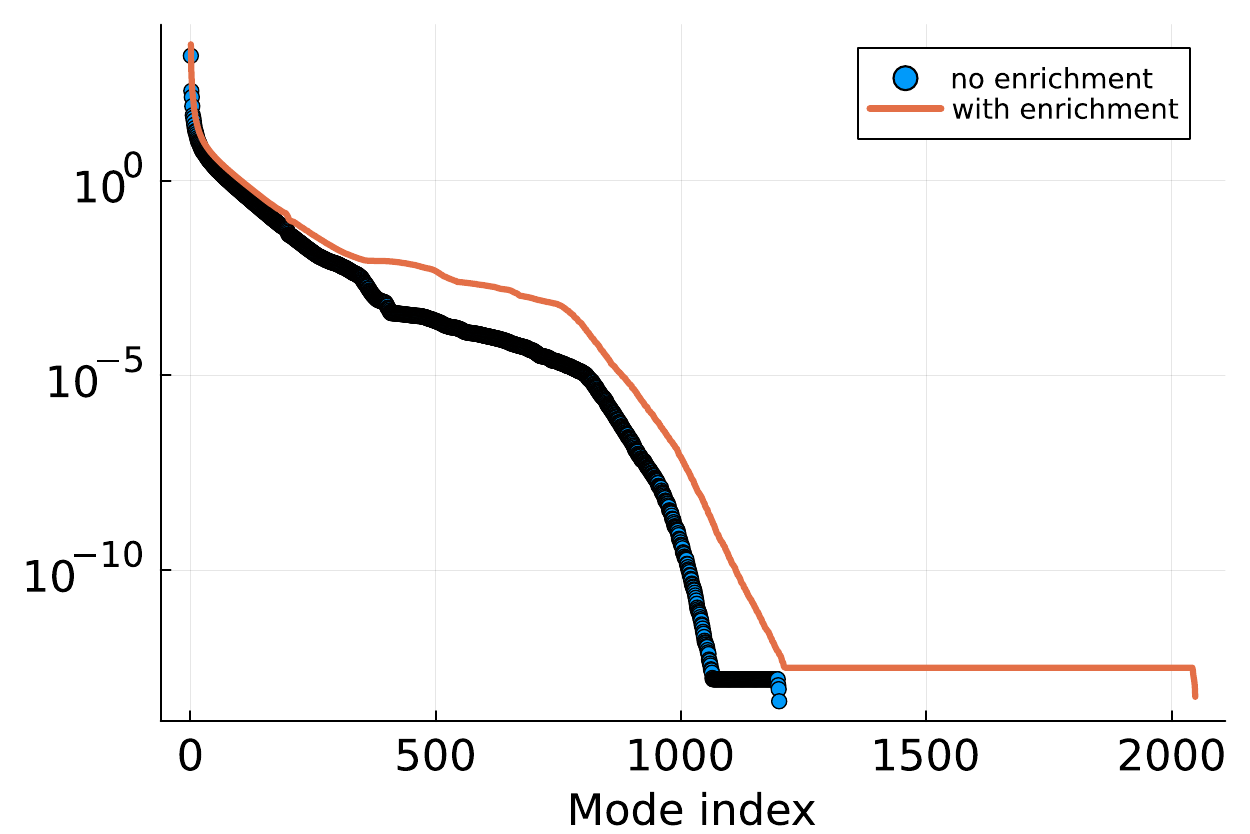}
\label{fig:DG2_sval}}
\subfloat[(b) Entropy projection error ($N=30$)]{\includegraphics[width=.49\textwidth]{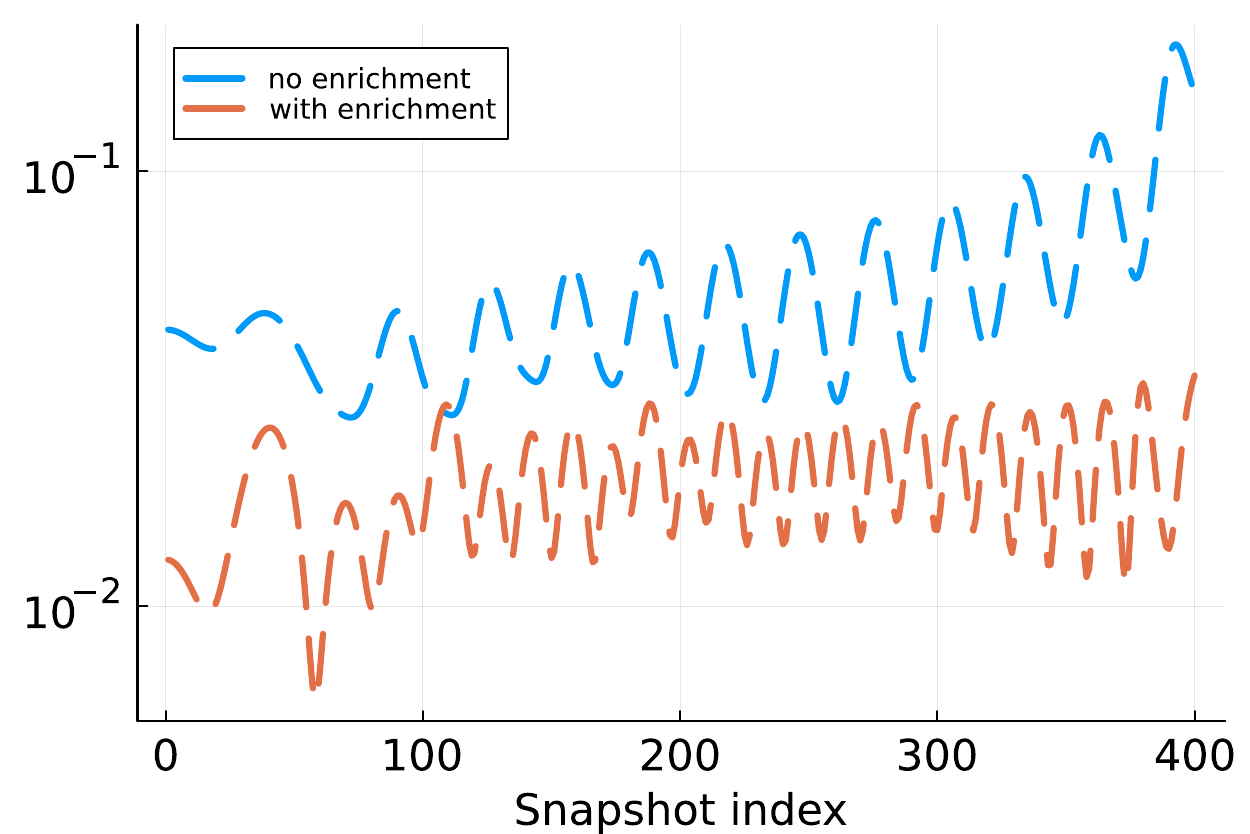}
\label{fig:DG2_proj_er}}
\caption{Singular values and entropy projection error with and without entropy enrichment.}
\end{figure}

\begin{figure}
\centering
\noindent
\subfloat[(a) ROM error (with and without hyper-reduction) and their difference]{\includegraphics[width=.49\textwidth]{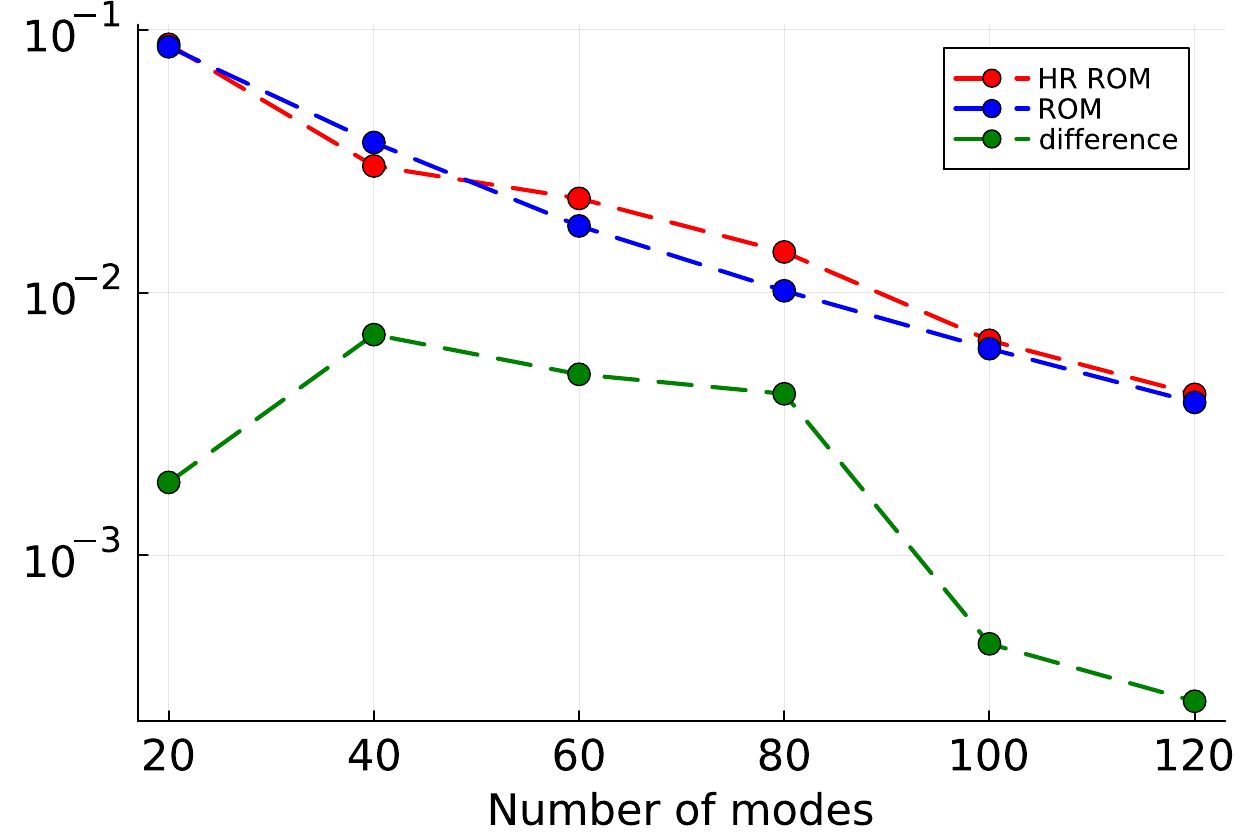}}
\subfloat[(b) Number of hyper-reduced nodes]{\includegraphics[width=.49\textwidth]{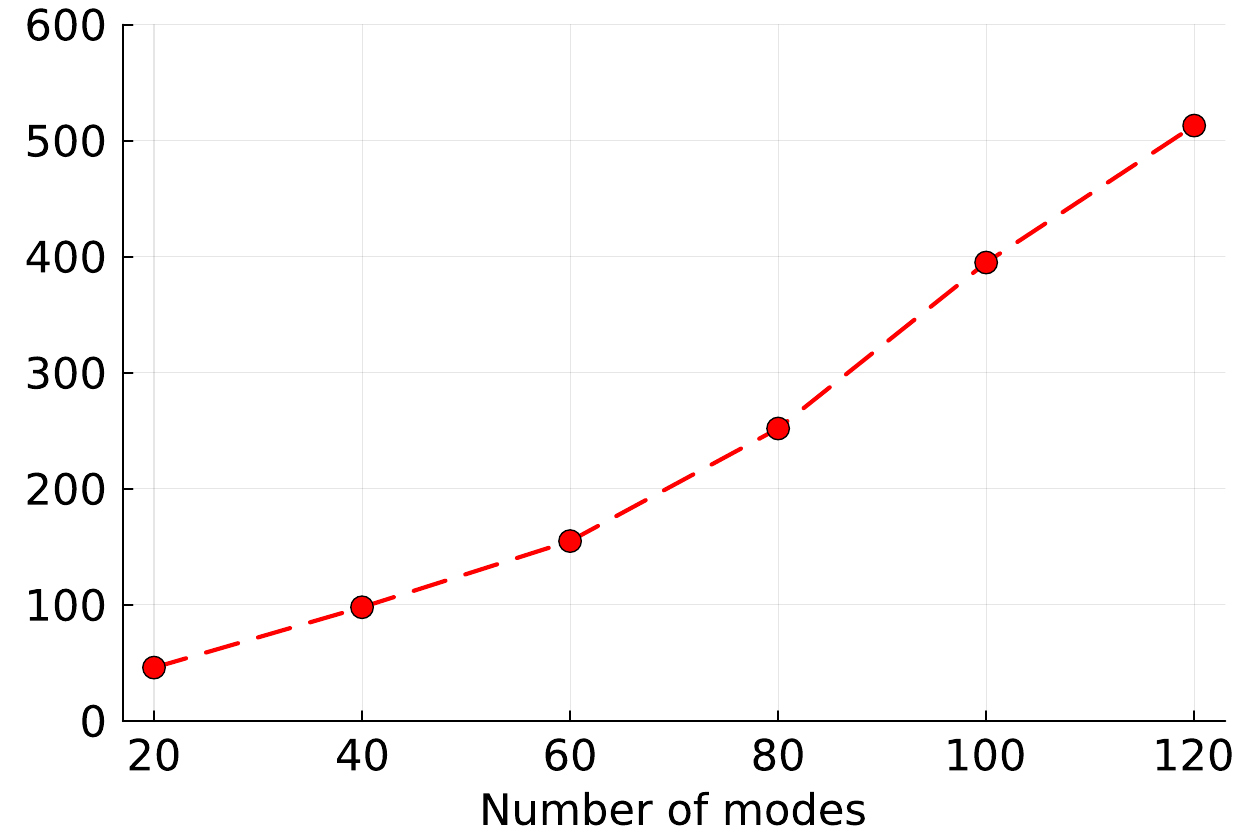}}
  \caption{ROM and HR-ROM performance (error and number of HR nodes).}
\label{fig2.2} 
\end{figure}

\begin{figure}
\centering
\noindent
\subfloat[(a) $N=20$]{\includegraphics[width=.49\textwidth]{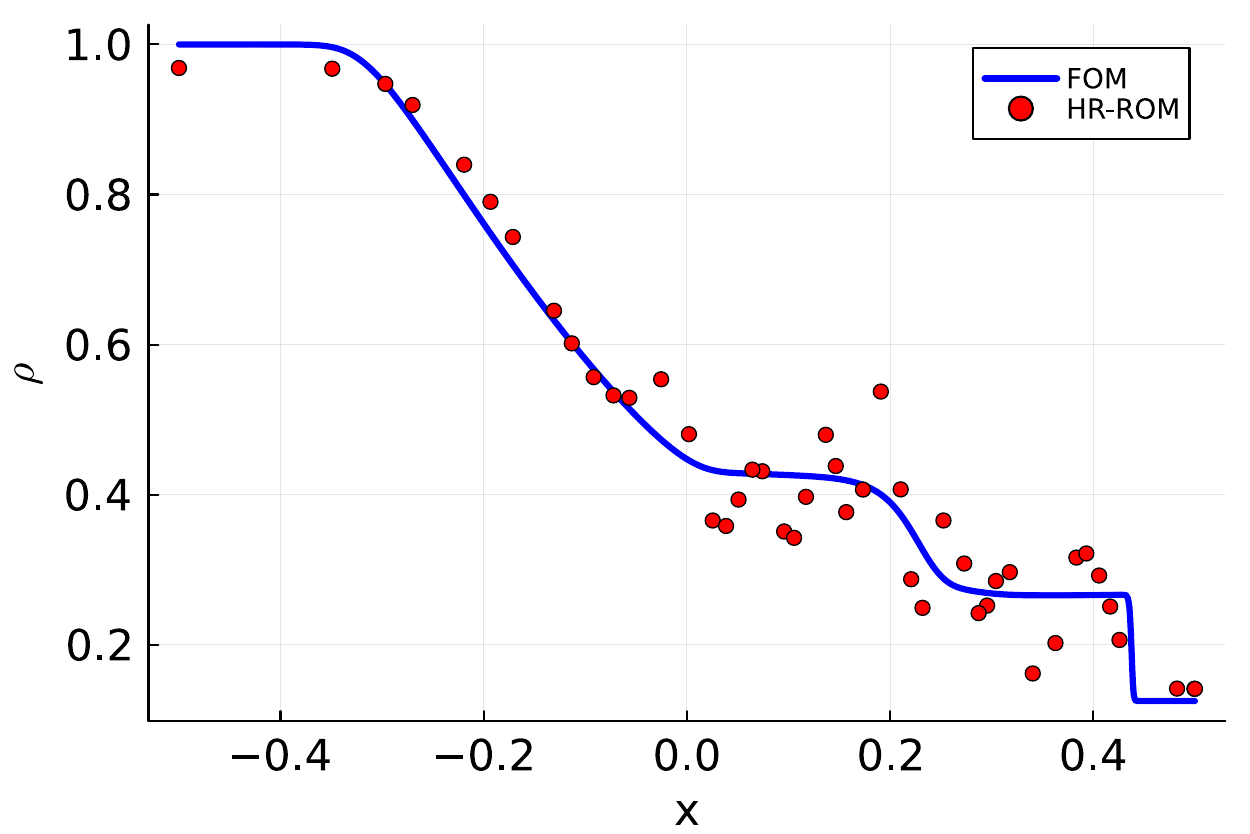}}
\subfloat[(b) $N=100$]{\includegraphics[width=.49\textwidth]{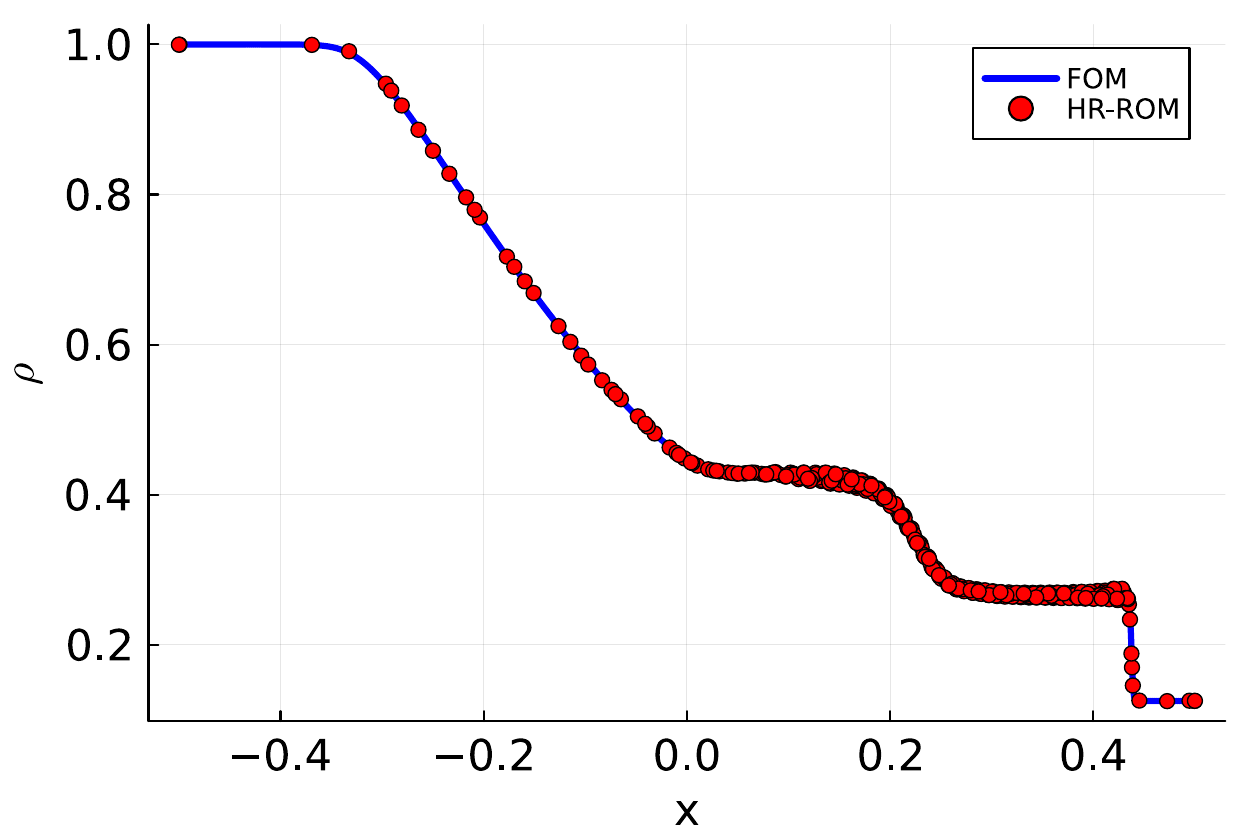}}
\caption{Density $\rho$ at $T=.25$. FOM solutions are in blue lines, and ROM solutions are in red nodes.}
\label{fig2.3}
\end{figure}

\begin{figure}
\centering
\noindent
\subfloat[(a) Convective entropy for $N=20$ HR-ROM]{\includegraphics[width=.49\textwidth]{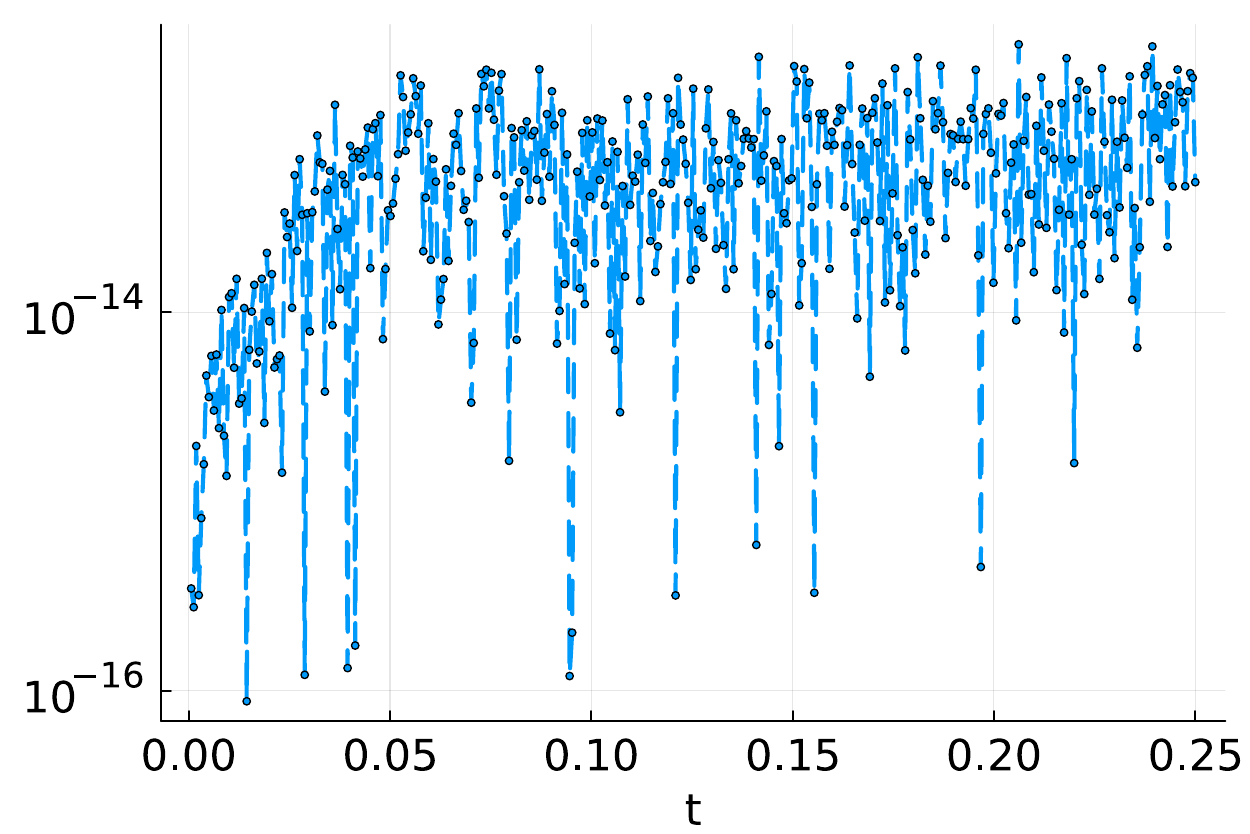}
\label{fig:DG2_conv_entro}}
\subfloat[(b) Entropy dissipation from viscosity for HR-ROMs]{\includegraphics[width=.49\textwidth]{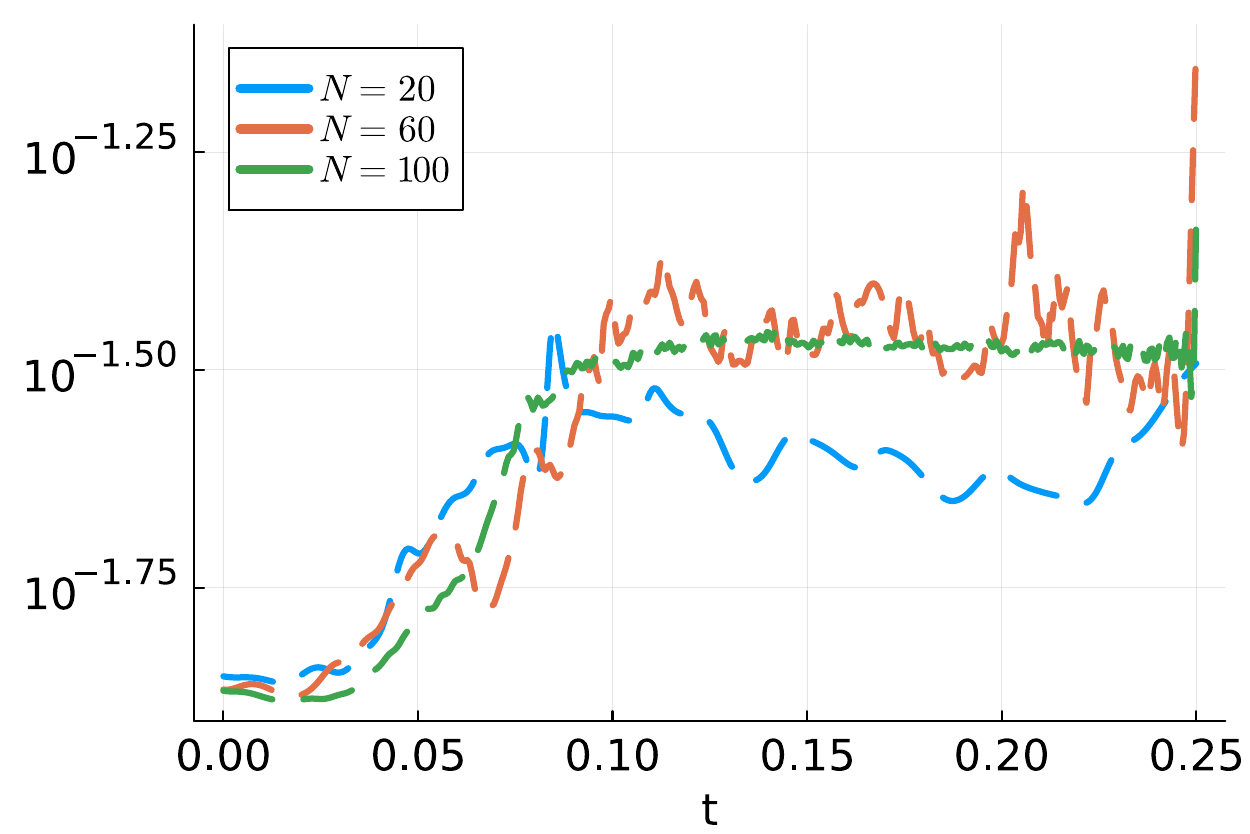}
\label{fig:DG2_entro_diss}}
\caption{Convective entropy and entropy dissipation for HR-ROMs.}
\end{figure}

\par
First, Fig.~\ref{fig:DG2_sval} displays singular values for models with and without entropy variable snapshots. We notice a marked slowdown in the decay of singular values with entropy enrichment. Subsequently, in Fig.~\ref{fig:DG2_proj_er}, we contrast the snapshot entropy projection error with 30 modes. The energy residuals $E_{30}$ for non-enriched and enriched snapshots are $9.58\times10^{-3}$ and $1.04\times 10^{-2}$, respectively. Entropy enrichment significantly reduces the projection error in every snapshot in this case, suggesting the necessity of enrichment even for 1D examples.

\par 
In Fig.~\ref{fig2.2}, we plot the relative error between FOM and ROM solutions and the number of hyper-reduced nodes for 20 to 120 modes as a function of the number of modes. In this example, the error in the hyper-reduced ROM solutions visibly deviates from the standard ROM solutions due to the more challenging problem setting. Nonetheless, the overall hyper-reduction process does not introduce excessively large errors in addition to ROM errors.

\par We then examine density plots in Fig.~\ref{fig2.3}. When $N=20$, the ROM exhibits oscillations around shocks, but the solution remains stable. Similar to previous example, when we increase the number of modes to $N=100$, the oscillations are significantly ameliorated. 

\par
Lastly, we present values of the convective entropy over simulation time using 20 modes in Fig.~\ref{fig:DG2_conv_entro} and compute the entropy dissipation from the viscosity with same modes in Fig.~\ref{fig:DG2_entro_diss}. We note that the convective entropy remains near zero, while entropy dissipation consistently stays positive, confirming the entropy conservation and viscous entropy dissipation.

\subsubsection{Courant–Friedrichs–Lewy condition}

The Courant–Friedrichs–Lewy (CFL) condition prescribes a maximum stable time-step when applying explicit time-stepping methods to systems of ODEs \cite{CFL}. It has been noted that projection-based linear ROMs can relax the CFL condition, allowing for an increased time step size in certain scenarios \cite{Marley15}. This modification is feasible because the ROM omits modes with small scale variations, which necessitate a small time-step to avoid instability. 

\par For all numerical experiments in this paper, we employ adaptive time stepping methods which automatically determine an appropriate explicit time-step size. We report the total number of time steps from previous 1D examples on \autoref{tb:CFL} to examine how ROMs influence the CFL condition. For consistency, we use a RK4 time-stepper for both the FOMs and ROMs. Although fewer time steps do not necessarily lead to faster runtimes, we report the minimum runtime for solving the ODEs over 20 trials in \autoref{tb:runtime}. In addition to the original examples, we include Example \ref{sec:1DEuler_weak} with a shorter simulation time of $T=0.2$, before the solution forms a shock.

\begin{table}
\begin{center}
\begin{tabular}{ | l | c | c | c | c | c | c |}
\hline
   Time steps& Example \ref{sec:1DEuler_weak} (1D Euler wall) & Example \ref{sec:sod} (Sod) \\ \hline
     FOM &  2381 &  1331  \\   \hline
   $N=20$ & 158  & 203  \\  \hline
  $N=40$ &  284 &  384\\   \hline
  $N=60$ & 386  &  682   \\   \hline
 $N=80$ &  471 &   505 \\   \hline
$N=100$ &  531 &   531 \\   \hline
\end{tabular}
\caption{Total number of time steps performed for FOMs and ROMs using adaptive RK4 time-stepping.}
\label{tb:CFL}
\end{center}
\end{table}

\begin{table}
\begin{center}
\begin{tabular}{ | l | c | c | c | c | c | c |}
\hline
   Runtime (ms)/ Error & Example \ref{sec:1DEuler_weak} ($T=.2$) & Example \ref{sec:1DEuler_weak} ($T=.75$) & Example \ref{sec:sod} \\ \hline
     FOM & 370 &  1410 & 769   \\   \hline
   $N=10$ & 3/ 4.42e-3 & 92/ 6.47e-2  &   7/ 1.43e-1\\  \hline
   $N=20$ & 26/ 3.59e-4 & 64/ 1.96e-2 &  43/ 8.80e-2\\  \hline
   $N=30$ & 94/ 2.62e-5 & 240/ 1.10e-2 &  182/ 6.41e-2\\  \hline
  $N=40$ & 382/ 5.82e-6 &  515/ 6.85e-3 & 391/ 3.03e-2\\   \hline
\end{tabular}
\caption{Online runtime (rounded to closest ms) and solution error using adaptive RK4 time-stepping.}
\label{tb:runtime}
\end{center}
\end{table}

We observe that the total number of time steps positively correlates with the number of modes chosen, yet it remains considerably smaller than that of the FOMs, suggesting that the CFL condition is relaxed for all ROMs considered in this work. We also observe consistently faster runtime with a reasonable number of modes for solutions both with and without shocks, despite the mesh size being much larger than $\epsilon$ (which corresponds to the under-resolved FOM regime).

We emphasize that the CFL condition relaxations observed in this work are based solely on numerical evidence. A theoretical analysis of CFL relaxation in ROMs has been performed for incompressible flows in  \cite{planariu2025}.

\subsection{2D experiments for DG ROMs}
\subsubsection{Kelvin-Helmholtz instability}
\label{sec:KH}
We begin our 2D experiments by examining the Kelvin-Helmholtz instability \cite{Miles59}, which is posed on a periodic square domain $\Go = [-1,1]^2$. We use a smoothed initial condition derived from \cite{Munz89}:
\begin{alignat*}{2}
        &\rho = 1+\f{1}{1+e^{-(y+1/2)/\sigma^2}}-\f{1}{1+e^{-(y-1/2)/\sigma^2}}\csp \\ & u_1 = \f{1}{1+e^{-(y+1/2)/\sigma^2}}-\f{1}{1+e^{-(y-1/2)/\sigma^2}}-\f{1}{2},\\
        &u_2 = \alpha\sin(2\pi x)(e^{-(y+1/2)^2/\sigma^2}-e^{-(y-1/2)^2/\sigma^2})\csp  && p = 2.5.
\end{alignat*}
In the following experiments, we set $\alpha=\sigma=0.1$ \cite{Chan20ROM}. For the DG FOM, we utilize $32\times32$ elements with a polynomial degree of 4. The artificial viscosity is set to $1\times10^{-3}$. We run until final time $T = 3.0$. 
 
\begin{figure}
\centering
\noindent
\subfloat[(a) Singular values of snapshots]{\includegraphics[width=.49\textwidth]{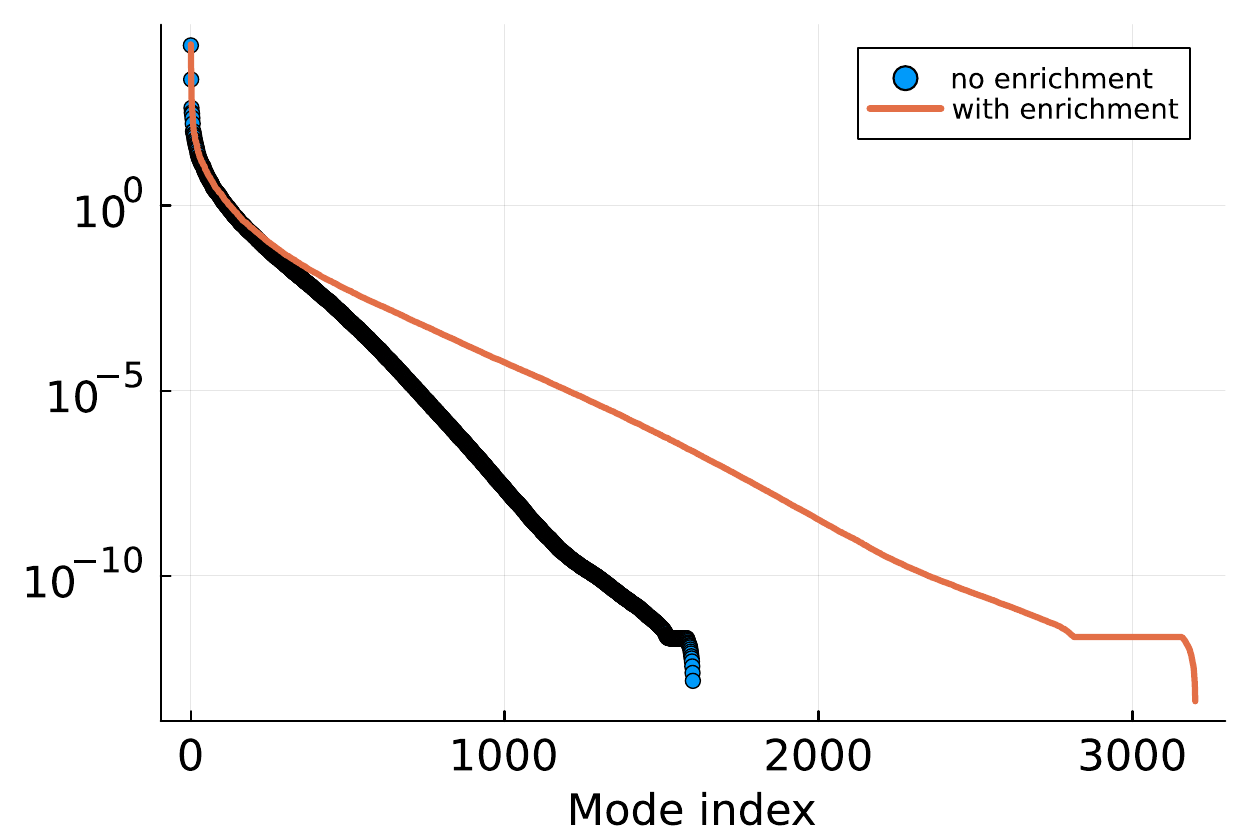}
\label{fig:DG3_sval}}
\subfloat[(b) Entropy projection error ($N=30$)]{\includegraphics[width=.49\textwidth]{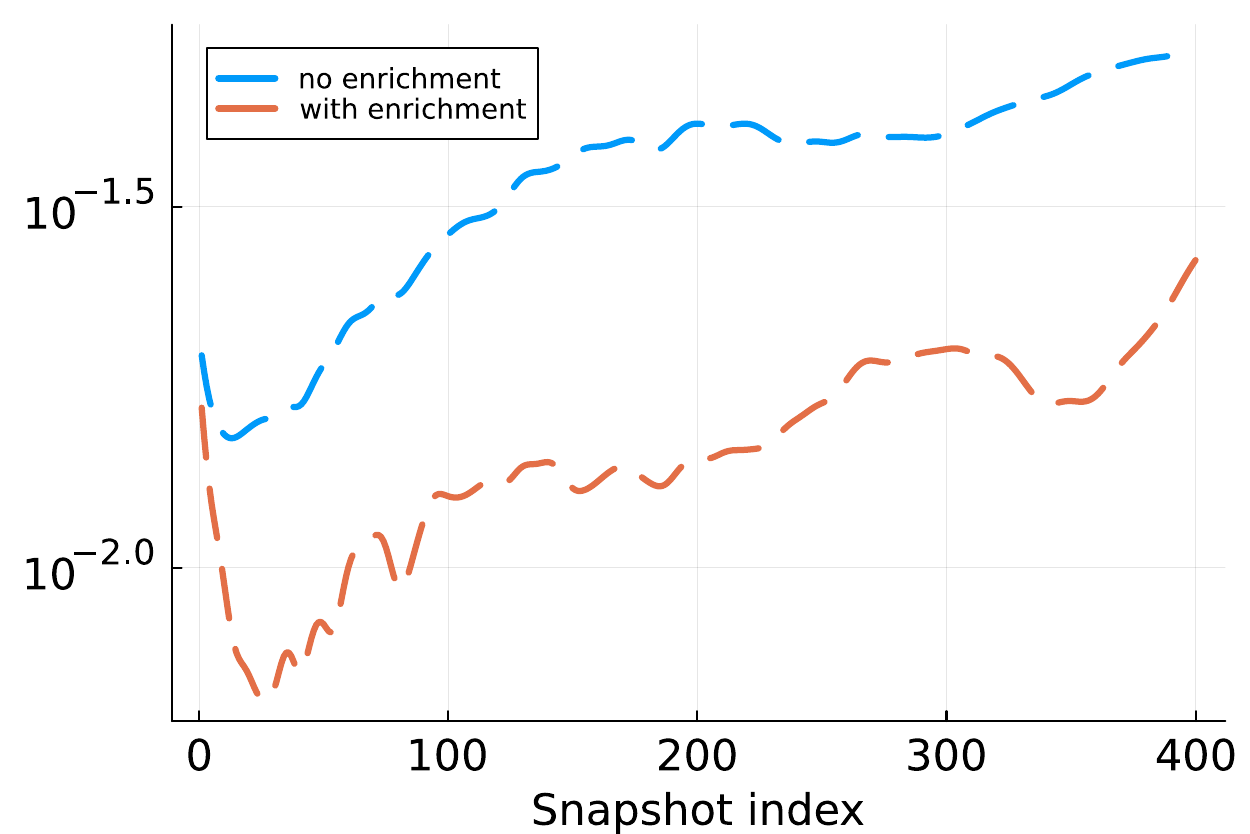}
\label{fig:DG3_proj_er}}

\caption{Singular values and entropy projection error with and without entropy enrichment.
}
\end{figure}

\begin{figure}
\centering
\noindent
\subfloat[(a) HR-ROM error]{\includegraphics[width=.49\textwidth]{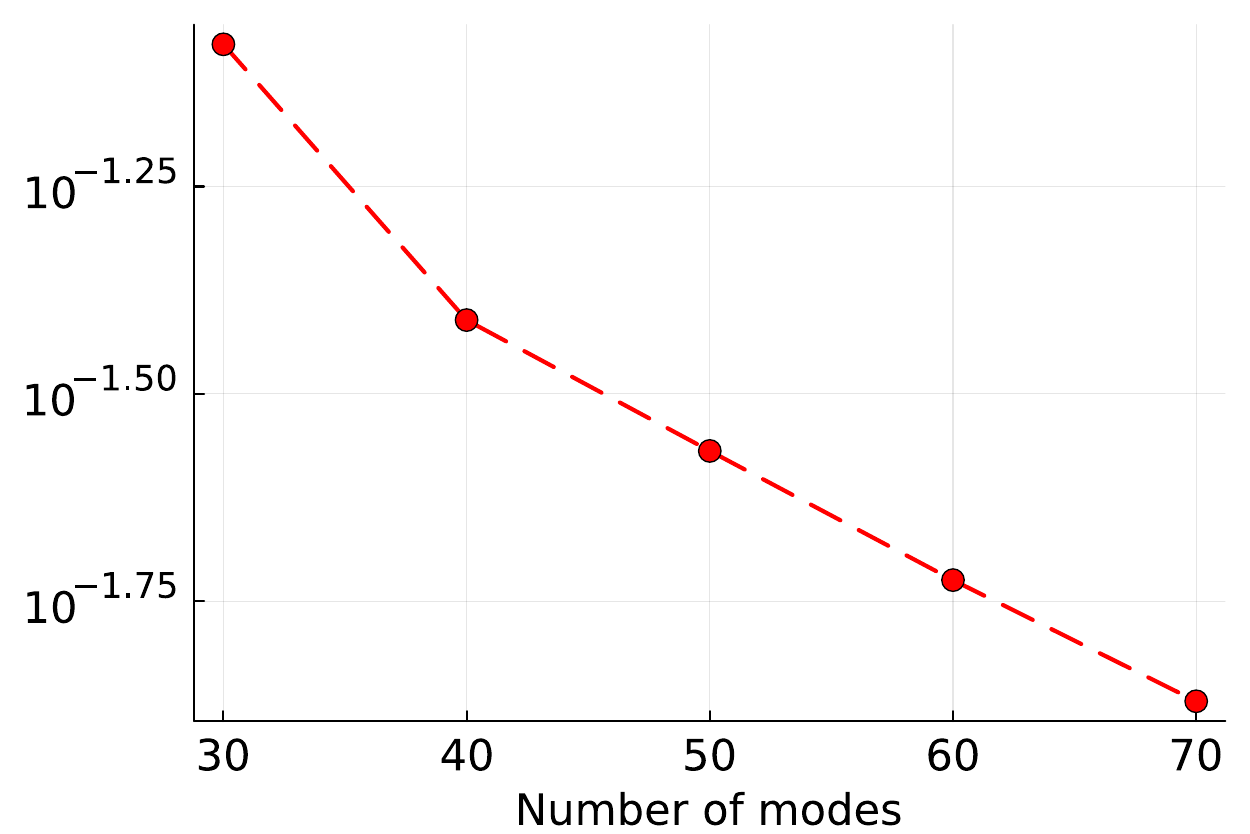}}
\subfloat[(b) Number of hyper-reduced nodes]{\includegraphics[width=.49\textwidth]{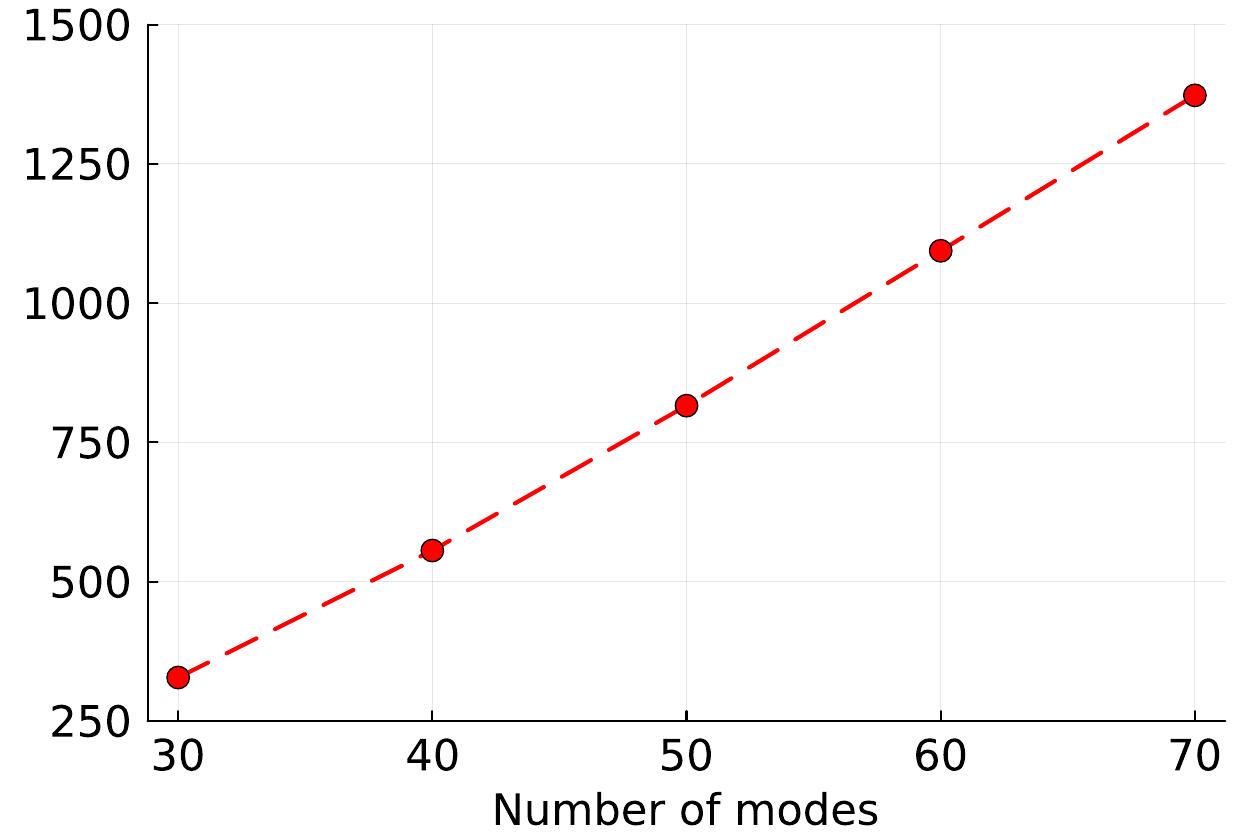}}
\caption{HR-ROM performance (error and number of HR nodes).}
\label{fig3.2} 
\end{figure}

\begin{figure}
\centering
\noindent
\subfloat[(a) FOM]{\includegraphics[width=.33\textwidth]{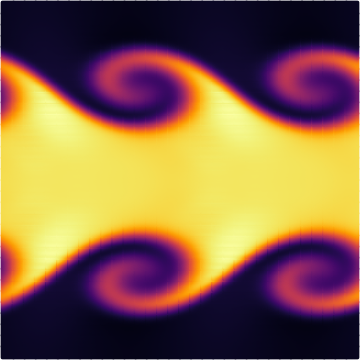}}
\subfloat[(b) $N=30$]{\includegraphics[width=.33\textwidth]{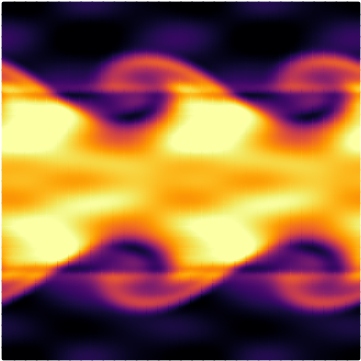}}
\subfloat[(c) $N=70$]{\includegraphics[width=.33\textwidth]{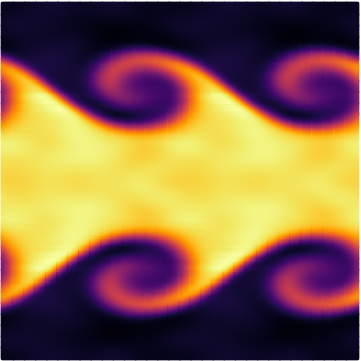}}
\caption{Density $\rho$ plots of Kelvin-Helmholtz instability.
}
\label{fig3.3}
\end{figure}

\begin{figure}
\centering
\noindent
\subfloat[(a) Convective entropy for $N=20$ HR-ROM]{\includegraphics[width=.49\textwidth]{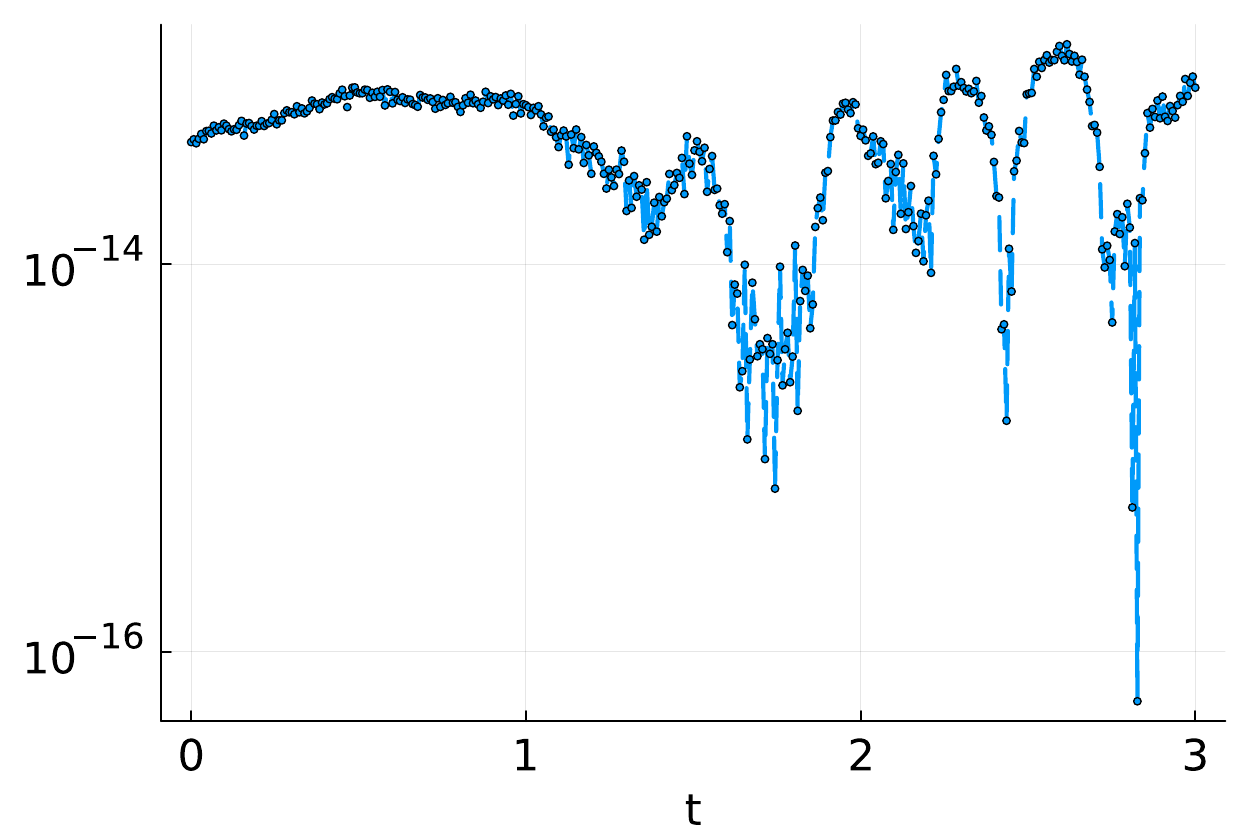}
\label{fig:DG3_conv_entro}}
\subfloat[(b) Entropy dissipation from viscosity for HR-ROMs]{\includegraphics[width=.49\textwidth]{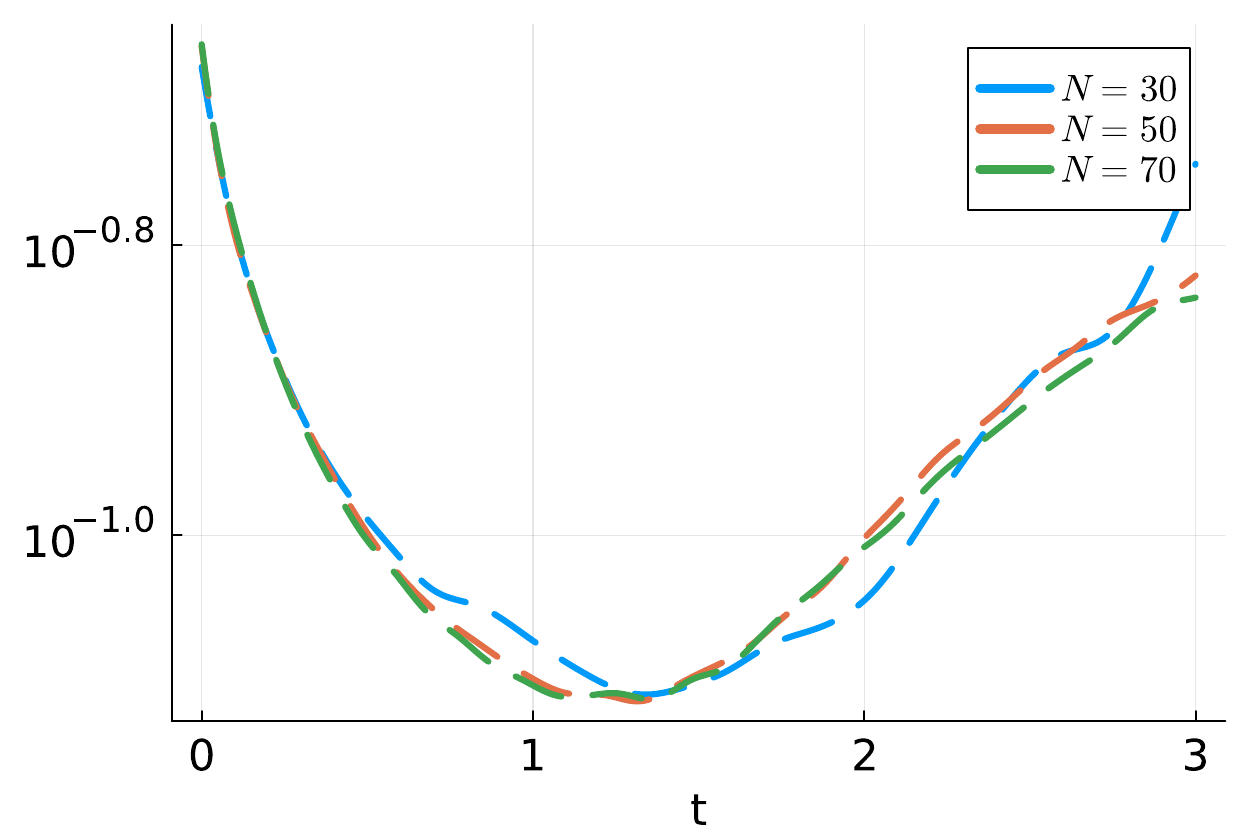}
\label{fig:DG3_entro_diss}}
\caption{Convective entropy and entropy dissipation for HR-ROMs.}
\end{figure}

\par As in the 1D case, we begin our analysis by evaluating the effect of entropy variable enrichment. In Fig.~\ref{fig:DG3_sval}, singular values of snapshot matrices are plotted, comparing cases with and without entropy variable enrichment. In Fig.~\ref{fig:DG3_proj_er}, we proceed to assess the relative approximation error under the entropy projection using 30 modes with and without entropy enrichment. While entropy enrichment results in a slower decay of singular values, the energy residuals $E_{30}$ are close ($2.69\times 10^{-3}$ without enrichment and $3.00\times 10^{-3}$ with enrichment). The results indicate that the first 30 modes with enrichment consistently result in smaller approximation errors when approximating each solution snapshot by the entropy projection.

\par 
Following this, Fig.~\ref{fig3.2} illustrates the relative $L^2$ error and the number of nodes used after hyper-reduction as the number of modes increases. 

\par Next, we examine the ROM solution at the final simulation time $T=3.0$.  In Fig.~\ref{fig3.3}, the solution displays recognizable patterns at $N=30$, with resolution around Kelvin-Helmholtz instabilities improving as $N$ increases to 70. We note that the FOM uses only artificial viscosity for stabilization, which tends to oversmooth the solution. Incorporating interface dissipation in both the FOM and hyper-reduced ROM would allow for smaller artificial viscosities and is part of our ongoing work. Again, no instability is observed in any of these ROM solutions.

\par Like previous examples, Fig.~\ref{fig:DG3_conv_entro} presents the computation of convective entropy for 20 modes. For periodic boundary conditions on high-dimensional domains, the convective entropy is 
\begin{equation*}
\LRb{\bm{v}_N^T\sumid\LRp{2\widebar{\bm{V}}_N^T\LRp{\widebar{\bm{Q}}^i\circ\bm{F}^i}\bm{1}}}.
\end{equation*}
The magnitude of the convective entropy in this example stays close to machine precision. Additionally, in Fig.~\ref{fig:DG3_entro_diss}, entropy dissipation due to viscous terms remains positive across varying mode counts.

Lastly, we compare the runtime and error in \autoref{tb:runtime_2D} using adaptive RK4. Again, we observe a significant runtime improvement for small numbers of modes. As the number of modes increases, the ROM error decreases, but the runtime correspondingly grows. We also note that the 2D FOM code is implemented using element-wise operations, which yields significantly faster runtimes compared to a global implementation.

\begin{table}
\begin{center}
\begin{tabular}{ | l | c | c | c | c | c | c |}
\hline
   Runtime (s)/ Error& Example \ref{sec:KH} ($T=3.0$) \\ \hline
     FOM & 36.44  \\   \hline
   $N=30$ & 2.24 / 5.92e-2 \\  \hline
   $N=50$ & 19.81/ 2.69e-2 \\  \hline
\end{tabular}
\caption{Online runtime (in seconds) and solution error using adaptive RK4 time-stepping.}
\label{tb:runtime_2D}
\end{center}
\end{table}

\subsubsection{Gaussian in 2D reflective wall boundary conditions}
We present a final experiment to illustrate how the hyper-reduction of boundary terms works with weakly-imposed boundary condition. We use a domain $\Go = [-1,1]^2$ and impose reflective wall boundary conditions, such that $\bm{f}^* = \bm{f}_{EC}(\bm{u}^+,\bm{u})$ at the boundary nodes. Here, $\bm{u}^+$ is defined as the ``mirror state'' 
$$\rho^+=\rho\csp u_1^+=-u_1\csp u_2^+=-u_2\csp p^+ = p.$$
We put a Gaussian pulse near left bottom corner of the domain with the following initial conditions
$$\rho = 1+0.5e^{-25((x+0.5)^2+(y+0.5)^2)}\csp u_1=u_2=0 \csp p=\rho^\gamma.$$

\begin{figure}
\centering
\noindent
\subfloat[(a) FOM]{\includegraphics[width=.33\textwidth]{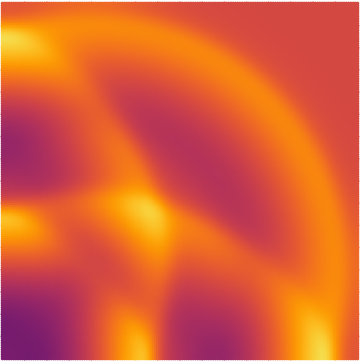}}
\subfloat[(b) HR-ROM ($N=30$)]{\includegraphics[width=.33\textwidth]{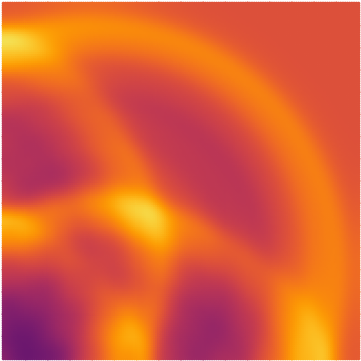}}
\subfloat[(c) HR nodes]{\includegraphics[width=.33\textwidth]{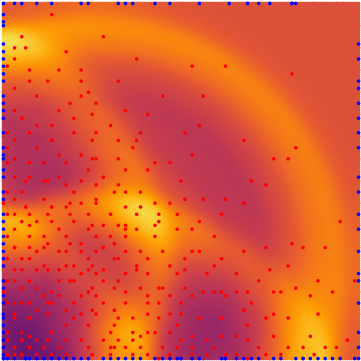}}
\caption{Density $\rho$ plots of 2D reflective wall simulation. Red and blue dots represent hyper-reduced volume and boundary nodes respectively.
}
\label{fig4.1}
\end{figure}

\begin{figure}
\centering
\noindent
\subfloat[(a) Convective entropy for $N=30$ HR-ROM]{\includegraphics[width=.49\textwidth]{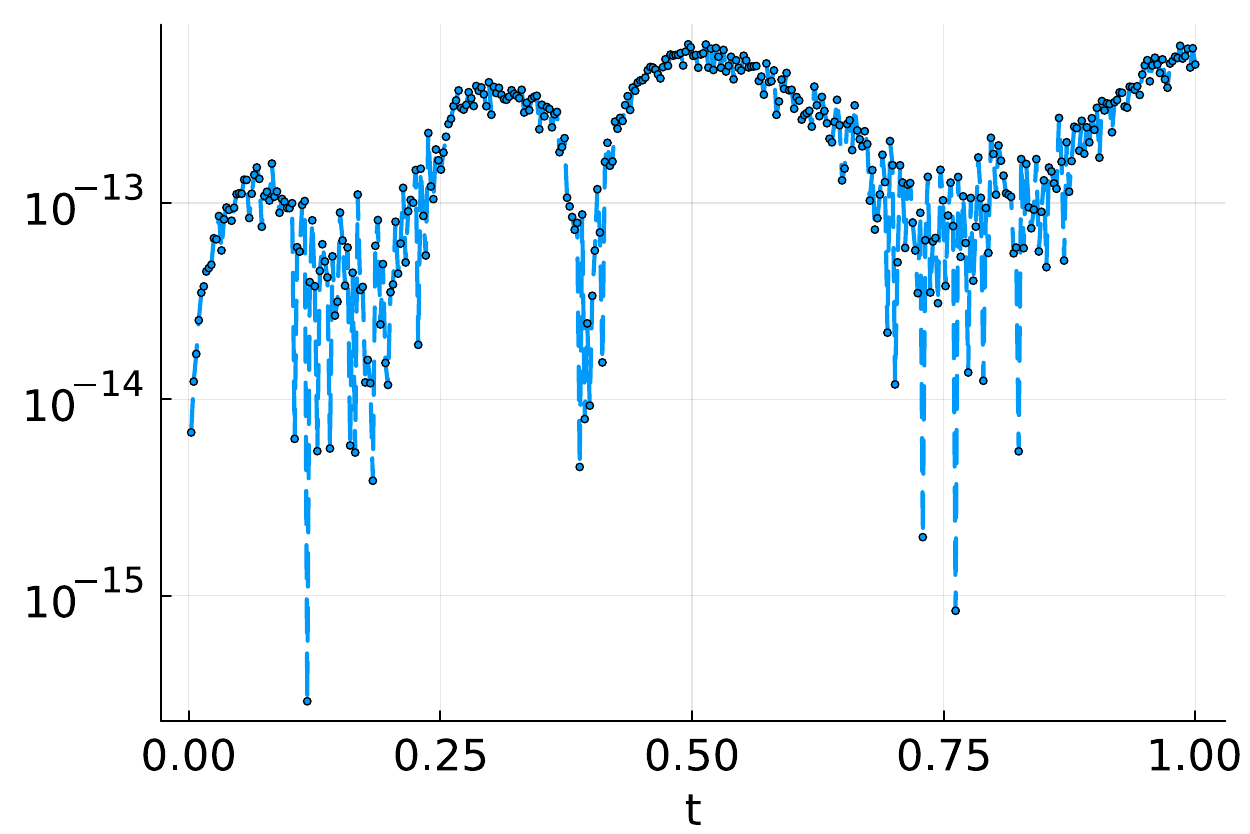}
\label{fig:DG4_conv_entro}}
\subfloat[(b) Entropy dissipation from viscosity for $N=30$ HR-ROM]{\includegraphics[width=.49\textwidth]{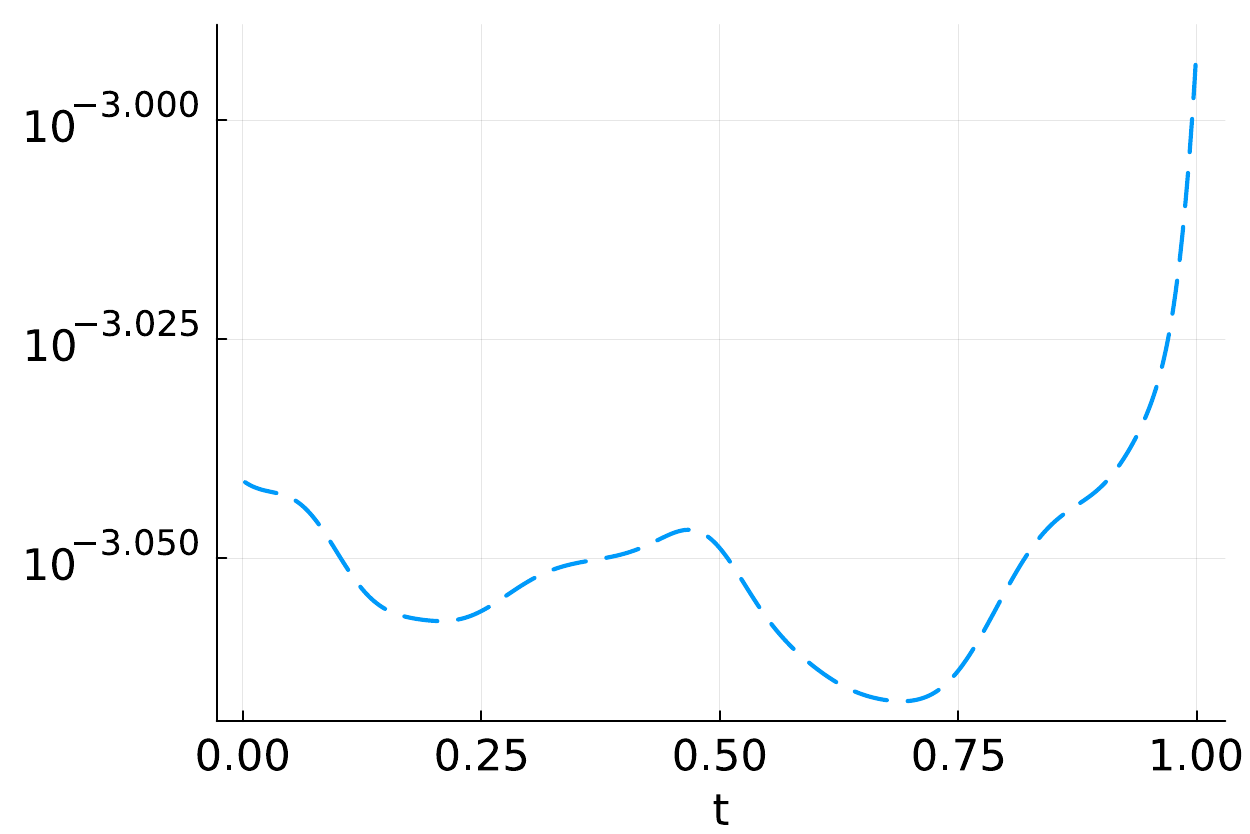}
\label{fig:DG4_entro_diss}}
\caption{Convective entropy and entropy dissipation for $N=30$ HR-ROM.}
\end{figure}
For the FOM, we utilize $16 \times 16$ elements with an interpolation degree of 4. The artificial viscosity is set to $\epsilon = 1\times10^{-3}$. We simulate the model until final time $T = 1.0$. 

\par In Fig.~\ref{fig4.1}, we plot FOM and the hyper-reduced volume and boundary nodes for 30 modes. In Fig.~\ref{fig:DG4_conv_entro}, we plot the convective entropy contribution, which is 
\begin{equation*}
   \LRb{\bm{v}_N^T\LRp{\sumid\LRp{\widebar{\bm{V}}_h^T\LRp{\LRp{\widebar{\bm{Q}}_h^i-\widebar{\bm{Q}}_h^{i,T}}\circ\bm{F}^i}\bm{1} + \widebar{\bm{V}}_{b}^T\widebar{\bm{B}}_b^i\bm{f}_b^{i,*}}}}
\end{equation*} with the variables defined in \eqref{eq:DG_ROM_HD_HR_weakBC1} for high-dimensional domains with weakly-imposed boundary conditions. We also plot the entropy dissipation from viscosity in Fig.~\ref{fig:DG4_entro_diss}. We observe that our simulation indeed satisfies the semi-discrete entropy conservation for convective entropy and entropy dissipation with viscosity. The ROM solution remains stable and approximates FOM solution well, even with a small number of modes and hyper-reduced nodes for both volume and boundary flux evaluations.

\section{Conclusion and future work}
\label{sec:conclusion}
\subsection{Conclusion}

\par In this work, we have successfully generalized the construction of entropy stable reduced order models for nonlinear conservation laws from finite volume-based FOMs to high order DG-based FOMs. We also introduce new hyper-reduction techniques: a weighted test basis for two-step hyper-reduction and Carath\'eodory pruning for the hyper-reduction of boundary terms. 

\par In extending from FVM to DG, we introduced a generalized construction of the ``test basis'' which accounted for mass matrices which are not scalar multiples of the identity. We prove that, for ``ideal'' hyper-reduction, the approximation error for the hyper-reduced differentiation matrix is orthogonal to the ROM approximation space under this modification of the test basis. This significantly improves the accuracy of the hyper-reduced ROM for higher order DG FOMs. This generalization also addresses the case of variable grid spacings for FVM FOMs. 

\par To preserve entropy stability during hyper-reduction of boundary terms, we then implement Carath\'eodory pruning based solely on numerical linear algebra. Unlike previous methods that rely on formulating a linear programming problem, this approach eliminates the need for sparsity-promoting LP solvers and provides a predictable number of hyper-reduced boundary nodes.

\par To validate the effectiveness of the proposed methods, we compared FVM and DG schemes on both 1D and 2D equations to observe their differences. We also tested both periodic and weakly-imposed boundary conditions. In all these tests, the convective entropy contribution consistently hovered around machine precision, highlighting the robustness of new entropy stable DG ROMs. We also observe numerically that our discretization of the viscous terms dissipates entropy.

\subsection{Future work}
Compared to standard ROMs, the main distinguishing features of our entropy-stable ROM are the use of flux differencing in place of conventional matrix-vector multiplication, and the application of an entropy projection to enforce entropy stability. However, a rigorous error analysis for entropy-stable ROMs remains an open problem and is left for future work.
\par 
In this work, we solely rely on a simple artificial viscosity to stabilize shocks in ROMs. However, interface flux dissipation in FVMs or DG methods can be another useful tool for reducing spurious oscillations. Yu and Hesthaven have shown the effectiveness of upwinding dissipation using discrete empirical interpolation methods (DEIM) to enhance model accuracy \cite{Yu22}; however, the discrete entropy stability of this approach remains to be proven. 

\subsubsection{Addressing ROM accuracy for transport-dominated problems}

While this work focuses on guaranteeing stability for a projection-based ROM, it does not address accuracy issues. A well-known challenge for projection-based ROMs is the accuracy of the POD approximation space for transport-dominated problems, which is related to the slow decay of the Kolmogorov $n$-width \cite{van2024modeling}.

A popular approach to address this issue is the use of nonlinear manifold approximation spaces. While this work is restricted to linear manifolds, we highlight the work of Klein et al.\, who extended a similar flux differencing entropy stable ROM formulation to nonlinear approximation spaces \cite{KLEIN2025}. In particular, Klein and Sanderse constructed entropy stable ROMs by exploiting projections onto the tangent space, and demonstrated that nonlinear approximation spaces improve accuracy for solutions with shocks or turbulent structures compared with POD-based linear reduced spaces. While Klein and Sanderse consider a finite volume discretization, their approach would be straightforward to combine with the nodal DG formulations in this work.

To the best of our knowledge, the nonlinear reduced spaces introduced in \cite{KLEIN2025} have not yet been combined with hyper-reduction techniques in an entropy stable fashion, and would constitute a promising future direction.

 Another promising avenue for future research lies in the integration of domain decomposition (DD) techniques into the ROM framework \cite{Xiao19, Iyengar22, Hoang21, Diaz24}, where the computational domain is partitioned into smaller subdomains. In addition to mollifying the slow decay of the Kolmogorov $n$-width, we expect DD approaches to increasing the sparsity of hyper-reduced differentiation matrices and reduce the total number of two-point entropy conservative flux evaluations. These two-point flux evaluations dominate the computational cost of entropy stable ROMs \cite{Chan20ROM}.

 

\section*{Acknowledgements}
Ray Qu gratefully acknowledges support from National Science Foundation under awards DMS-1943186. Jesse Chan gratefully acknowledges support from National Science Foundation under awards DMS-1943186 and DMS-223148. Akil Narayan gratefully acknowledges support from AFOSR FA9550-23-1-0749.

\appendix 
\section{Burgers' equation with small viscosity}

\begin{figure}[h]
    \centering
    \subfloat[(a) $N=30$ solutions w/o interface dissipation]{\includegraphics[width=.49\textwidth]{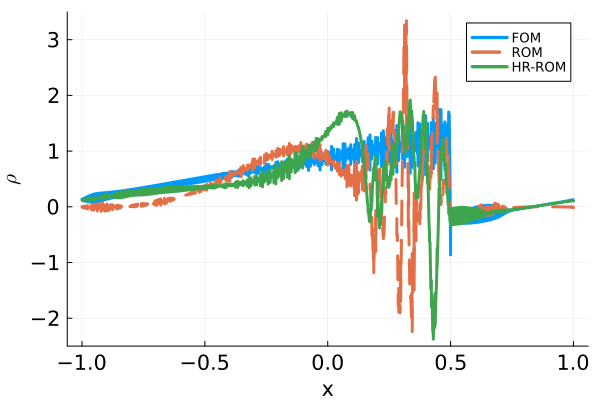}}
\subfloat[(b) $N=30$ solutions w/ interface dissipation]{\includegraphics[width=.49\textwidth]{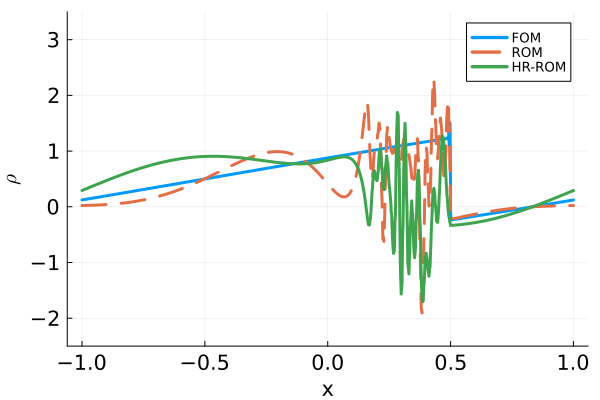}}
\caption{Example \ref{sec:burgers} with artificial viscosity coefficient $\epsilon=1\times10^{-4}$}
\end{figure}

\bibliographystyle{elsarticle-num}
\bibliography{reference.bib}

\end{document}